\documentclass[11pt]{amsart}
\usepackage{tikz-cd}
\usepackage{enumitem}
\setlist[description]{leftmargin=\parindent,labelindent=\parindent}
\usetikzlibrary{matrix,arrows,decorations.pathmorphing}
\usepackage{amssymb}
\usepackage{amsmath,amscd}
\usepackage{mathrsfs}
\usepackage{url}
\usepackage{cite}
\usepackage{fullpage}
\usepackage{hyperref}
\usepackage{verbatim}
\usepackage{comment}
\usepackage{fancyvrb}
\usepackage{fvextra}
\usepackage{caption}
\usepackage{ytableau}
\usepackage{stmaryrd}
\usepackage{appendix}

\allowdisplaybreaks

\newtheorem{thm}{Theorem}
\newtheorem{prop}[thm]{Proposition}
\newtheorem{lem}[thm]{Lemma}
\newtheorem{cor}[thm]{Corollary}
\newtheorem{conj}[thm]{Conjecture}

\theoremstyle{definition}
\newtheorem{definition}[thm]{Definition}
\newtheorem{example}[thm]{Example}
\newtheorem{rem}[thm]{Remark}

\newtheorem{speculation}[thm]{Speculation}

\newcommand{\X}{\mathcal{X}}

\newcommand{\N}{\mathcal{N}}

\newcommand{\cc}{\mathbb{C}}

\newcommand{\qq}{\mathbb{Q}}
\newcommand{\R}{\mathsf{R}}

\newcommand{\C}{\mathcal{C}}
\newcommand{\M}{\mathcal{M}}

\newcommand{\A}{\mathcal{A}}
\newcommand{\F}{\mathcal{F}}

\newcommand{\CH}{\mathsf{CH}}
\newcommand{\IH}{\mathsf{IH}}
\renewcommand{\H}{\mathsf{H}}
\newcommand {\W}{\mathsf{W}}

\newcommand{\E}{\mathcal{E}}
\renewcommand{\O}{\mathcal{O}}

\newcommand{\Xb}{\mathcal{X}}
\newcommand {\Yb}{\mathcal {Y}}
\DeclareMathOperator{\ct}{ct}

\DeclareMathOperator{\ev}{ev}
\renewcommand{\tilde}{\widetilde}

\DeclareMathOperator{\Sp}{Sp}

\DeclareMathOperator{\Aut}{Aut}

\DeclareMathOperator{\Ext}{Ext}

\DeclareMathOperator{\im}{Im}
\newcommand{\Tor}{\mathsf{Tor}}

\DeclareMathOperator{\Sym}{Sym}

\usepackage{wrapfig}

\makeatletter
\def\@defaultbiblabelstyle#1{[#1]}
\makeatother

\DeclareMathAlphabet{\pazocal}{OMS}{zplm}{m}{n}

\title[Tautological and
non-tautological cycles on the moduli space of abelian varieties]{Tautological and non-tautological cycles\\ on the moduli space of abelian varieties}

\author{Samir Canning}
\address{Department of Mathematics, ETH Z\"urich}
\email {samir.canning@math.ethz.ch}

\author{Dragos Oprea}
\address{Department of Mathematics, University of California, San Diego}
\email {doprea@math.ucsd.edu}

\author{Rahul Pandharipande}
\address{Department of Mathematics, ETH Z\"urich}
\email {rahul@math.ethz.ch}

\date{}
\begin{document}
\baselineskip=17pt
\footskip=1.5\normalbaselineskip

\begin{abstract} 
The tautological Chow ring of the moduli space 
$\mathcal{A}_g$
of
principally polarized abelian varieties of dimension $g$ was defined and calculated 
by van der Geer in 1999. 
By studying the Torelli pullback of algebraic
cycles classes from $\mathcal{A}_g$ to
the moduli space 
$\mathcal{M}_g^{\ct}$
of genus $g$
of curves 
of compact type, 
 we prove that the product class 
$[\mathcal{A}_1\times \mathcal{A}_5]\in \mathsf{CH}^{5}(\mathcal{A}_6)$
is non-tautological, the first
construction of an interesting non-tautological algebraic class on 
the moduli spaces of abelian varieties. 
For our proof, we use the complete description of
the the tautological ring $\mathsf{R}^*(\mathcal{M}_6^{\ct})$
in genus 6
conjectured by Pixton and recently proven by Canning-Larson-Schmitt. The tautological ring $\mathsf{R}^*(\mathcal{M}_6^{\ct})$
has a 1-dimensional Gorenstein kernel, which is geometrically
explained by the Torelli pullback of $[\mathcal{A}_1\times \mathcal{A}_5]$. More generally, the Torelli pullback of the difference between $[\mathcal{A}_1\times \mathcal{A}_{g-1}]$ and its tautological projection always lies in the Gorenstein kernel of 
$\mathsf{R}^*(\mathcal{M}_g^{\ct})$.

The product map $\mathcal{A}_1\times \mathcal{A}_{g-1}\rightarrow \mathcal{A}_g$ is a Noether-Lefschetz locus with general Neron-Severi rank 2. A natural extension
of van der Geer's tautological ring is
obtained by including more general Noether-Lefschetz
loci.
Results and conjectures related to cycle classes of 
Noether-Lefschetz loci
for all $g$ are presented. 
\end{abstract}

\date{August 2025}
\maketitle
\setcounter{tocdepth}{1}
\tableofcontents{}
\section{Introduction}

\subsection{Moduli of abelian varieties} Let $g\geq 1$ be an integer, and let $\mathfrak H_g$ denote the Siegel upper half-space $$\mathfrak H_g=\{\Omega \in \operatorname {Mat} _{g \times g}(\mathbb {C} ):\,\,\,\Omega ^{T}=\Omega ,\,\,\,\im (\Omega )>0\} \, .$$ To each $\Omega\in \mathfrak H_g$, we associate the abelian variety $$X_{\Omega}=\mathbb C^{g}/(\Omega \mathbb Z^g+\mathbb Z^g)\, ,$$ which is naturally principally polarized by the matrix $\text{Im\,} (\Omega)^{-1}$. There is an action of the symplectic group $\Sp_{2g}(\mathbb Z)$ on $\mathfrak H_g$ given by 
$$\begin{pmatrix} A & B \\ C & D\end{pmatrix} \Omega= (A\Omega + B)(C\Omega+D)^{-1} \,.$$
Two principally polarized abelian varieties 
$X_{\Omega}$ and $X_{\Omega'}$
are isomorphic if and only if
$\Omega$ and $\Omega'$ are
in the same $\Sp_{2g}(\mathbb Z)$-orbit:
$$X_{\Omega}\simeq X_{\Omega'}\ \iff\  
\exists M\in \Sp_{2g}(\mathbb Z)  \ \, {\text {such that}} \
\, \Omega'=M\Omega\, .$$
The quotient space \begin{equation}\label{present}\mathcal A_g=\left[\Sp_{2g}(\mathbb Z)\backslash \mathfrak H_g\right]\end{equation} is the moduli  of principally polarized abelian varieties. The action of $\Sp_{2g}(\mathbb Z)$ on $\mathfrak H_g$ has finite stabilizers. The space $\mathcal A_g$ is a nonsingular Deligne-Mumford stack of dimension 
$\binom{g+1}{2}$.
We refer the reader to \cite{BL} for the 
foundations of the study of the moduli of
abelian varieties.

Since $\mathfrak{H}_g$
is contractible,
the rational cohomology{\footnote{All cohomology and Chow theories in the paper will be taken with $\qq$-coefficients.}} of $\mathcal A_g$ can be identified with the rational cohomology of the group $\Sp_{2g}(\mathbb Z)$,
$$\mathsf{H}^*(\mathcal{A}_g)=
\mathsf{H}^*_{\Sp_{2g}(\mathbb Z)}(\bullet)\, ,$$
via  the presentation \eqref{present}.  
 By a fundamental result of Borel \cite{B}, the 
stable cohomology of $\Sp_{2g}(\mathbb Z)$ as $g$ increases is  the free polynomial algebra 
\begin{equation} \label{borel}
\lim_{g\rightarrow \infty}
\mathsf{H}^*_{\Sp_{2g}(\mathbb Z)}(\bullet)
= \mathbb{Q}[\lambda_1,
\lambda_3, \lambda_5, \ldots]
\end{equation}
in variables $\lambda_{k}$ of
degree $2k$, where $k$ is an odd positive integer. 

Let $\pi:\X_g\rightarrow \A_g$ denote the universal principally polarized abelian variety. The Hodge bundle is the rank $g$ vector bundle
\[
\mathbb{E}_g=\pi_{*} (\Omega_{\pi}) . 
\]
The $\lambda$ classes in Borel's
stability result \eqref{borel}
are the Chern classes of the
Hodge bundle,
$$\lambda_i= c_i(\mathbb{E}_g)\, .$$
Only the  odd Chern classes of $\mathbb{E}$ appear in
the stability result.

For fixed $g$, complete calculations of the cohomology of $\A_g$ have so far been restricted to low dimensions. Complete results are available for $g\leq3$, see \cite{H}. For $g=4$, partial results can be found in \cite {HT1}. Further studies of the cohomology of $\A_g$ (together with the
cohomology of various compactifications) can be found in \cite {GHT, BBC+, BCGP, toptop}.  Other related results are surveyed in \cite {HT2}.

\subsection{The tautological ring} 
For all $g\geq 1$, 
van der Geer \cite{vdg}
proved that the Chern classes of
the Hodge bundle
satisfy two basic relations
in $\mathsf{CH}^*(\mathcal{A}_g)$:
\begin{equation} 
\label{lambdavanish}
\lambda_g=0\,,
\end{equation}
\begin{equation}
\label{mumford} (1+\lambda_1+\lambda_2+\ldots+\lambda_g)(1-\lambda_1+\lambda_2-\ldots+(-1)^g\lambda_g)=1\,.
\end{equation}
Esnault and Viehweg \cite{EV} showed that relation \eqref{mumford} also extends to toroidal compactifications of $\A_g.$
As a consequence  of \eqref{mumford}, usually called Mumford's relation, the even degree $\lambda$ classes can be expressed in terms of the $\lambda$ classes of odd degree  (which explains the omission
of even $\lambda$ classes in 
Borel's result \eqref{borel}).

Motivated by stability,
van der Geer \cite{vdg} defined the tautological ring $$\R^*(\A_g)\subset \CH^*(\A_g)$$ to be the $\qq$-subalgebra generated by the odd $\lambda$ classes.
The definition of van der Geer is entirely
parallel to  
Mumford's definition \cite{DMumford}
of the tautological ring
$$\R^*(\mathcal{M}_g)\subset \CH^*(\mathcal{M}_g)$$
of the moduli space of curves
as the $\qq$-subalgebra generated by the $\kappa$
classes (the free 
generators of the stable cohomology of
the mapping class group \cite{MadWeiss}).
A central result of \cite{vdg}
is the
 complete determination of
  $\R^*(\A_g)$.

\begin{thm}[van der Geer]\label{vdgthm}\leavevmode  The following
properties hold:
\begin{enumerate}
    \item[\textnormal{(i)}] The kernel
    of 
    the quotient
     $$\mathbb{Q}[\lambda_1,
\lambda_2, \lambda_3,  \ldots, \lambda_g] \rightarrow
\mathsf{R}^*(\mathcal{A}_g)\rightarrow 0$$
is generated as an ideal by
the  relations \eqref{lambdavanish} and \eqref{mumford}. 
   \item [\textnormal{(ii)}] $\R^*(\A_g)$ is a Gorenstein local ring with socle in codimension $\binom{g}{2}$,
   $$\R^{\binom{g}{2}}(\A_g) \cong 
   \qq\, .$$
   The class $\lambda_1\lambda_2 \lambda_3 \cdots \lambda_{g-1}$ is
   a generator of the socle.
    \item [\textnormal{(iii)}] For $g\leq 3$, $\R^*(\A_g)=\CH^*(\A_g)$.
\end{enumerate}
\end{thm}
Statements (i) and (ii) are found in \cite {vdg}.  
The presentation (i) implies $$\mathsf {R}^*(\mathcal A_g)\cong \CH^*(\mathsf{LG}_{g-1})$$ where ${\mathsf{LG}_{g-1}}$ denotes the Lagrangian Grassmannian of $(g-1)$-dimensional Lagrangian subspaces of $\mathbb C^{2g-2}.$ Statement (ii) is consistent with this isomorphism since $\dim {\mathsf {LG}}_{g-1}=\binom{g}{2}.$  
Statement (iii) is established in \cite{vdg2}.

Many interesting cycle classes on $\A_g$ admit explicit expressions in the tautological ring, see \cite {vdgsurvey} for a survey.  

\subsection{Curves of compact type}
For $g \geq 2$,
let $\mathcal M_g^{\ct}$ denote the moduli space of curves of compact type.  The moduli space $\mathcal M_g^{\ct}$ also carries a Hodge bundle $$\mathbb E_g=\pi_{*}(\omega_{\pi})\, ,$$ where $\pi:\mathcal C_g\to \mathcal M_g^{\ct}$ is the universal curve. The Torelli map $$\Tor: \mathcal M_g^{\ct}\to \mathcal A_g\, , \ \ \ \Tor([C])=[\mathsf{Jac}(C),\Theta]$$ 
sends a curve $C$ to the Jacobian $\mathsf{Jac } (C)$ parameterizing line bundles over $C$ of degree $0$ on every irreducible component. The Jacobian
has a canonical principal polarization given by the theta divisor $\Theta$.
A simple check shows that
the Torelli map respects  the two Hodge bundles, $$\Tor^* \,\mathbb E_g = \mathbb E_g \, .$$ 

Let $\R^*(\overline{\mathcal M}_g)$ denote
the tautological ring of $\overline {\mathcal M}_g$. The tautological ring $\mathsf R^*(\mathcal M_g^{\ct})$ is defined by restriction, as the image $$\R^*(\overline \M_g)\subset \CH^*(\overline {\mathcal{M}}_g)\to \CH^*(\M_g^{\ct})\, . $$ 
A survey of definitions, results, and conjectures about the tautological rings
of the moduli spaces of cuves can be found in \cite{FP3,P1}.

We can also consider the smaller $\qq$-subalgebra generated by $\lambda$ classes $$\Lambda^*(\M_g^{\ct})\subset \R^*(\mathcal M_g^{\ct})\, .$$  Since the Torelli  map respects the Hodge bundles, the image of $$\Tor^*: \R^*(\A_g)\rightarrow \CH^*(\M_g^{\ct})$$ is contained in   $\Lambda^*(\M_g^{\ct})$. 

\subsection{The \texorpdfstring{$\lambda_g$}{Lg}-pairing} By \cite[Section 5.6] {GV} and \cite[Proposition 3]{FP2}, we have \begin{equation}\label{vanish}\R^{2g-3}(\M_g^{\ct})=\mathbb Q\, , \quad \R^{>2g-3}(\M_g^{\ct})=0\, .\end{equation} Furthermore, as noted in \cite{FP2}, there exists a canonical evaluation \begin{equation}\label{evaluation}\epsilon^{\ct}:\R^{2g-3}(\M_g^{\ct}) \to \mathbb Q\, , \ \quad \alpha \mapsto \int_{\overline {\mathcal M}_g} \overline{\alpha} \cdot \lambda_g\, .\end{equation} The integration requires a lift $\overline{\alpha}$ of $\alpha$ to the compactification. The answer is well-defined (independent of lift) since $\lambda_g$ vanishes on the complement $\overline \M_g\smallsetminus \M_g^{\ct}$. The evaluation $\epsilon^{\ct}$ induces a pairing between classes of complementary degrees, $$\R^k(\M_g^{\ct})\times \R^{2g-3-k}(\M_g^{\ct})\,\to\,  \R^{2g-3}(\M_g^{\ct})\, \cong\,  \mathbb Q\, ,\ \quad (\alpha, \beta)\mapsto \int_{\overline \M_g}\overline{\alpha}\cdot \overline{\beta}\cdot \lambda_g\, ,$$
which is called the {\em $\lambda_g$-pairing}.

The $\lambda_g$-pairing 
arises naturally
in the Gromov-Witten theory of curves \cite{GeP}. See
\cite{FP4, JP,P2} for explicit
formulas and structures related
to the $\lambda_g$-pairing.

\subsection{The product locus $\mathcal{A}_1\times \mathcal{A}_{g-1}$}

Via the product of principally polarized abelian
varieties, there is a proper morphism
$$\mathcal{A}_1 \times \mathcal{A}_{g-1}\rightarrow \mathcal{A}_g\, .$$
By the dimension formula, the image is of
codimension $g-1$ in $\mathcal{A}_g$.
For $g\geq 1$,
let $$[\mathcal{A}_1\times \mathcal{A}_{g-1}]\in \mathsf{CH}^{g-1}(\mathcal{A}_g)$$ be the pushforward
of the fundamental class. In the $g=1$  case, 
$$[\A_1\times \A_0]=[\A_1] \in \mathsf{CH}^0(\A_1)\,.$$

\begin{prop}\label{tautprop} For $g\geq 1$,
if $[\mathcal{A}_1\times \mathcal{A}_{g-1}] \in
\mathsf{CH}^{g-1}(\mathcal{A}_g)$ is
a tautological class, then{\footnote{$B_{2g}$ is the Bernoulli number.}}
$$[\mathcal{A}_1\times \mathcal{A}_{g-1}]=\frac{g}{6|B_{2g}|}\lambda_{g-1} \in 
\mathsf{CH}^{g-1}(\mathcal{A}_g)\, .$$
\end{prop}

To show Proposition \ref{tautprop}, we
use properties of the tautological ring 
$\mathsf{R}^*(\mathcal{A}_g)$ together
with a study of
the pullback to $\mathcal{M}_g^{\mathrm{ct}}$
via the Torelli map and the $\lambda_g$-pairing.\footnote {An alternative proof can be found in \cite {CMOP}, where the tautological projection of every cycle of the form $$[\A_{g_1}\times \ldots \A_{g_\ell}]\in \CH^*(\A_g)\,, \quad g_1+\ldots+g_\ell=g$$ is defined and explicitly calculated. The answer is given as a Schur determinant in the Hodge classes.} A version of Proposition \ref{tautprop} for $g\leq 5$ was proven earlier by Grushevsky and Hulek, see \cite[Lemma 8.1, Proposition 9.3]{GH}. The formula of Proposition \ref{tautprop}
was also found independently by Faber in unpublished work.

\subsection{Main results} \label{mainr}
But is $[\mathcal{A}_1\times \mathcal{A}_{g-1}]$
tautological?
Proposition
\ref{tautprop} provides no answer to the
latter question.
Motivated by Proposition \ref{tautprop},
we define
   $$\Delta_g=[\mathcal A_1\times \mathcal A_{g-1}]-\frac{g}{6|B_{2g}|}\lambda_{g-1}\ \in  \CH^{g-1}(\A_g)$$
for $g\geq 1$.

The class $\Delta_g$ detects
whether $[\mathcal{A}_1\times\mathcal{A}_{g-1}]\in \CH^{g-1}(\A_g)$ is
tautological: 
$$
[\A_1\times \A_{g-1}]\in \R^{g-1}(\A_g) \iff 
\Delta_g=0 \in
\CH^{g-1}(\A_g)
\, .$$
The vanishing $\Delta_1=0\in \mathsf{CH}^0(\A_1)$
is trivial. For $g=2$ and $g=3$, the classes $[\A_1\times \A_1] \in \CH^1(\A_2)$ and $[\A_1\times \A_2]\in \CH^2(\A_3)$ are tautological by Theorem \ref{vdgthm}(iii). The vanishings $$\Delta_2=0\, , \quad \Delta_3=0$$ were also noted in \cite [Lemma 2.2, Proposition 2.1]{vdg2}. 

For higher $g$, we will use the Torelli map to study the class $\Delta_g$.
While {\em a priori}, we know only that $\Tor^*\Delta_g\in \mathsf{CH}^{g-1}(\M_g^{\ct})$, we prove the following
stronger result by an explicit analysis of Fulton's
excess intersection class \cite{Fulton}
for the fiber product 
\[
\begin{tikzcd}
\Tor^{-1}(\A_1\times \A_{g-1}) \arrow[d] \arrow[r] & \M_g^{\ct} \arrow[d, "\Tor"] \\
\A_1\times \A_{g-1} \arrow[r]       & \ \A_g \, .                       
\end{tikzcd}
\]

\begin{thm} \label{tortaut} We have $\Tor^*\Delta_g\in \R^{g-1}(\M_g^{\ct})$.\end{thm} 

Our proof yields a formula for $\Tor^*\Delta_g$ 
in tautological classes on
$\mathcal{M}_g^{\ct}$.
By evaluating{\footnote{The evaluations are presented in
Propositions \ref{delta4} and \ref{delta5} of Section \ref{lowgenus}.}} the
formula for $g=4$ and $g=5$ and using
Pixton's relations \cite{Janda, PPZ, Pix},
we obtain the vanishings \begin{equation}\label{tordelta45}\Tor^*\Delta_4=0\, , \quad \Tor^*\Delta_5=0\, .\end{equation}

Further vanishing is established
in the following result related to the geometry of 
the moduli space of curves of compact type.

\begin{thm}\label{vang}
    For all $g$,
    the class $\Tor^*\Delta_g\in \mathsf{R}^{g-1}(\M^{\ct}_g)$ 
    lies in the kernel 
    of the $\lambda_g$-pairing
    on $\mathsf{R}^*(\mathcal{M}_g^{\ct})$.  
    \end{thm}

As a consequence of Theorem \ref{vang}, if
$\mathsf{R}^*(\mathcal{M}_g^{\ct})$
is a Gorenstein ring, then
$$\Tor^*\Delta_g=0\, .$$
The first case for which
$\mathsf{R}^*(\mathcal{M}_g^{\ct})$
is {\em not} Gorenstein is
$g=6$. The full structure
of $\mathsf{R}^*(\mathcal{M}_g^{\ct})$
has been conjectured 
by Pixton \cite{Pix} and has been proven by Canning-Larson-Schmitt \cite{CLS} for $g\leq 7$.
The kernel of the $\lambda_6$-pairing 
(called the {\em Gorenstein kernel}) is 1-dimensional
and lies in $\mathsf{R}^5(\mathcal{M}_6^{\ct})$. More precisely,
the $\lambda_6$-pairing
$$\R^4(\M_6^{\ct})\times \R^{5}(\M_6^{\ct})\,\to\,  \mathbb Q\, $$
has rank 71 while we have
$$\dim_\qq \R^4(\M_6^{\ct}) = 71\, , \ \ \
\dim_\qq \R^5(\M_6^{\ct}) = 72\, .$$

    \begin{thm}\label{Delta6}
    The class
    $\Tor^*\Delta_6=\Tor^*[\A_1\times \A_5]-\frac{2370}{691}\lambda_5$
    generates the 1-dimensional kernel of the $\lambda_6$-pairing
    $$\R^4(\M_6^{\ct})\times \R^{5}(\M_6^{\ct})\,\to\,  \mathbb Q\, .$$
    Therefore, $\Tor^*\Delta_6\neq 0\in \mathsf{R}^5(\M_6^{\ct})$ and
    $[\A_1\times \A_5]\notin \mathsf{R}^5(\A_6)$.
\end{thm}

The class $[\mathcal{A}_1 \times \mathcal{A}_5]\in \mathsf{CH}^5(\mathcal{A}_6)$ is the
first interesting non-tautological algebraic cycle class constructed on the moduli of abelian varieties. While the idea of using the
intersection theory of the Torelli map is basic,
there are reasons the study had not been
undertaken before. The first is that 
the fiber product $\Tor^{-1}(\A_1\times \A_{g-1})$
consists of many intersecting components
of excess dimension. 
The calculation of Fulton's excess class here is subtle
and requires, in particular, knowledge 
of the precise scheme structure
of $\Tor^{-1}(\A_1\times \A_{g-1})$.  The second, and perhaps more
fundamental reason, is that, until recently, the
structure of $\mathsf{R}^*(\M_g^{\ct})$
was completely unknown. Pixton's conjecture \cite{Pix} offers a
framework for understanding $\mathsf{R}^*(\M_g^{\ct})$
and plays a crucial role
in our work.


In genus $g=7$, the tautological ring $\mathsf{R}^*(\mathcal{M}_7^{\ct})$
has a 1-dimensional Gorenstein kernel \cite{CLS} as
predicted by Pixton.
We have 
\[
\dim_\qq \R^5(\M_7^{\ct}) = 277\, , \ \ \
\dim_\qq \R^6(\M_7^{\ct}) = 278\, .
\]
 The $\lambda_7$-pairing
\[
\R^5(\M_7^{\ct})\times \R^6(\M_7^{\ct})\rightarrow \qq
\]
has rank $277$.
But a surprise occurs:
the class 
$\Tor^*\Delta_7\in \R^6(\M_7^{\ct})$
does {\em not} generate the kernel of the pairing!

\begin{prop}\label{Delta7}
   We have 
   $\Tor^*\Delta_7=\Tor^*[\A_1\times \A_6]-\lambda_6=0 \in
   \mathsf{R}^6(\M_7^{\ct})$.
\end{prop}

\noindent The generator of the kernel in
$\mathsf{R}^6(\M_7^{\ct})$ of the $\lambda_7$-pairing
is constructed from
$\Tor^*\Delta_6$ in Proposition \ref{p777} of Section \ref{g777}.

For the moduli space of
curves and
abelian varieties, let 
$$
\mathsf{RH}^*(\M_g^{\ct}) \subset \mathsf{H}^*(\M_g^{\ct})\ \ \text{and} \ \ 
\mathsf{RH}^*(\A_g) \subset \mathsf{H}^*(\A_g)$$ denote the
images of 
$\mathsf{R}^*(\M_g^{\ct})$ 
and
$\mathsf{R}^*(\A_g)$
under the cycle class map (which doubles
the degree index).   
For $g=6$, the
cycle class map is an isomorphism
$$\mathsf{R}^*(\M_6^{\ct})\,  \simeq
\, \mathsf{RH}^*(\M_6^{\ct})$$
by \cite{CLS}. Hence,
$\Tor^*\Delta_6\neq 0\in  
\mathsf{RH}^{10}(\M_6^{\ct})$ and
    $$[\A_1\times \A_5]\notin \mathsf{RH}^{10}(\A_6)\, .$$


In fact, $g=6$ is the first genus where algebraic classes can
be non-tautological in cohomology.
\begin{prop}\label{algcyclee}
All algebraic cycles are tautological in cohomology for $g\leq 5$.     
\end{prop}
\begin{proof} By \cite[Theorem 17, Theorem 32]{HT2}, the intersection cohomology $\mathsf{IH}^*(\A_g^{\mathrm{Sat}})$ of the Satake compactification is tautological when $g\leq 5$. Since $\mathsf{IH}^*(\A_g^{\mathrm{Sat}})$ surjects onto the pure weight cohomology of $\A_g$, see \cite[Lemma 2]{Durfee}, and algebraic cycles are of pure weight, the Proposition follows.\end{proof}
Taïbi \cite[Theorem 33]{HT2} has furthermore shown that $\mathsf{IH}^k(\A_g^{\mathrm{Sat}})$ is tautological for $k<2g-2$. Therefore, all algebraic cycles of codimension less than $g-1$ are tautological in cohomology. 

 Based on Theorem \ref{Delta6},
 Proposition \ref{Delta7}, and Proposition \ref{algcyclee}, our expectation is
 \begin{align*}
\Delta_g = 0 \in \mathsf{CH}^{g-1}(\A_g) & \text{\qquad for $2\leq g\leq 5$ and $g=7$\,,}\\
\Delta_g \neq 0 \in \mathsf{CH}^{g-1}(\A_g) & \text{\qquad for $g\geq 6$, $g\neq 7$}\,.
\end{align*}
Iribar L\'opez \cite{Iribar} has 
subsequently
found
a proof of the non-vanishing 
of $\Delta_g$ in
Chow for $g=12$ and even $g\geq 16$.
So for even $g$, only the cases
$g=4,8,10,14$  are open.

\subsection{Product extension}
Since basic classes such as product loci
should be included in
a tautological calculus for
$\mathcal{A}_g$, proposals
to enlarge the tautological
ring are natural to consider.
The simplest extension of $\mathsf{R}^*(\mathcal{A}_g)$  is obtained by considering the
closure 
$$\mathsf{R}^*_{\mathrm{pr}}(\mathcal{A}_g)
\subset \CH^*(\mathcal{A}_g)$$
of $\mathsf{R}^*(\mathcal{A}_g)$
under all product maps.

\begin{definition}\label{rprod}
    Define 
$\mathsf{R}^*_{\mathrm{pr}}(\mathcal{A}_g)
\subset \CH^*(\mathcal{A}_g)
$
to be the $\qq$-vector subspace generated by
all classes
$$[\mathcal{A}_{g_1} \times \mathcal{A}_{g_2} \times \cdots \times \mathcal{A}_{g_\ell}, \mathsf{P}(\lambda^1,\lambda^2, \ldots, \lambda^\ell)]\, 
\in\,  \CH^*(\mathcal{A}_g)
$$
with $g=\sum_{i=1}^\ell g_i$ and all $g_i\geq 1$. Here,
$\lambda^i$ denotes the set of all $\lambda$ classes on the factor $\A_{g_i}$,
$$\lambda_1, \ldots, \lambda_{g_i}\in \CH^*(\mathcal{A}_{g_i})\,, $$
and 
$\mathsf{P}\in \qq[\lambda^1,\ldots, \lambda^\ell]$
is an arbitrary polynomial.
\end{definition}

While the definition of $\mathsf{R}^*_{\mathrm{pr}}(\mathcal{A}_g)$
leaves behind the connection to the
stable cohomology of
$\Sp_{2g}(\mathbb Z)$, the closure under
products is natural from the perspective of the tautological ring
of $\overline{\mathcal{M}}_{g,n}$ with
respect to the boundary gluing maps.

\begin{prop}\label{p10} The subspace $\mathsf{R}^*_{\mathrm{pr}}(\mathcal{A}_g)$ satisfies the following properties:
\begin{enumerate}
\item[\textnormal{(i)}] 
$\mathsf{R}^*_{\mathrm{pr}}(\mathcal{A}_g)$ is closed under multiplication, so is a $\qq$-algebra.
\item[\textnormal{(ii)}]   There is a product pushforward
$$\mathsf{R}^*_{\mathrm{pr}}(\mathcal{A}_{g_1})
\times 
\mathsf{R}^*_{\mathrm{pr}}(\mathcal{A}_{g_2})
\rightarrow 
\mathsf{R}^*_{\mathrm{pr}}(\mathcal{A}_{g_1+g_2})\, .$$
\item[\textnormal{(iii)}] $\mathsf{R}^{>\binom{g}{2}}_{\mathrm{pr}}(\mathcal{A}_{g}) =0$.
\item[\textnormal{(iv)}] 
$\mathsf{R}^*(\mathcal{A}_{6}) \subsetneq
\mathsf{R}^*_{\mathrm{pr}}(\mathcal{A}_{6})$.
\end{enumerate}
\end{prop}
Part (ii) holds by definition, and part (iv) is
consequence of Theorem \ref{Delta6}. Parts (i) and (iii)
will be proven in Section \ref{extended}. A natural conjecture concerns the
codimension $\binom{g}{2}$ classes.

\begin{conj}\label{socprod}  For all $g\geq 1$, $\mathsf{R}^{\binom{g}{2}}_{\mathrm{pr}}(\mathcal{A}_{g})  \cong 
\qq$.
\end{conj}

The class of the locus of abelian varieties that
factor completely, 
$[\mathcal{A}_1 \times \cdots\times \mathcal{A}_1] \in \mathsf{R}_{\mathrm{pr}}^{\binom{g}{2}}(\mathcal{A}_g),$
provides
a candidate for the generator of $\mathsf{R}^{\binom{g}{2}}_{\mathrm{pr}}(\mathcal{A}_{g})$.
Conjecture \ref{socprod} is equivalent to
the following assertion: {\em for all $g \geq 1$},
\begin{equation}\label{ffvvgg}
[\, \underbrace{\mathcal{A}_1  \times \cdots\times \mathcal{A}_1}_{g}\, ] \in \mathsf{R}^{\binom{g}{2}}(\mathcal{A}_g)\, .
\end{equation}
In fact, a
sharper claim can be made \cite[Theorem 6]{CMOP}: 
{\em if \eqref{ffvvgg} holds, then}
$$[\, \underbrace{\mathcal{A}_1  \times \cdots\times \mathcal{A}_1}_{g}\, ] =\left(\prod_{k=1}^{g} \frac{k}{6|B_{2k}|}\right) \lambda_{1}\cdots \lambda_{g-1} \in
\mathsf{R}^{\binom{g}{2}}(\mathcal{A}_g)
\, .$$ 
For $g\leq 3$, we have $\mathsf{R}^*(\mathcal{A}_{g}) =
\mathsf{R}^*_{\mathrm{pr}}(\mathcal{A}_{g})$
since both are the full Chow ring by Theorem \ref{vdgthm}(iii). Therefore,  
Conjecture \ref{socprod} is 
true for $g\leq 3$. For $g=4$, 
Conjecture 10 is proven
in Proposition \ref{4tautological} in Section \ref {excesstheory}.
 For $g\geq 5$, the question is open.

\subsection{Noether--Lefschetz loci}
A further expansion of $\mathsf{R}^*(\mathcal{A}_g)$ via
Noether-Lefschetz loci 
is motivated by the
study of tautological classes \cite{MOP,PY} on the 
the moduli space of quasi-polarized $K3$ surfaces. 

The very general principally
polarized abelian variety $(X,\Theta)$ has N\'eron-Severi group
$$\mathsf{NS}(X)\cong\mathbb{Z}\,.$$
However, the N\'eron-Severi
rank can jump on special subvarieties of $\mathcal{A}_g$.
For each $r$, let
$$\mathsf{NL}^r_g \subset \mathcal{A}_g$$
be the 
{\em Noether-Lefschetz locus} of
abelian varieties with 
$$\mathsf{NS}(X)\cong \mathbb{Z}^r\, .$$
The locus $\mathsf{NL}^r_g$ is a
countable union of irreducible locally closed
substacks of $\mathcal{A}_g$.

Let $\mathsf{NL}^r_g\subset 
\overline{\mathsf{NL}}^r_g$ denote
the Zariski closure in $\A_g$.
A {\em marked irreducible component of
$\overline{\mathsf{NL}}^r_g$}
is a moduli space $\mathcal{S}$ of principally polarized abelian varieties 
$(X,\Theta,\phi)$
with the data of a {\em marking}
$$\phi: \mathbb{Z}^r \hookrightarrow \mathsf{NS}(X)\, $$
satisfying two properties:
\begin{enumerate}
    \item [(i)] the polarization lies in the image of $\phi$, $$\phi(1,0,\ldots,0)= \Theta\,,$$
    \item[(ii)] the induced map
    $\iota_{\mathcal{S}}: \mathcal{S} \rightarrow \overline{\mathsf{NL}}^r_g\subset \A_g$
    surjects
    onto an irreducible component of $\overline{\mathsf{NL}}^r_g$.
\end{enumerate}
Two marked abelian varieties $(X,\Theta, \phi)$ and $(X',\Theta',\phi')$
are isomorphic
if there are isomorphisms
$$\alpha:X \rightarrow X'\, , \ \ \ \beta:\mathbb{Z}^r\rightarrow \mathbb{Z}^r$$
satisfying $\alpha^*\Theta' =\Theta$ and
$\alpha^*\circ \phi'= \phi \circ \beta $.

Marked irreducible components $\mathcal{S}$
 are algebraic, see \cite{CDK} or \cite[Remarque 1]{DL}. Moreover,
 $\mathcal{S}$ admits a canonical quotient
 presentation with respect to a subgroup
 $\mathsf{G}_{\mathcal{S}}\subset \mathsf{Sp}(2g, \mathbb{Z})$ and carries
 automorphic algebraic vector bundles.

\begin{definition}
    Define 
$\mathsf{R}^*_{\mathrm{NL}}(\mathcal{A}_g) \subset \CH^*(\mathcal{A}_g)$
to be the $\qq$-subalgebra generated by
all classes
$$\iota_{\mathcal{S}*}(\mathsf{P})\, 
\in\,  \CH^*(\mathcal{A}_g)\, ,
$$
where $\mathcal{S}$ is a marked 
irreducible component of
$\overline{\mathsf{NL}}^r_g$ 
 and $\mathsf{P}$
is a polynomial in the Chern classes
of automorphic algebraic vector bundles
on $\mathcal{S}.$
\end{definition}

Further extensions of the tautological ring of $\A_g$ would more generally include  Hodge loci corresponding to arbitrary Hodge types, see \cite [Section 3]{MO} or \cite{GGK, V} for definitions. We will not pursue these directions here. Related constructions regarding the tautological rings of Shimura varieties via Chern classes of automorphic bundles are discussed in \cite {TZ}. 

By definition{{\footnote{We easily see
that the classes $\lambda^i$ arise
from automorphic bundles on the
marked irreducible component of $\overline{\mathsf{NL}}^\ell_g$ determined by
$\mathcal{A}_{g_1} \times \mathcal{A}_{g_2} \times \cdots \times \mathcal{A}_{g_\ell}
\rightarrow \overline{\mathsf{NL}}^\ell_g \subset \A_g$.}}, we have inclusions of tautological rings
\begin{equation} \label{propinc}
\mathsf{R}^*(\mathcal{A}_g) \subset
\mathsf{R}^*_{\mathrm{pr}}(\mathcal{A}_g) \subset
\mathsf{R}^*_{\mathrm{NL}}(\mathcal{A}_g)\, . 
\end{equation}
Both inclusions are equalities for $g\leq 3$.
We have seen that the first inclusion in \eqref{propinc} is strict
for $g=6$. Iribar L\'opez observes
in \cite{Iribar} that the Torelli pullback $\mathrm{NL}_{2,g}$ to $\M_g$ is exactly the
bielliptic locus. Here, $\mathrm{NL}_{2,g}$ is the locus of abelian $g$-folds containing an elliptic curve such that the induced polarization on the elliptic curve is of degree $2$. As the bielliptic locus is non-tautological in $\mathsf{CH}^{g-1}(\mathcal{M}_g)$
for $g=12$ and $g\geq 16$ even \cite{ACC+,Zelm}, 
Iribar L\'opez concludes
$$[\mathrm{NL}_{2,g}] \notin
\mathsf{R}^{g-1}_{\mathrm{pr}}(\mathcal{A}_g) \,.$$
Therefore, the
second inclusion
of \eqref{propinc} is
also strict.

\begin{conj}\label{soclevan}
    The ring $\mathsf{R}^*_{\mathrm{NL}}(\A_g)$ satisfies the following socle and vanishing properties:
    \begin{enumerate}
        \item [\textnormal{(i)}]$\mathsf{R}^{\binom{g}{2}}_{\mathrm{NL}}(\A_g)\cong \qq\,$.
        \item [\textnormal{(ii)}] $\mathsf{R}^{k}_{\mathrm{NL}}(\A_g)=0$ for $k>\binom{g}{2}\,$.
    \end{enumerate}
\end{conj}
By Proposition \ref{algcyclee} and Theorem \ref{vdgthm}, the conjecture is true in cohomology for $g\leq 5$. In fact, the stronger vanishing $$\H^{k}(\A_g, \mathbb Q)=0, \quad k>2\binom{g}{2}$$ is expected, see \cite [Question 1.1]{toptop}, as well as equations (1) and (2) there for supporting results.

\subsection{Noether-Lefschetz loci of rank 2 and virtual fundamental classes}

The  
Noether--Lefschetz locus 
$\mathsf{NL}^2_g$
plays a special
role in the geometry of $\mathcal{A}_{g}$.
Debarre and Laszlo \cite{DL} have classified  the irreducible components of the Noether--Lefschetz locus
of rank 2. 
\begin{thm}[Debarre-Laszlo] \label{DebLaz} The irreducible components of the closure of the Noether--Lefschetz locus 
$\mathsf{NL}_g^2 \subset \mathcal{A}_g$ are:
\begin{itemize}
    \item[\textnormal{(i)}] For each integer $1\leq k\leq \frac{g}{2}$, the locus of principally polarized abelian varieties containing an abelian subvariety of dimension $k$ such that the induced polarization is of a fixed degree.
    \item[\textnormal{(ii)}] For every divisor $n$ of $g$, $n\neq g$, the irreducible components of the locus of Shimura--Hilbert--Blumenthal varieties. 
\end{itemize}
\end{thm}

The Shimura--Hilbert--Blumenthal varieties parametrize abelian varieties $X$ whose endomorphism algebra $\textnormal{End}(X)\otimes_{\mathbb Z} \mathbb Q$ contains a totally real subfield. The components in Theorem \ref{DebLaz}(i) that arise when the induced polarization is principal are exactly the product loci $\A_k\times \A_{g-k}$. 

The expected codimension of the Noether--Lefschetz locus
of rank 2 is $$\binom{g}{2}=\dim H^{2,0}(X)\, ,$$
for any abelian variety $X$, while the actual dimension can be different. Every marked irreducible component of 
$\mathsf{NL}^2_g$ carries a virtual fundamental class
$$[\mathcal{S}]^{\mathrm{vir}} \in  \CH^{\binom{g}{2}}(\A_g)\, ,$$
as constructed in Section \ref{virtual}.

\begin{prop}\label{vclass}
    The virtual class of the locus $\A_k\times \A_{g-k}$ of products is given by
    \[
    [\A_k\times \A_{g-k}]^{\mathrm{vir}}=(-1)^{\binom{k}{2}}\prod_{i=1}^{k-1}\lambda_i\otimes (-1)^{\binom{g-k}{2}}\prod_{j=1}^{g-k-1}\lambda_j\,.
    \]
\end{prop}

If Conjecture \ref{socprod} is correct, then we have the
following consequence:
$$  [\A_k\times \A_{g-k}]^{\mathrm{vir}} \in  \mathsf{R}^{\binom{g}{2}}(\A_g)\, . $$
Perhaps the virtual fundamental classes are always in van der Geer's
tautological ring?

\begin{speculation}\label{socvir}  For all $g\geq 1$
and all
marked irreducible components $\mathcal{S}$ of $\mathsf{NL}^2_g \subset \mathcal{A}_g$, we have
$$  [\mathcal{S}]^{\mathrm{vir}} \in  \mathsf{R}^{\binom{g}{2}}(\A_g)\, . $$
\end{speculation}
\noindent 
Whenever Speculation \ref{socvir} is true, the structure of the proportionalities $$[\mathcal{S}]^{\mathrm{vir}}\in \mathsf{R}^{\binom{g}{2}}(\A_g) \cong \qq$$
as $\mathcal{S}$ varies among irreducible components is an interesting question.
The $g=2$ case, where the
virtual and fundamental classes coincide, has been solved
by van der Geer \cite{vdgold} in
terms of 
a Fourier expansion of a modular form.\footnote{See \cite{Iribar} for a discussion.}

\subsection {Plan of the paper}  We start in Section \ref{extended} by studying intersections of product loci and properties of the product tautological rings. In particular, Proposition \ref{p10} is established. In Section \ref{elllocus}, we  compute the class of the product $[\A_1\times \A_{g-1}]$ when tautological, thus proving Proposition \ref{tautprop}.  Theorem \ref{vang} is also proven in Section \ref{elllocus}.  In Sections \ref{tpullbacks} and \ref{excesstheory}, we calculate the class $\Tor^*\Delta_g$ via excess intersection theory, and establish Theorem \ref{tortaut}. In Section \ref{lowgenus}, we present low genus examples, and prove Theorem \ref{Delta6} and Proposition \ref{Delta7}. 
In Section \ref{virtual}, we discuss the virtual fundamental classes of the Noether-Lefschetz loci and prove Proposition \ref{vclass}.

\subsection*{Acknowledgments}
We thank Valery Alexeev, Ben Bakker, Carel Faber, 
Gerard van der Geer, 
Francois Greer,
Sam Grushevsky, 
Klaus Hulek, 
Daniel Huybrechts, Carl Lian, Aitor Iribar L\'opez, Hannah Larson, Sam Molcho, Dan Petersen,  
Johannes Schmitt,  Olivier Ta\"ibi, Burt Totaro, and Don Zagier 
for conversations which
have significantly improved our understanding of cycles on 
$\M_g^{\ct}$ and $\A_g$.
Aaron Pixton has contributed many ideas (including both the formulation and the strategy of proof of Theorem \ref{vang} and the formula
of Section \ref{Pixformula}).
We have used the software package \texttt{admcycles} \cite{admcycles} 
for computations in the tautological rings of the moduli of curves.

The paper started with a 
discussion at the piano
bar {\em Vincent} overlooking
the Spree during the conference
{\em Resonance, topological invariants of groups, and moduli} at HU Berlin in November 2022 organized by Gavril Farkas. 
The research here was motivated
in part by earlier work with
Davesh Maulik on the moduli space
of $K3$ surfaces \cite{MaulikP}.

S.C. was supported by a Hermann-Weyl-Instructorship from the Forschungsinstitut f\"ur Mathematik at ETH Z\"urich.
R.P. was supported by
SNF-200020-182181,  
SNF-200020-219369,
ERC-2017-AdG-786580-MACI, and SwissMAP. 

This project has
received funding from the European Research Council (ERC) under the European Union Horizon 2020 research and innovation program (grant agreement No. 786580).

\section{Intersection theory of product loci}\label{extended}
 
 \subsection{Overview}
We prove here that product loci in $\A_g$ always intersect each other trivially in $\mathsf{CH}^*(\A_g)$. As a consequence, we give a proof of Proposition \ref{p10}.

\subsection{Unique decomposition 
for abelian varieties}

A principally polarized abelian
variety
$(A,\Theta)\in \A_g$ is {\em decomposable} if 
$$(A,\Theta) \in \A_{g_1}\times \A_{g_2}$$ for some $g_1+g_2=g$ with $g_i\neq 0, g$. If $(A,\Theta)$ is decomposable, then the theta divisor $\Theta$ is reducible. Conversely, if $\Theta$ is reducible, then $(A,\Theta)$ is decomposable by a result of Shimura, see \cite[Lemma 3.20]{CG}. 
The following result is \cite[Corollary 3.23]{CG}. 
\begin{prop}\label{Splittingtheorem}
 A principally polarized abelian variety $(A,\Theta)$ decomposes uniquely, up to reordering, as a product of indecomposable principally polarized abelian varieties.
\end{prop}

\subsection{The intersection product}
Associated to the partition $g=g_1+\ldots+g_\ell$ is  
the finite morphism $$\A_{g_1}\times \ldots\times \A_{g_\ell}\to \A_g\, .$$ The pushforward of the fundamental class is the cycle
$[\A_{g_1}\times \ldots \times \A_{g_\ell}] \in \CH^*(\A_g).$

\begin{prop}\label{zeroint}
    The intersection product
    vanishes, \[
    [\A_{g_1}\times \ldots\times \A_{g_\ell}]\cdot [\A_{h_1}\times \ldots\times \A_{h_k}] = 0 \in \CH^*(\A_g)\, ,
    \]
for all partitions $g_1+\ldots+g_{\ell}=h_1+\ldots+h_k=g$ with $\ell\geq 2$ and  $k\geq 2$. 
\end{prop}
\begin{proof}
First, we establish that \begin{equation}\label{self}[\A_{g_1}\times \ldots\times \A_{g_\ell}]^2=0\, .\end{equation} Strictly speaking, the case \eqref{self} does not require a separate discussion, but the simpler analysis  illustrates the main point. Using the self-intersection formula, it suffices to prove that the normal bundle of the morphism $$p: \A_{g_1}\times \ldots\times \A_{g_\ell}\to \A_g$$ has vanishing Euler class. The tangent bundle to the moduli stack of principally polarized abelian varieties is $T\A_g=\Sym^2 \mathbb E^{\vee}_g.$ Furthermore, we have the splitting \begin{equation}\label{hodgespl}p^*\mathbb E_{g} = \mathbb E_{g_1} \boxplus \ldots \boxplus \mathbb E_{g_\ell}.\end{equation} Therefore, the normal bundle of the morphism $p$ equals \begin{equation}\label{normalbun}\mathcal N=\text{Sym}^2 (\mathbb E_{g_1} \boxplus \ldots \boxplus \mathbb E_{g_\ell})^{\vee}-\Sym^2\mathbb E_{g_1}^{\vee} -\ldots - \Sym^2 \mathbb E_{g_{\ell}}^{\vee}=\bigoplus_{\{i, j\}} \mathbb E_{g_i}^{\vee} \boxtimes \mathbb E_{g_j}^{\vee}.\end{equation} The sum is taken over the 2-element sets $\{i, j\}\subset \{1, \ldots, \ell\}.$ We will repeatedly use the following remark concerning the Euler classes of two vector bundles $\mathcal V , \mathcal W$ and their tensor product $\mathcal V\otimes \mathcal W$: \begin{equation}\label{eulertensor}\mathsf e(\mathcal V)=0 \ \ \text{and} \ \  \mathsf e(\mathcal W)=0\ \implies\ \mathsf e(\mathcal V\otimes \mathcal W)=0\, .\end{equation} This assertion is clear if $\mathcal V, \mathcal W$ are both line bundles, while the general case follows by the splitting principle. In our case, the Hodge bundles $\mathbb E_{g_i}$ have trivial Euler classes, so \eqref{eulertensor} implies that the same is true about the normal bundle $\mathcal N.$ We conclude the
vanishing  \eqref{self}. 

Before going to the general case, we consider another simpler situation, 
\begin{equation}\label{ee}[\A_{1} \times \A_{g-1}] \cdot [\A_{k} \times \A_{g-k}] = 0\, .\end{equation} We may assume $k\neq 1$ since the case $k=1$ was considered above. Let $Z = \A_{1} \times \A_{g-1}$, and let $W =\A_{k} \times \A_{g-k}$. 
Consider the fiber product diagram:
\[
\begin{tikzcd}
\F \arrow[d] \arrow[r] & W\arrow[d] \\
Z \arrow[r]       & \, \A_g \, .                       
\end{tikzcd}
\]
 By Proposition \ref{Splittingtheorem}, there are $2$ disjoint components of the fiber product $\F$ when $g\neq 2k$, corresponding to whether the elliptic factor in $Z$ belongs to the dimension $k$ or dimension $g-k$ factor in $W$. Therefore, $$X=\A_{1}\times \A_{k-1}\times \A_{g-k} \ \  \text{and}\ \  Y=\A_{1}\times \A_{k}\times \A_{g-k-1}$$ of codimension $k-1$ and $g-k-1$ in $W$ respectively. The case $g=2k$ is special: $\F$ has a single component $X=Y$.

The contributions of $X$ and $Y$ to the intersection product \eqref{ee} are found by an excess bundle calculation. For $X$, we compute the excess bundle with the aid of \eqref{normalbun}. We find 
$$ \mathcal N_{Z/\A_g}\bigg|_X - \mathcal N_{X/W} = \mathbb E^{\vee}_1 \boxtimes (\mathbb E^{\vee}_{k-1} \boxplus\mathbb E^{\vee}_{g-k}) -\mathbb E^{\vee}_1 \boxtimes \mathbb E^{\vee}_{k-1} = \mathbb E^{\vee}_1 \boxtimes \mathbb E^{\vee}_{g-k}\,.$$ 
Since $Z$ has codimension $g-1$ in $\A_g$,
we must select the Chern class of degree 
$$g-k=(g-1)-(k-1)\, ,$$ which is the Euler class of the tensor product $\mathbb E_1^{\vee}\boxtimes \mathbb E_{g-k}^{\vee}$.
The Euler class vanishes by \eqref{eulertensor}. The analysis for $Y$ is similar.

For the general case, we form the fiber product diagram \[
\begin{tikzcd}
\mathcal F\arrow[d] \arrow[r] & \A_{h_1}\times \ldots \times\A_{h_k}\arrow[d] \\
\A_{g_1} \times \ldots \times\A_{g_\ell}\arrow[r]       & \A_g \, .                       
\end{tikzcd}
\] For simplicity, we write $$Z=\A_{g_1} \times \ldots\times \A_{g_\ell}\ \ \text{and}\ \  W=\A_{h_1}\times \ldots \times \A_{h_k}\, .$$ To identify the components of the fiber product $\mathcal F,$ we follow an argument similar to \cite [Proposition 9]{GP} in the context of the moduli of curves. A partition  $\sigma_1+\ldots+\sigma_p=g$ {\em refines} the partition $\tau_1+\ldots+\tau_n = g$ if there exists a decomposition into disjoint sets $$\{1, \ldots, p\}=I_1\sqcup \ldots \sqcup I_{n}\,$$ such that for all $1\leq j\leq n$, we have $$\sum_{i\in I_{j}} \sigma_i = \tau_j\,.$$ Each refinement $\sigma$ of $\tau$ determines the tuple $(I_1, \ldots, I_n)$ inducing a morphism $$\A_{\sigma_1}\times \ldots \times \A_{\sigma_{p}}\to \A_{\tau_1} \times \ldots \times \A_{\tau_n}.$$ We write $\sigma\to \tau$ to indicate refinement, with the sets $(I_1, \ldots, I_n)$ being understood (though not explicitly recorded by the notation). 

Let us abbreviate $\vec g, \vec h$ for the two partitions $g=g_1+\ldots+g_{\ell},$ and $g=h_1+\ldots+h_{k}.$ Let $\Sigma$ denote the set of all partitions $\sigma$ that refine both $\vec g$ and $\vec h$, or more precisely triples $$(\sigma,\,\, \sigma \to \vec g,\,\, \sigma\to \vec h)\,.$$ Each $\sigma\in \Sigma$ induces morphisms $$\A_{\sigma_1} \times \ldots\times \A_{\sigma_{p}}\to \A_{g_1}\times \ldots \times \A_{g_\ell}\, , \ \ \ \A_{\sigma_1} \times \ldots \times \A_{\sigma_{p}}\to \A_{h_1}\times \ldots \times \A_{h_k}\,,$$ and thus a morphism to the fiber product $\A_{\sigma_1} \times \ldots\times \A_{\sigma_{p}}\to \mathcal F.$ 
The set $\Sigma$ can be ordered by (further) refinement. We consider the extremal refinements $\sigma$ which do not arise as further refinements of other members of $\Sigma$. Then $$\mathcal F=\bigsqcup_{\sigma \text{ extremal}} \A_{\sigma_1}\times \ldots \times \A_{\sigma_p}.$$ The disjoint union $\mathcal F$ is indexed by the extremal partitions $\sigma$, and for each such partition, the index $p$ is defined as the length of $\sigma$. For each component $X=\A_{\sigma_1}\times \ldots \times \A_{\sigma_p}$, the excess bundle equals
$$\mathsf {V}_X=\mathcal N_{Z/\A_g}\bigg|_X - \mathcal N_{X/W}\,.$$ Using \eqref{normalbun}, we have $$\mathcal N_{Z/\A_g}=\bigoplus_{\{i,j\}} \mathbb E_{g_i}^{\vee} \boxtimes \mathbb E_{g_j}^{\vee}\implies \mathcal N_{Z/\A_g}\bigg|_{X}=\bigoplus_{\{i,j\}} \bigoplus_{\alpha \in I_i, \,\beta\in I_j} \mathbb E_{\sigma_{\alpha}}^{\vee} \boxtimes \mathbb E_{\sigma_{\beta}}^{\vee}.$$ Here $\{i, j\}\subset \{1, \ldots, \ell\}$ is any set with $2$ distinct elements, and $(I_1, \ldots, I_\ell)$ correspond to the refinement $\sigma \to \vec g.$ Similarly, let $(J_1, \ldots, J_k)$ denote the sets corresponding to the refinement $\sigma \to \vec h.$ Then, using \eqref{normalbun} again, we find  
$$\mathcal N_{X/W}=\bigoplus_{s} \bigoplus_{\{\alpha, \beta\} \subset J_s} \mathbb E_{\sigma_{\alpha}}^{\vee}\boxtimes \mathbb E_{\sigma_{\beta}}^{\vee}\, .$$ 

Of course, $\mathsf {V}_X$ is an actual bundle. Indeed, for each set $\{\alpha,\beta\}\subset J_s$ with two elements, we let $\alpha \in I_i$ and $\beta\in I_j$ for some $\{i, j\}\subset \{1, \ldots, \ell\}$. We only need to show $i\neq j$. Assuming $i=j$, we can form the partition $\tau$ replacing the parts $(\sigma_\alpha, \sigma_\beta)$ of $\sigma$ by the sum  $\sigma_\alpha+\sigma_\beta$. Furthermore, we place the sum in the sets $I_i$ and $J_s$. The new partition $\tau$ thus remains a common refinement of $\vec g$ and $\vec h$, so $\tau\in \Sigma$. Furthermore $\sigma$ is a refinement of $\tau$, which contradicts  the
extremality of $\sigma$ in $\Sigma$. As a result, $\mathsf{V}_X$ is sum of various tensor products of Hodge bundles $\mathbb E_{\sigma_\alpha}^{\vee}\boxtimes \mathbb E_{\sigma_\beta}^{\vee}$, so the Euler class of $\mathsf{V}_X$ vanishes by \eqref{eulertensor}. \end{proof}

\subsection{Proof of Proposition \ref{p10}} Part (i) follows from Proposition \ref{zeroint} which shows more generally that the product of two classes supported on product loci vanishes. Part (ii) is clear by definition, while part (iv) is a consequence of Theorem \ref{Delta6}, which will be established below. 



To establish part (iii), consider a nonzero class of the form
$$[\mathcal{A}_{g_1} \times \mathcal{A}_{g_2} \times \cdots \times \mathcal{A}_{g_\ell}, \mathsf{P}(\lambda^1,\lambda^2, \ldots, \lambda^\ell)]\, 
\in\,  \CH^*(\mathcal{A}_g)\,.
$$ Since $\mathsf{R}^*(\A_{g_i})$ vanishes in degree $>\binom{g_i}{2},$ the above class has degree at most 
$$\sum_{i=1}^{k} \binom{g_i}{2}+\mathrm{codim } (\A_{g_1}\times \ldots \times \A_{g_k}/\A_g)=\binom{g}{2}\,,$$ as claimed. 
\qed

\section{Products with an elliptic factor}\label{elllocus}

\subsection{Overview}
We prove here Proposition \ref{tautprop}, which determines the class $[\A_1\times \A_{g-1}]\in \CH^*(\A_g)$ in the  tautological case. See \cite[Theorem 6]{CMOP} for a different argument. The proof below gives slightly more and will be used to establish Theorem \ref{vang}.\vskip.1in

\subsection{Proof of Proposition \ref{tautprop}} By  \cite {vdg}, the monomials $\lambda_J=\prod_{j\in J} \lambda_j$ with $J\subset \{1, 2, \ldots, g-1\}$ determine a basis for the $\mathbb Q$-vector space $\R^*(\mathcal A_g)$. 
If $[\A_1\times \A_{g-1}]$ is tautological, we can write
    \[
    [\A_1\times \A_{g-1}]=\sum_{J} c_J \lambda_J\, , \ \ \ c_J \in \mathbb{Q}\, .
    \]
    The summation here runs over subsets $J\subset \{1,\dots,g-1\}$ such that the sum of all elements in $J$ is $g-1$.
    
    As seen in \eqref{hodgespl}, the Hodge bundle splits as a sum over the factors $$\mathbb E_{g}\bigg|_{\mathcal A_1\times \mathcal A_{g-1}}=\mathbb E_1\boxplus \mathbb E_{g-1}\,.$$ Using the vanishing \eqref{lambdavanish} applied to $\A_1$ and $\A_{g-1}$, we find 
    \[
    \lambda_{g-1}[\A_1\times \A_{g-1}]=0\ \implies\ \sum_{J} c_J \lambda_{g-1}\lambda_J = 0\,.
    \]
In the sum on the right, the term corresponding to $J=\{g-1\}$ vanishes by the Mumford relation $$\lambda_{g-1}^2=0\,.$$ For the remaining terms, we must have $g-1\not \in J$ since the sum of elements of $J$ is $g-1$. Then, the monomials $\lambda_{g-1}\lambda_J$ are part of the basis for $\R^*(\mathcal A_g),$ and therefore $c_J=0$. We conclude that 
 \[
    [\A_1\times \A_{g-1}]=c\, \lambda_{g-1}\,.
    \]
for some constant $c.$ 

    To determine the constant $c$, we pull back to $\M_g^{\ct}$ under the Torelli map $\Tor$, and we intersect both sides with $\lambda_{g-2}$. In $\CH^{2g-3}(\M_g^{\ct}),$ we obtain
    \[
    \lambda_{g-2}\cdot \Tor^*[\A_1\times \A_{g-1}]=c\,\lambda_{g-2}\lambda_{g-1}\,.
    \] On the left hand side, $\Tor^*[\A_1\times \A_{g-1}]$ will be computed in Section \ref{excesstheory} via excess intersection theory. As we will see in \eqref{torexpl} below, the resulting expression takes the form $$\Tor^*[\A_1\times \A_{g-1}]=\sum_{\mathsf T}\frac{1}{|\mathrm{Aut}_{\mathsf T}|}\iota_{\mathsf T*\,}\textsf{Cont}_{\mathsf T}\,.$$ 
Here, $\mathsf T$ is a tree whose vertices carry genus decorations. The tree possesses a genus $1$ root, and the remaining genera sum up to $g-1$. In addition, all genus $0$ vertices must have valence at least $3$. The contribution $\mathsf{Cont}_\mathsf{T}$ corresponding to the tree $\mathsf T$ is supported on the boundary stratum in $\M_g^{\ct}$ of curves with dual graph $\mathsf T$. The map $\iota_{\mathsf T}$ denotes the inclusion of this stratum. 

Multiplying 
$\Tor^*[\A_1\times \A_{g-1}]$
by $\lambda_{g-2}$ sends all but one of the contributions to zero. Indeed, using \eqref{lambdavanish}, we see that $\lambda_{g-2}$ vanishes on all trees whose vertices have genera at most $g-2$. The remaining contribution comes from the divisor $$\iota: \M_{1, 1}^{\ct}\times \M_{g-1,1}^{\ct}\to \M_g^{\ct}\,.$$ We will see in equation \eqref{excess} that the excess contribution equals $\left[ \frac{c(\mathbb{E}^{\vee})}{1-\psi_1}\right]_{g-2},$ where the subscript denotes selecting the indicated degree. The Hodge bundle and the $\psi$-class here are over the second factor. Therefore, 
    \[
    \lambda_{g-2}\cdot \Tor^*[\A_1\times \A_{g-1}]=\lambda_{g-2}\cdot \iota_* \left[ \frac{c(\mathbb{E}^{\vee})}{1-\psi_1}\right]_{g-2}=c \,\lambda_{g-2}\lambda_{g-1}\,. 
    \] 
    
    Next, we apply the canonical evaluation $\epsilon^{\ct}$ introduced in \eqref{evaluation} to both sides of the above identity. Both sides extend naturally to the compactification $\overline \M_g$. Therefore $$\int_{\overline \M_g} \lambda_{g-2}\lambda_g \cdot \iota_*  \left[ \frac{c(\mathbb{E}^{\vee})}{1-\psi_1}\right]_{g-2}=c\int_{\overline \M_g} \lambda_{g-1} \lambda_{g-2} \lambda_g\,.$$
Using the splitting of the Hodge bundle \eqref{hodgespl}, we see that $$\iota^* (\lambda_{g-2} \lambda_{g})=\lambda_1\boxtimes \lambda_{g-1}\lambda_{g-2}$$ over $\overline \M_{1, 1}\times \overline \M_{g-1, 1}.$ We have $\int_{\overline \M_{1, 1}}\lambda_1=\frac{1}{24}$. Furthermore, $$ \int_{\overline \M_{g}} \lambda_{g} \lambda_{g-1}\lambda_{g-2} = \frac{1}{2} \int_{\overline \M_{g}} \lambda_{g-1}^3=\frac{1}{2(2g-2)!}\cdot \frac{|B_{2g}|}{2g} \cdot \frac{|B_{2g-2}|}{ 2g-2}\,.$$ The first equality follows from Mumford's relations, while the second integral was calculated in \cite [Theorem 4]{FP1}. We therefore obtain $$\frac{1}{24} \int_{\overline \M_{g-1, 1}} \frac{c(\mathbb E^\vee)}{1-\psi_1}\lambda_{g-1}\lambda_{g-2} =  \frac{c}{2(2g-2)!} \cdot \frac{|B_{2g}|}{2g} \cdot \frac{|B_{2g-2}|}{ 2g-2}\,.$$ To confirm the value of the constant $c=\frac{g}{6|B_{2g}|},$ we must show \begin{equation}\label{integral}\int_{\overline \M_{g, 1}} \frac{c(\mathbb E^{\vee})}{1-\psi_1} \lambda_{g} \lambda_{g-1}=\frac{|B_{2g}|}{2g\cdot (2g)!}\,,\end{equation} where we have shifted from $g-1$ to $g$. 

The integral \eqref{integral} can also be extracted from \cite {FP1}. 
Set $$\Lambda(z)=\sum_{i=0}^{g} z^i \lambda_{g-i}\,.$$
The series $$g(z, t)=1+\sum_{g=1}^{\infty} t^{2g} \int_{\overline \M_{g, 1}} \frac{\Lambda(-1)\Lambda(0) \Lambda(z)}{1-\psi_1}=\left( \frac{\sin (t/2)}{t/2} \right)^{-z}$$ is computed by \cite [Propositions 3 and 4]{FP2}. Differentiating with respect to $z$, we find $$\sum_{g=1}^{2g} t^{2g} \int_{\overline M_{g, 1}} \frac{c(\mathbb E^\vee)}{1-\psi_1} \lambda_g \lambda_{g-1}= \frac{\partial}{\partial z} \left( \frac{\sin (t/2)}{t/2} \right)^{-z}\bigg|_{z=0}=-\log \left(\frac{\sin (t/2)}{t/2}\right).$$ Finally, the identity $$\sum_{g=1}^{\infty} \frac{|B_{2g}|}{2g \cdot (2g)!} \cdot t^{2g}=-\log \left(\frac{\sin (t/2)}{t/2}\right)$$ is established in \cite [Lemma 3]{FP2}. Equation \eqref{integral} follows.  
\qed

\subsection{Proof of Theorem \ref{vang}} The strategy of the
proof is due to Aaron Pixton. By Theorem \ref{tortaut}, which will be proven in Section \ref{excesstheory}, the class $\Tor^*\Delta_g$ is tautological on $\M_g^{\ct}$. We wish to show that $\Tor^*\Delta_g$ is in the kernel of the pairing $$\R^{g-2}(\M_g^{\ct})\times \R^{g-1}(\M_g^{\ct})\to \R^{2g-3}(\M_g^{\ct})\cong \mathbb Q\,.$$ 
Let $j:\M_{g_1, 1}^{\ct}\times \M_{g_2, 1}^{\ct}\to \M_g^{\ct}$ be a boundary divisor, where $g_1+g_2=g$. Then \begin{equation}\label{torpull}j^*\Tor^*\Delta_g=0\,.\end{equation} Indeed, we easily see that $$j^* \lambda_{g-1}=\lambda_{g_1-1}\boxtimes \lambda_{g_2}+\lambda_{g_1}\boxtimes \lambda_{g_2-1}=0$$ using that the top Hodge class vanishes on curves of compact type. Furthermore, the morphism $\Tor\circ j$ factors as $$\M_{g_1, 1}^{\ct}\times \M_{g_2, 1}^{\ct}\stackrel{\Tor\times \Tor}{\longrightarrow} \A_{g_1}\times \A_{g_2}\stackrel{p}{\to} \A_g\,.$$ By Proposition \ref{zeroint}, $p^* [\A_1\times \A_{g-1}]=0.$ Therefore,  
$$j^*\Tor^*[\A_1\times \A_{g-1}]=0\,,$$ establishing \eqref{torpull}. 

On the other hand, the proof of Proposition \ref{tautprop} shows that \begin{equation}\label{pairtrivially}\lambda_{g-2}\cdot \Tor^*\Delta_g=0\,.\end{equation}

To finish the argument, we note that $\R^{g-2}(\M_g)$ is generated by $\lambda_{g-2}$, so all classes in $\R^{g-2}(\M_g^{\ct})$ can be written as $$c \lambda_{g-2}+\text{classes supported on the boundary}.$$ By \eqref{pairtrivially},   $\Tor^*\Delta_g$ pairs trivially with $\lambda_{g-2}$, while from \eqref{torpull}, $\Tor^*\Delta_g$ pairs trivially with all classes supported on the boundary. Thus, $\Tor^*\Delta_g$ is in the Gorenstein kernel. 
\qed

\section{Local equations for the Torelli pullback}\label{tpullbacks}
\subsection{Overview} In Sections \ref{excesstheory} and  \ref{lowgenus},  we will compute the class 
$$\Tor^*[\A_1\times \A_{g-1}]\in \CH^{g-1}(\M_g^{\ct})$$ using Fulton's intersection theory \cite {Fulton}. Consider the fiber product diagram
\[
\begin{tikzcd}
\Tor^{-1}(\A_1\times \A_{g-1}) \arrow[d] \arrow[r] & \M_g^{\ct} \arrow[d, "\Tor"] \\
\A_1\times \A_{g-1} \arrow[r]       & \ \A_g \, .                       
\end{tikzcd}
\]
The class $\Tor^*[\A_1\times \A_{g-1}]$ is the pushforward to $\M_g^{\ct}$ 
of a refined 
intersection class on the fiber product $\Tor^{-1}(\A_1\times\A_{g-1})$. 
The intersection calculation is subtle because
$\Tor^{-1}(\A_1\times\A_{g-1})$ has many excess components that meet each other.
Knowledge of the scheme structure of the fiber product $\Tor^{-1}(\A_1\times\A_{g-1})$
is required for the excess analysis.  We will find local equations for 
$\Tor^{-1}(\A_1\times \A_{g-1})$ and prove that  the scheme structure is reduced.

While we use the superscript $-1$ in the notation, the stack 
$\Tor^{-1}(\A_1\times \A_{g-1})$ is {\em not}
a substack of $\M_g^{\ct}$. This is due to the fact that the morphism 
\begin{equation}\label{hh44hh}
\A_1\times \A_{g-1} \rightarrow \A_g
\end{equation}
is {\em not} an embedding because it is not injective. However, since \eqref{hh44hh} induces an injection on tangent spaces,
$$\Tor^{-1}(\A_1\times \A_{g-1}) \rightarrow \M_g^{\ct}$$
is \'etale locally (on the domain) an embedding.

\subsection{Extremal trees and the strata of $\Tor^{-1}(\A_1\times \A_{g-1})$ }\label{ets}

The points of the fiber product $\Tor^{-1}(\A_1\times \A_{g-1})$ are simple to understand.
Using Proposition \ref{Splittingtheorem}, we have the following result.
\begin{cor}\label{components}
If $C$ is a genus $g$ curve of compact type with Jacobian isomorphic (as 
 a principally polarized abelian variety) to a 
product
$$J(C)\cong X_1\times X_{g-1}  \ \text{with}\  X_1\in \A_1 \ \text{and}\  X_{g-1}\in \A_{g-1}\, ,$$
then $C$ has an irreducible component $C_1$ of genus $1$ satisfying
$J(C_1)\cong X_1$.
\end{cor}

Corollary \ref{components} leads to a natural stratification of the fiber product $\Tor^{-1}(\A_1\times \A_{g-1})$ indexed by {\it extremal trees}. 

\begin{definition}
    Let $g\geq 2$ be an integer. An {\it extremal tree} $\mathsf{T}$ of genus $g$ is a rooted tree 
    with a genus assignment on the vertices,
    $$\mathsf{g}: \mathsf{V}(\mathsf{T}) \rightarrow \mathbb{Z}_{\geq 0}\, ,$$
    which satisfies the following properties:
    \begin{enumerate}
    \item[(i)] $\mathsf{T}$ is 
    stable{\footnote{Stability of the tree is equivalent here to
    the condition that all vertices of genus 0 have valence at least 3.}}
    with respect to $\mathsf{g}$,
        \item [(ii)] the $\mathsf{root}$ vertex has genus $1$,
        \item [(iii)] all internal vertices{\footnote{A {\em leaf} vertex is a non-root vertex of valence 1.
        An {\em internal vertex} is a vertex
        that is neither a root nor a leaf.}} of $\mathsf{T}$ have genus $0$,
        \item [(iv)] the genus condition $g= \sum_{v\in \mathsf{V}(\mathsf{T})} \mathsf{g}(v)$ holds.
    \end{enumerate}
\end{definition}
\noindent Stability (i) implies that the genus of every
leaf vertex is positive. An extremal tree $\mathsf{I}$ such that every vertex is a root or leaf is called \emph{irreducible}. 

The figure below shows an extremal tree. The root is shown as a black dot, while the leaves of genera $a, b$ are shown as gray dots. The remaining internal vertex has genus $0$. Because of the internal vertex,  the extremal tree is not irreducible. 

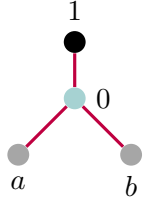
\begin{center}
\begin{tikzpicture}[scale=.5,
    mycirc/.style={circle,fill=cyan!40, minimum size=0.5cm}
    ]
    \node[circle,fill=black, scale=.8, label=above:{$1$}] (n1) at (2,0) {};
    \node[circle,fill=teal!35, scale=.8, label=right:{$0$}] (n2) at (2,-1.5) {};
    \node[circle,fill=black!35, scale=.8, label=below:{$a$}] (n3) at (0.5,-3) {};
    \node[circle,fill=black!35, scale=.8, label=below:{$b$}] (n4) at (3.5,-3) {};
    
 \draw [very thick, purple](n1) -- (n2); 
\draw [very thick, purple](n2) -- (n3);
\draw [very thick, purple](n2) -- (n4); 
\end{tikzpicture}
\end{center}

 An automorphism of an extremal tree $\mathsf{T}$ is an automorphism of the underlying tree that 
 fixes the root and respects $\mathsf{g}$.
Given an extremal tree $\mathsf{T}$ of genus $g$, we define
\[
\M^{\ct}_\mathsf{T}=\prod_{v\in \mathsf{V}(\mathsf{T})} \M^{\ct}_{\mathsf{g}(v),\mathsf{n}(v)}\,,
\]
where $\mathsf{n}(v)$ is valence of $v$.

We denote the canonical Torelli map from $\M^{\ct}_\mathsf{T}$ to $\mathcal{A}_1 \times \mathcal{A}_{g-1}$ by
$$\mathsf{Tor}^{\mathsf{T}}_{1,g-1}: \M^{\ct}_\mathsf{T} \rightarrow \mathcal{A}_1 \times \mathcal{A}_{g-1}\, ,$$
where the root of $\mathsf{T}$ corresponds to the $\mathcal{A}_1$ factor. Let
\[
\iota_{\mathsf{T}}: \M^{\ct}_\mathsf{T}\rightarrow \M^{\ct}_{g}\,
\]
be the gluing morphism associated to $\mathsf{T}$.
Since $\mathsf{Tor}^{\mathsf{T}}_{1,g-1}$ and $\iota_\mathsf{T}$ are equal after mapping to $\mathcal{A}_g$, we obtain
a canonical map
$$\epsilon_{\mathsf{T}}: \M^{\ct}_\mathsf{T}\rightarrow  \Tor^{-1}(\A_1\times \A_{g-1})\, .$$
Moreover, because $\iota_{\mathsf T}$ and the map
\[
\Tor^{-1}(\A_1\times \A_{g-1})\rightarrow \M_{g}^{\ct}
\]
are proper, so is $\epsilon_{\mathsf{T}}$.
By definition, the image of $\epsilon_{\mathsf{T}}$ is the {\em closed stratum} determined by $\mathsf{T}$. The irreducible components of $\Tor^{-1}(\A_1\times \A_{g-1})$ are the closed strata determined by irreducible extremal trees $\mathsf{I}$.

The {\em strict stratum} determined by $\mathsf{T}$ is the open subset of points of $\M^{\ct}_\mathsf{T}$
which do not lie in any closed strata for extremal trees $\mathsf{T'}$ which are nontrivial 
degenerations{\footnote{Degenerations of extremal trees will be defined 
in Section \ref{tstd} below. The definition of $\mathcal{M}^{\circ\ct}_{\mathsf{T}}$
as the complement in $\mathcal{M}^{\ct}_{\mathsf{T}}$ of the closed strata
of nontrivial degenerations will be proven there.
The definition of $\mathcal{M}^{\circ\ct}_{\mathsf{T}}$ by \eqref{n44g} is explicit. 
}} of $\mathsf{T}$. Let 
\begin{equation} \label{n44g}
\M^{\circ \ct}_\mathsf{T}=\prod_{v\in V(\mathsf{T})} \M^{\circ\ct}_{\mathsf{g}(v),\mathsf{n}(v)}\, \subset \, \M^{ \ct}_\mathsf{T} ,
\end{equation}
where we define 
\begin{itemize}
    \item $\M^{\circ\ct}_{\mathsf{g}(v),\mathsf{n}(v)}= \M_{0,\mathsf{n}(v)}$ if $v\in \mathsf{V}(\mathsf{T})$ is an internal vertex, 
    \item $\M^{\circ\ct}_{\mathsf{g}(v),\mathsf{n}(v)}= \M_{1,\mathsf{n}(v)}$ if $v\in \mathsf V({\mathsf T})$ is the root, 
\item $\M^{\circ\ct}_{\mathsf{g}(v),\mathsf{n}(v)}\subset \M^{\ct}_{\mathsf{g}(v),\mathsf{n}(v)}$
is the open locus where the marking{\footnote{For a leaf $v$, $\mathsf{n}(v)=1$.}} lies on a component of positive genus
if $v\in \mathsf{V}(\mathsf{T})$ is a leaf. \end{itemize} Then, the strict stratum determined by $\mathsf{T}$ is 
$$\epsilon_\mathsf{T}(\M^{\circ \ct}_\mathsf{T})\subset  \Tor^{-1}(\A_1\times \A_{g-1})\, .$$
For notational convenience, we will refer to the closed and strict strata of 
$\Tor^{-1}(\A_1\times \A_{g-1})$
determined by $\mathsf{T}$ by $\M_{\mathsf{T}}^{\ct}$ and $\M_{\mathsf{T}}^{\circ \ct}$
respectively.
\subsection{Irreducible components of the fiber product}\label{settheory}

We now show that the fiber product is nonsingular away from the intersections of the components. Let $\mathsf{I}$ be an irreducible extremal tree, $\M_{\mathsf{I}}^{\circ \ct}$ the associated strict stratum, and 
\[
\epsilon^{\circ}_{\mathsf{I}}:\M^{\circ \ct}_{\mathsf{I}}\rightarrow \Tor^{-1}(\A_1\times \A_{g-1})
\]
the restriction of $\epsilon_{\mathsf{I}}$ to the strict stratum.
\begin{prop} \label{strictns}
    The stack theoretic image of $\epsilon^{\circ}_{\mathsf{I}}$ is nonsingular. In particular, $\Tor^{-1}(\A_1\times \A_{g-1})$ is nonsingular away from the intersection of its components. 
\end{prop}
\begin{proof}

    The tangent space of $\Tor^{-1}(\A_1\times \A_{g-1})$ at a point $(C,(E, B))$ is the fiber product of the tangent space $\Ext^1(\Omega_C,\O_C)$ to $\M_g^{\ct}$ at $C$ with the tangent space $\Sym^2 H^0(\Omega_E)^{\vee}\oplus \Sym^2 H^0(\Omega_B)^{\vee}$ to $\A_1\times \A_{g-1}$ at $(E,B)$ over the tangent space $$\Sym^2 H^0(\Omega_E)^{\vee}\oplus (H^0(\Omega_E)^{\vee}\otimes H^0(\Omega_B)^{\vee})  \oplus \Sym^2 H^0(\Omega_B)^{\vee}$$  at $J(C)\cong E\times B$ of $\A_g$. 
    
    Let $k$ denote the number of leaves of the extremal tree $\mathsf{I}$. 
    Assume first $k=1$, for simplicity. A general point of $\epsilon_{\mathsf{I}}^{\circ}(\M^{\circ \ct}_{\mathsf{I}})$ is the Jacobian of a curve $C=E\cup D$, where $E$ is nonsingular of genus $1$ and $D$ is nonsingular of genus $g-1$. There is a tangent vector $v\in \Ext^1(\Omega_C,\O_C)$ corresponding to the smoothing of the node $p$ of $C$, hence $v\in T_pE \otimes T_pD$. 
    Under the differential of the Torelli map
    \[
    \Ext^1(\Omega_C,\O_C)\rightarrow \Sym^2 H^0(\omega_C)^{\vee}\cong \Sym^2 H^0(\Omega_E)^{\vee}\oplus (H^0(\Omega_E)^{\vee}\otimes H^0(\Omega_{D})^{\vee})  \oplus \Sym^2 H^0(\Omega_D)^{\vee}\,,
    \]
    $v$ maps to a nonzero vector in $(H^0(\Omega_E)^{\vee}\otimes H^0(\Omega_{D})^{\vee})$. Hence, $v$ does not lie in the tangent space to $\Tor^{-1}(\A_1\times \A_{g-1})$ at $(C, (E,J(D)))$, which thus has codimension at least $1$ in $\Ext^1(\Omega_C,\O_C)$. Because $\M^{\circ \ct}_{\mathsf{I}}$ is of codimension $1$ in $\M_g^{\ct}$, we see that $\epsilon_{\mathsf{I}}^{\circ}(\M^{\circ \ct}_{\mathsf{I}})$ is nonsingular at $(C,(E,J(D)))$. 
    
    Next, we suppose $C=E\cup D$, where $D=\cup_{i=0}^{n} D_i$ is a compact type curve of genus $g-1$ glued to $E$ at a exactly one point on $D_0$, where $D_0$ is nonsingular of genus $0<h<g-1$. Again, we consider the tangent vector $v$ corresponding to a family of curves smoothing the node $E\cap D_0$. Under the Torelli map, this family maps to $\A_{h+1}\times \A_{g-h-1}\subset \A_g$. 
    Therefore, we can view the codomain of the differential of the Torelli map as 
    \[
    \Sym^2 H^0(\Omega_E)^{\vee}\oplus (H^0(\Omega_E)^{\vee}\otimes H^0(\Omega_{D_0})^{\vee}) \oplus \Sym^2 H^0(\Omega_{D_0})^{\vee}\oplus \Sym^2 \left(\bigoplus_{i=1}^{n} H^0(\Omega_{D_i})^{\vee}\right)\,,
    \]
    where the first three summands correspond to the $\A_{h+1}$ factor and the latter summands correspond to the $\A_{g-h-1}$ factor.
    As above, the tangent vector $v$ has nonzero image in the $(H^0(\Omega_E)^{\vee}\otimes H^0(\Omega_{D_0})^{\vee})$ summand. 
    Hence, the vector $v$ does not lie in the tangent space to $\Tor^{-1}(\A_1\times \A_{g-1})$ at $(C, (E,J(D)))$, and the conclusion follows as in the previous paragraph. 

    The cases when $k>1$ are proved analogously by analyzing the image of the tangent vectors corresponding to smoothings of the nodes represented by edges in $\mathsf{I}$.
     \end{proof}

\subsection{Correspondence with stable maps}
To apply excess intersection theory in Section \ref{excesstheory} below, we will require local
equations for  $\Tor^{-1}(\A_1\times \A_{g-1})$. Our analysis will show  
 that the scheme structure of $\Tor^{-1}(\A_1\times \A_{g-1})$
is reduced. 

For the study of the scheme structure,
we will use a fundamental correspondence which relates $\Tor^{-1}(\A_1\times \A_{g-1})$ to a moduli space of stable maps to the universal elliptic curve. 
For the
correspondence, a marked point is required.
Consider the fiber product diagram
\begin{equation}\label{Tor1}
\begin{tikzcd}
\Tor_1^{-1}(\A_1\times \A_{g-1}) \arrow[d] \arrow[r] & \M_{g,1}^{\ct} \arrow[d, "\Tor_1"] \\
\A_1\times \A_{g-1} \arrow[r]       & \A_g \, .                       
\end{tikzcd}
    \end{equation}
The map $\Tor_1$ is defined as the composition of the forgetful map $\M_{g,1}^{\ct}\rightarrow \M_g^{\ct}$ and $\Tor$. The set-theoretic description in Section \ref{settheory} generalizes directly to $\Tor_1^{-1}(\A_1\times \A_{g-1})$: the components and their intersections correspond to extremal trees, and the additional marked point is allowed to lie on any vertex. 

Let $u:\E\rightarrow \M_{1,1}$ denote the universal elliptic curve, $s:\M_{1,1}\rightarrow \E$ the universal section, and $\M_{g,1}^{\ct}(u,1)$ the moduli space of $u$-relative stable maps of fiber degree $1$ from compact type curves of genus $g$. There is a forgetful morphism
\[
v:\M_{g,1}^{\ct}(u,1)\rightarrow \M_{g,1}^{\ct}
\]
and an evaluation morphism
\[
\ev:\M_{g,1}^{\ct}(u,1)\rightarrow \E\,.
\]
Define $Q_{g,1}$ by the fiber product diagram
\begin{equation}\label{Qg1}
\begin{tikzcd}
{Q_{g,1}} \arrow[d] \arrow[r]            & {\M_{1,1}} \arrow[d, "s"] \\
{\M_{g,1}^{\ct}(u,1)} \arrow[r, "\ev"'] & \E   \,.                    
\end{tikzcd}
\end{equation}
The space $Q_{g,1}$ is the closed substack of $\M_{g,1}^{\ct}(u,1)$ parametrizing stable maps that send the marked point to the origin in each fiber of $u$.
\begin{prop} \label{natisom}
There is a natural isomorphism
\[
F:\Tor_1^{-1}(\A_1\times \A_{g-1})\rightarrow Q_{g,1}\, .
\]
\end{prop}
\begin{proof}
    We begin by constructing the morphism $F$. There is a universal pointed curve 
    \[
    \begin{tikzcd}
\C \arrow[d]                                                      \\
\Tor_1^{-1}(\A_1\times \A_{g-1}) \arrow[u, "\sigma"', bend right]
\end{tikzcd}
    \]
pulled back from $\M_{g,1}^{\ct}$. Let $\X_g\rightarrow \A_g$ be the universal abelian variety. Using the section $\sigma$, we obtain a well-defined Abel--Jacobi map
\[
\C\rightarrow \X_g\,.
\]
The map factors through $\E\times \X_{g-1}$. Projecting to $\E$, we obtain a map $\C\rightarrow \E$ sending the section $\sigma$ to the origin in each fiber of $\E$. The construction defines the morphism $F$.

To show $F$ is an isomorphism, we construct an inverse. We have a map $Q_{g,1}\rightarrow \M_{g,1}^{\ct}$ defined by the composition $$Q_{g,1}\rightarrow \M_{g,1}^{\ct}(u,1)\rightarrow \M_{g,1}^{\ct}\, .$$ Further composing with $\Tor_{1}$ defines a map $Q_{g,1}\rightarrow \A_g$. We show that this map factors through a map $Q_{g,1}\rightarrow \A_1\times \A_{g-1}$, and thus induces a morphism $$G:Q_{g,1}\rightarrow \Tor_1^{-1}(\A_1\times \A_{g-1})\, .$$ By pulling back from $\M_{g,1}^{\ct}(u,1)$, we see that the universal curve $\C'$ over $Q_{g,1}$ admits a universal evaluation morphism $\C'\rightarrow \E$ sending the universal section of $\C'$ to the origin of $\E$ in the fibers. Taking Jacobians shows that the universal Jacobian $J(\C')$ has an elliptic curve factor. Therefore, $Q_{g,1}\rightarrow \A_g$ factors through $\A_1\times \A_{g-1}$. 
After unwinding the definitions, $F$ and $G$ are easily seen to be inverses to each other.
\end{proof}

A more general version of Proposition \ref{natisom} (showing also the
compatibility of virtual classes) has
been recently proven by
Greer and Lian \cite{GreerLian}.

\subsection{Scheme structure} \subsubsection{Reducedness}
The central result that
controls the scheme structure of the fiber product $\Tor^{-1}(\A_1\times \A_{g-1})$ is reducedness.

\begin{thm} \label{reddd}
The fiber product $\Tor^{-1}(\A_1\times \A_{g-1})$ has reduced scheme structure.
\end{thm}

Our proof of Theorem \ref{reddd} uses several special properties of the locus
$\mathcal{A}_1\times \mathcal{A}_{g-1}$ including the connection between the fiber product $\Tor^{-1}(\A_1\times \A_{g-1})$
with the moduli space of stable maps to a moving elliptic curve provided by Proposition \ref{natisom}.

Whether reducedness is special for the fiber product with
the Noether-Lefschetz locus $\A_1\times \A_{g-1}$ or a property that holds
for more general Torelli fiber products of Noether-Lefschetz loci is an interesting question.
While preliminary calculations suggest the fiber product $\Tor^{-1}(\A_2\times \A_{g-2})$
is also reduced, we do not have a proof.

\subsubsection{Strategy of proof}
The proof of Theorem \ref{reddd} will be given in several steps. Consider first the strict stratum 
$$\M_\mathsf{I}^{\circ \ct} \subset \Tor^{-1}(\A_1\times \A_{g-1})$$ 
determined by an {\em irreducible} extremal tree $\mathsf{I}$. 
In the irreducible case,  $\M_\mathsf{I}^{\circ \ct}$ is a Zariski open
set of $\Tor^{-1}(\A_1\times \A_{g-1})$.
By Proposition \ref{strictns}, 
$\Tor^{-1}(\A_1\times \A_{g-1})$ is reduced (and, in fact, nonsingular) on
the Zariski open disjoint union
\begin{equation} \label{biguu}
\coprod_{\mathsf{I} \, \text{irr}} \M_\mathsf{I}^{\circ \ct}
\subset \Tor^{-1}(\A_1\times \A_{g-1})\,. 
\end{equation}

We will prove Theorem \ref{reddd} by adding strict strata 
 $$\M_\mathsf{T}^{\circ \ct} \subset  \Tor^{-1}(\A_1\times \A_{g-1})$$
 for non-irreducible extremal trees $\mathsf{T}$
to \eqref{biguu}
one at a time until all of $\Tor^{-1}(\A_1\times \A_{g-1})$ is covered.
At each stage, we must ensure that we have a Zariski open
set of  $\Tor^{-1}(\A_1\times \A_{g-1})$ and that the
scheme structure on the Zariski open is reduced.

\subsubsection{$\mathsf{T}$-structures and degenerations} \label{tstd}
To define the order of addition of $\M_\mathsf{T}^{\circ \ct}$ to \eqref{biguu},
we introduce 
$\mathsf{T}$-structures and 
degenerations
of extremal trees.
\begin{definition}\label{defdef}
Let $\mathsf{T}$ be an extremal
tree of genus $g$ with $\ell$ vertices,
and let 
$\mathsf{T}'$ be an extremal tree of genus $g$. 
A $\mathsf{T}$-structure on $\mathsf{T}'$
is given by a set
partition{\footnote{Every element of
a set partition here is required to be non-empty.}}
of the vertex set of $\mathsf{T}'$,
$$ \mathcal{V}=\{\mathcal{V}_1, \ldots, \mathcal{V}_\ell\}\, , \ \ \
\mathcal{V}_1\cup  \ldots \cup 
\mathcal{V}_\ell =\mathsf{V}({\mathsf{T}'})\, ,$$
together with a bijection
$$\phi: \mathsf{V}(\mathsf{T}) \rightarrow
\{1,\ldots,\ell\}$$
satisfying the following properties:
\begin{enumerate}
\item[(i)] The bijection
$\phi$ respects the root structure, 
$\mathsf{root}\in \mathcal{V}_{\phi(\mathsf{root})}$.
\item[(ii)] 
For all $v\in \mathsf{V}(\mathsf{T})$,
the vertex subset
$\mathcal{V}_{\phi(v)}\subset \mathsf{V}(\mathsf{T}')$
determines a connected subtree of
$\mathsf{T}'$ with
$$\mathsf{g}(v)= \sum_{v'\in
\mathcal{V}_{\phi(v)}} \mathsf{g}(v')\, .
$$
\item[(iii)] An edge $e\in \mathsf{E}(\mathsf{T})$ connects the vertices
$v,w\in \mathsf{V}(\mathsf{T})$ if and only
if there exists an edge $e'\in \mathsf{E}(\mathsf{T}')$ which connects a vertex
of $\mathcal{V}_{\phi(v)}$ to a vertex
of $\mathcal{V}_{\phi(w)}$.
\end{enumerate}

\end{definition}

For an extremal tree $\mathsf{T}'$ to carry a $\mathsf{T}$-structure, we must have
\begin{equation}\label{dfft}
|\mathsf{V}(\mathsf{T})| \leq |\mathsf{V}(\mathsf{T}')|\, .
\end{equation}
Moreover, if equality holds for \eqref{dfft},
then a $\mathsf{T}$-structure on $\mathsf{T}'$ is equivalent to an
isomorphism of 
$\mathsf{T}$ and $\mathsf{T}'$ as
extremal trees.
We define $\mathsf{T}'$ to be a {\em nontrivial degeneration} of $\mathsf{T}$ 
if
$\mathsf{T}'$ carries a nontrivial $\mathsf{T}$-structure. We denote
nontrivial degenerations by
$$\mathsf{T} \rightsquigarrow \mathsf{T}'\, .$$
We also refer to $\mathsf{T}$ as a \emph{smoothing} of $\mathsf{T}'$.

 \begin{lem} 
The strict stratum $\mathcal{M}^{\circ\ct}_{\mathsf{T}}$
is the complement in $\mathcal{M}^{\ct}_{\mathsf{T}}$ of the union of
closed strata
corresponding to nontrivial degenerations 
$\mathsf{T}'$
of $\mathsf{T}$,
$$\mathcal{M}^{\circ\ct}_{\mathsf{T}} =  \mathcal{M}^{\ct}_{\mathsf{T}}\  \setminus \
\bigcup_{\ \mathsf{T} \rightsquigarrow \mathsf{T}'} \mathcal{M}^{\ct}_{\mathsf{T}'}\, .
$$
\end{lem}

\begin{proof} From the definitions.
\end{proof}

 A chain of nontrivial
 degenerations of {\em length $d$} is a sequence of extremal trees
of genus $g$
$$\mathsf{T}_0 \rightsquigarrow \mathsf{T}_1\rightsquigarrow \ldots
\rightsquigarrow \mathsf{T}_d $$
where $\mathsf{T}_{i+1}$ is
nontrivial degeneration of $\mathsf{T}_i$ 
for $0\leq i \leq d-1$.

\begin{definition}
An extremal tree $\mathsf{T}$  of genus $g$ has {\em depth} $d$
    if the maximal chain of
    nontrivial degenerations of
    extremal trees ending with $\mathsf{T}$ has length $d$.
\end{definition}

 For example, an irreducible tree $\mathsf I$ admits no nontrivial degenerations. Hence, the depth of
 $\mathsf{I}$ is 0 in the irreducible case. On the other hand, the tree below has depth $1$. 
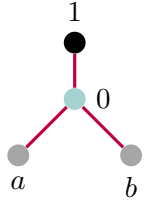
\begin{center}
\begin{tikzpicture}[scale=.5,
    mycirc/.style={circle,fill=cyan!40, minimum size=0.5cm}
    ]
    \node[circle,fill=black, scale=.8, label=above:{$1$}] (n1) at (2,0) {};
    \node[circle,fill=teal!35, scale=.8, label=right:{$0$}] (n2) at (2,-1.5) {};
    \node[circle,fill=black!35, scale=.8, label=below:{$a$}] (n3) at (0.5,-3) {};
    \node[circle,fill=black!35, scale=.8, label=below:{$b$}] (n4) at (3.5,-3) {};
    
 \draw [very thick, purple](n1) -- (n2); 
\draw [very thick, purple](n2) -- (n3);
\draw [very thick, purple](n2) -- (n4); 
\end{tikzpicture}
\end{center}

We will add the strict strata 
 $\M_\mathsf{T}^{\circ \ct}$
to \eqref{biguu} in order of increasing depth. 
We start with all the depth 0 extremal trees to obtain \eqref{biguu}.
We then add all the  strict 
strata corresponding to the trees of depth exactly 1. Next,  
we add all the strict strata corresponding to the extremal trees of depth exactly 2, and so on.
The  resulting subsets, indexed by depth, are
$$\coprod_{\mathsf{I} \, \mathrm{irr}} \M_\mathsf{I}^{\circ \ct}=U_0 \subset U_1 \subset U_2 \subset \ldots \subset 
\Tor^{-1}(\A_1\times \A_{g-1})\, .
$$

\begin{lem}
The subsets $U_i \subset \Tor^{-1}(\A_1\times \A_{g-1})$
constructed by the increasing depth procedure are Zariski open 
and cover  $\Tor^{-1}(\A_1\times \A_{g-1})$ after finitely many steps.
\end{lem}

\begin{proof} From the definitions.
\end{proof}

\subsubsection{Induction step: set up} \label{indsetup}
By Proposition \ref{strictns}, $U_0
 \subset \Tor^{-1}(\A_1\times \A_{g-1})$ is 
a reduced open set. Let $d\geq 1$ and 
assume that $$U_{d-1} \subset \Tor^{-1}(\A_1\times \A_{g-1})$$ is 
a reduced open set.
We will show then that 
$$U_{d} \subset \Tor^{-1}(\A_1\times \A_{g-1})$$ is 
also a reduced open set.

Let $\mathsf{T}$ be an extremal tree of genus $g$ and of depth exactly $d$. Let 
\begin{equation}
\label{b33l}
\pi: (C,p) \rightarrow (E,0)
\end{equation}
be a stable map with $[\pi] \in Q_{g,1}$. Such a stable map has a unique irreducible component $\widehat{E}$ of the domain which maps isomorphically to the target
$$\pi|_{\widehat{E}}: \widehat{E} \cong
E\, .$$
Our first assumption is:
\begin{enumerate}
    \item[(i)] the marking $p$ lies on $\widehat{E}$ (and is
mapped to $0$ under $\pi$ by the definition of $Q_{g,1}$).
\end{enumerate}
Via Proposition \ref{natisom}, $[\pi]\in Q_{g,1}$ corresponds to the point 
$$F^{-1}([\pi]) \in \Tor^{-1}_1(\A_1\times \A_{g-1})\, . $$
After forgetting the
marking $p$, we obtain a point $$A_\pi \in \Tor^{-1}(\A_1\times \A_{g-1})\,.$$
Our second assumption is:
\begin{enumerate}
    \item[(ii)] $A_\pi$ is an element of  $\M_{\mathsf{T}}^{\circ\ct}$  .
\end{enumerate}

The construction can be reversed. Given $A_\pi \in \M_{\mathsf{T}}^{\circ\ct}$, we can find
a stable map \eqref{b33l} satisfying conditions (i) and (ii). By the
isomorphism of Proposition \ref{natisom} and 
the smoothness of
$$\Tor_1^{-1}(\A_1\times \A_{g-1}) \rightarrow \Tor^{-1}(\A_1\times \A_{g-1})$$
at $F^{-1}([\pi]) \in \Tor^{-1}_1(\A_1\times \A_{g-1})$,
we can study the scheme structure of $Q_{g,1}$ near $[\pi]$ to prove
the reducedness of $\Tor_1^{-1}(\A_1\times \A_{g-1})$ near $A_\pi$.

\subsubsection{Local equations for the reduced scheme structure}
Our next goal is to find local equations for the reduced scheme structure $$Q_{g,1}^{\mathsf{red}}\subset 
Q_{g,1}$$ 
near
the point $[\pi: (C,p) \rightarrow (E,0)]$ satisfying conditions (i) and (ii) of Section \ref{indsetup}. 
Since 
$$\Tor^{-1}(\A_1\times \A_{g-1}) \rightarrow \A_g$$
is an immersion,
$Q_{g,1}$ is a closed subscheme of $\M_{g,1}^{\ct}$
locally at 
$[\pi]\in Q_{g,1}$ in the analytic topology.
In particular, there exists an analytic open set $W$,
$$[\pi]\in W\subset \M_{g,1}^{\ct}\, ,$$
such that $Q_{g,1}$ is, locally at $[\pi]$, cut out
by equations in $W$.
We may take $W$ to be
a versal deformation space of 
$[\pi]$. We have a map from $W$ to the deformation spaces of the nodes of $C$, 
$$\mu:W \rightarrow \prod_{e\in \mathsf{E}(\mathsf{T})} \mathbb{C}_e\,,$$
where $\mathbb{C}_e$ is the 1-dimensional versal deformation
space of the node of $C$ corresponding to the edge $e$ of $\mathsf{T}$.

Let $x_e$ be the standard coordinate on $\mathbb{C}_e$.
Let $v\in \mathsf{V}(\mathsf{T})$ be a leaf. To $v$, we associate
a monomial $\mathsf{Mon}(v)$ in the variables $\{x_e\}_{e\in \mathsf{E}(\mathsf{T})}$
by the following equation:
$$\mathsf{Mon}(v) = \prod_{e\, \in\,  \mathsf{path}(v)} x_e\, ,$$
where the product is over all edges $e\in \mathsf{E}(\mathsf{T})$ that lie on the minimal 
path from the leaf $v$ to the root of $\mathsf{T}$.

\begin{prop} \label{red999} The reduced subscheme $Q^{\mathsf{red}}_{g,1}$ is defined in $W$ by the pullback
from $\prod_{e\in \mathsf{E}(\mathsf{T})} \mathbb{C}_e$
of
the 
monomial set
$$\{ \, \mathsf{Mon}(v) \, | \,  \text{\em $v$ is a leaf of $\mathsf{T}$}\, \} \subset \mathbb{C}\big[\{x_e\}_{e\in \mathsf{E}(\mathsf{T})}\big]\, .$$
    \end{prop}

    \begin{proof} The vanishing of the monomial set $\mathsf{Mon}(v)$ defines a reduced scheme, locally cut out by union of linear subspaces.   Every monomial ideal with  generators given by products of distinct variables is reduced (as can be proven by induction on the number of variables). Since $\mu$ is formally smooth,
reducedness still holds after pullback. 
\end{proof}

\subsubsection{Induction step: deformation theory}
Let $\mathsf{T}$ be an extremal tree of genus $g$ and of depth exactly $d$. Let 
\begin{equation}
\label{b33lf}
\pi: (C,p) \rightarrow (E,0)
\end{equation}
be stable map with $[\pi] \in Q_{g,1}$ satisfying conditions (i) and (ii) of Section \ref{indsetup}.

Near points of $Q_{g,1}\cap W$ not in the strict stratum associated to $\mathsf{T}$, 
$Q_{g,1}$ is reduced by the induction hypothesis (via Proposition \ref{natisom} and the smoothness of the point choice).
To complete the induction step, we need only check that the pullbacks via $\mu$ of the monomials
\begin{equation} \label{ffvvb}
\{ \, \mathsf{Mon}(v) \, | \,  \text{$v$ is a leaf of $\mathsf{T}$}\, \} 
\end{equation}
vanish on all flat deformations of the stable map \eqref{b33lf}
over Artinian bases. Since these monomials generate the reduced structure by Proposition \ref{red999},
we then conclude that $Q_{g,1}\cap W$ is reduced. Therefore,
$$U_d\subset \Tor^{-1}(\A_1\times \A_{g-1})$$
is also reduced.

Consider a flat deformation of the stable map \eqref{b33lf} over $(A, \mathfrak m)$, where $\mathfrak m$ is the maximal ideal in a local Artin ring $A$. We have a diagram 
\begin{equation}\label{xxyy}\begin{tikzcd}
(\mathcal {C},\mathcal{P}) \arrow[d] \arrow[r, "\widetilde{\pi}"] & (\mathcal{E},0) \arrow[d] \\
\mathsf{Spec}(A)   \arrow[r, "\sim"]     & \mathsf{Spec}(A) \, .                       
\end{tikzcd}
\end{equation} Here, $\mathcal P$ is a section of $\mathcal C\to \mathsf{Spec }(A).$
Such a deformation maps to $W$ and then (via $\mu$) to $\prod_{e\in \mathsf{E}(\mathsf{T})} \mathbb{C}_e$,
so we can pull back the monomials \eqref{ffvvb}. 

The first simplification is that 
we can assume the deformation of the target $(\mathcal{E},0) \rightarrow  \mathsf{Spec}(A)$ is trivial.
The moduli of stable maps $Q_{g,1}$ has locally trivial structure over the moduli space
of elliptic targets
$$Q_{g,1} \rightarrow \M^{\ct}_{1,1}\, .$$
Locally analytically near  $[\pi] \in Q_{g,1}$ the moduli space
of stable maps is isomorphic to the moduli space of maps to a fixed elliptic target,
$$Q^E_{g,1} = \mathrm{ev}^{-1}_p(0) \subset \M^{\ct}_{g,1}(E,1)\, ,$$
times an open set of $\mathbb{C}$. Moreover, the pullbacks of the monomials 
\eqref{ffvvb}
factor through the
projection to $Q^E_{g,1}$. Therefore,
we can restrict our attention to
the simpler deformation:
\begin{equation}
\label{vvnnh}
\begin{tikzcd}
(\mathcal {C},\mathcal{P}) \arrow[d] \arrow[r, "\widetilde{\pi}"] & (E,0)  \\
\mathsf{Spec}(A)\,.         &  \,                        
\end{tikzcd}
\end{equation}
In addition, we may assume that the $\mathsf{root}$ of $\mathsf T$ has valence $1$, since in the argument below we can treat the connected components of the curve $C\smallsetminus E$ one at a time. 

Let $v$ be an arbitrary leaf of $\mathsf{T}$. We will show that the monomial $\mathsf{Mon}(v)$ 
vanishes when pulled back to $\mathsf{Spec}(A)$ via the family \eqref{vvnnh}.
Let 
\begin{equation} \label{vll4} v - v_1 - v_2 - \ldots - v_k - \mathsf{root}\end{equation}
be the minimal path from the leaf $v$ to the root of $\mathsf{T}$, and let
$$D-P_1 - P_2 - \ldots - P_k - E $$
be the corresponding closed subcurves of $C$.
The curve $D$ is of compact type, but may be reducible. By definition, 
$$[D,q] \in \M_{\mathsf{g}(v),1}^{\circ \ct}\, $$
where $q$ is the point where $D$ meets $P_1$, and 
the intermediate subcurves $P_1,\ldots, P_k$
are all isomorphic to $\mathbb{P}^1$ by assumption (ii) of Section \ref{indsetup}. 

Let $s\in E$ denote the nodal point corresponding to the intersection of $E$ with $P_k$. In the target $E$, we choose a local parameter $z\in \mathcal O_{E,s}$ which we represent by a regular function  
$$z:\Delta\to \mathbb C$$ in a neighborhood $s\in \Delta\subset E$. The function $$F=\tilde \pi^{*} z$$ is regular on the open subcurve $\tilde \pi^{-1}(\Delta)\subset \mathcal C$. Let $\mathcal E^-$ be the open subcurve of $\mathcal C$ obtained by removing {\footnote{Since
$A$ is Artinian, the Zariski topologies of $C$ and $\mathcal {C}$
are the same.}} all the components of $C$ other than $E$. Then $\mathcal E^-$ is a flat deformation over $\mathsf{Spec }(A)$ of the smooth affine curve $E^-= E\smallsetminus \{s\}$. Such a deformation is necessarily trivial by \cite[Theorem 1.2.4]{Ser}. 
Consequently, the regular functions on $\mathcal E^-$ are of the form $A\otimes_{\mathbb C} \mathcal O(E^{-})$. The restriction of $F$ to $\mathcal E^-\cap \tilde \pi^{-1}(\Delta)$ is the function $1\otimes z$. 

We will construct a different function $G$ on a subcurve of $\mathcal C$ (the domain will be specified below), which agrees with $1\otimes z$ on $\mathcal E^-\cap \tilde \pi^{-1}(\Delta)$. The strategy is then to compare $F$ and $G$ on the common domain and use the comparison to show the vanishing of the monomial $\mathsf{Mon} (v)$. 

To specify the domain of $G$, we require a few preliminary constructions.  
Let $\mathcal L$ denote the set of leaves
of the tree $\mathsf{T}$.
Each leaf in $\mathsf{T}$ determines a positive genus subcurve of $C$, not necessarily irreducible. Removing from $\mathcal C$ the closed subcurves corresponding to the leaves in $\mathcal{L}$ yields an open curve $\mathcal C^{-}$. 
The function $G$ will be defined on ${\mathcal C}^{-}\cap \widetilde \pi^{-1}(\Delta)$.

Before constructing $G$ explicitly, we need to single out two more curves. 
\begin{itemize} 
\item The subcurve
$\mathcal {C}_v^{\circ}\subset \mathcal {C}$ is obtained
by removing from $\mathcal {C}$ the Zariski closed set
consisting of
all components {\em not contained} on the path \eqref{vll4}. 


\item We define $\mathcal C_v^-=\mathcal C_v^{\circ}\smallsetminus D$. 
\end{itemize} 
Clearly, $\mathcal C_v^-$ is a subcurve of $\mathcal C^{-}$. In general, the difference between these two curves is due to the internal nodes of $\mathsf T$ lying on minimal paths from the leaves in $\mathcal L$ to the root which are {\em not} on the path \eqref{vll4}. 

For example, consider the following extremal tree:

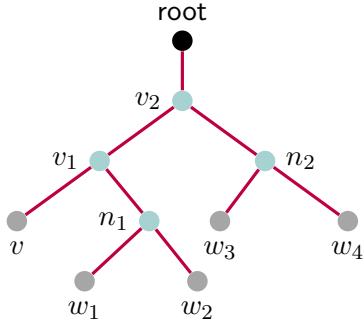
\begin{center}
\begin{tikzpicture}[scale=.4]
\node[circle,fill=black, scale=.75, label=above:{$\textsf{root}$}] (n1) at (11-8,0) {};
    \node[circle,fill=teal!35, scale=.75, label=left:{$v_2$}] (n2) at (3,-2) {};
 \node[circle,fill=teal!35, scale=.75, label=left:{$v_1$}] (n3) at (0.25,-4) {};
  \node[circle,fill=teal!35, scale=.75, label=right:{$n_2$}] (n4) at (5.75,-4) {};
   \node[circle,fill=black!35, scale=.75, label=below:{$v$}] (n5) at (-2.5,-6) {};
    \node[circle,fill=teal!35, scale=.75, label=left:{$n_1$}] (n6) at (1.9,-6) {};
    \node[circle,fill=black!35, scale=.75, label=below:{$w_1$}] (n10) at (-0.25,-8) {};
    \node[circle,fill=black!35, scale=.75, label=below:{$w_2$}] (n9) at (3.5,-8) {};
    \node[circle,fill=black!35, scale=.75, label=below:{$w_3$}] (n7) at (4.25,-6) {};
     \node[circle,fill=black!35, scale=.75, label=below:{$w_4$}] (n8) at (8.5,-6) {};
 \draw [very thick, purple](n1) -- (n2); 
   \draw [very thick, purple](n2) -- (n3);
   \draw [very thick, purple](n2) -- (n4);
\draw [very thick, purple](n3) -- (n5);
\draw [very thick, purple](n3) -- (n6);
\draw[very thick, purple] (n4) -- (n7);
\draw[very thick, purple] (n4) -- (n8);
\draw[very thick, purple] (n6) -- (n9);
\draw[very thick, purple] (n6) -- (n10);
\end{tikzpicture}
\end{center}
We have $\mathcal L=\{v, w_1, w_2, w_3, w_4\}.$ Furthermore, 
\begin{itemize} 
\item the curve $\mathcal C^{-}$
is obtained by removing from $\mathcal C$ the components corresponding to the leaves $v, w_1, w_2, w_3, w_4$, 
\item the curve $\mathcal C_v^{\circ}$ is obtained by
removing from $\mathcal C$ the components corresponding to the vertices $w_1, w_2, w_3, w_4, n_1, n_2$,
\item the curve $\mathcal C_v^-$ is obtained by removing from $\mathcal C_v^{\circ}$ the component corresponding to $v$. 
\end{itemize}
The difference between the two curves $\mathcal C^{-}$ and $\mathcal C_v^-$ is due to the nodes $n_1$ and $n_2$.

For general $\mathsf {T}$, the map 
$\mathcal {C}^-_v \rightarrow \mathsf{Spec}(A)$
is a flat deformation.  
The fiber $C_v^-$ over the closed point of $\mathsf{Spec(A)}$ is the chain of punctured rational curves
$$P_1^- -P_2^--\ldots - P_k^--E\,.$$ The punctures in $P_j^-$ correspond to the removal of various components. For the $i^{th}$ rational curve $P_i$, we fix standard coordinates:
$$[x_i:1]\,,\quad [1:y_i]\,, \quad x_i=\frac{1}{y_i}\, .$$
Our conventions are: 
\begin{itemize}
   \item $P_k$ is attached to $E$ at   $[0:1]$,
\item $[1:0]\in P_j$ is identified with $[0:1]\in P_{j-1}$, for all $j$,
\item the curve $D$ is attached to the point $[1:0]\in P_1$.
\end{itemize}
For simplicity, assume first that via the versal deformation space,
the equations of the deformation of $C^{-}_v$ are given by the deformations
at the nodes
\begin{equation} \label{dd44}
zx_k-a_k\, ,\ \  y_kx_{k-1}-a_{k-1}\, ,\ \ y_{k-1}x_{k-2} - a_{k-2}\, , \ \ \ldots \, , \ \ 
y_2x_1 - a_1\,, 
\end{equation}
where $a_k,a_{k-1},a_{k-2}, \ldots, a_1 \in A$.
The general case will be considered shortly. 

We define $G$ on $\mathcal E^- \cap \tilde \pi^{-1}(\Delta)$ as $1\otimes z$. We first extend
$G$ to ${\mathcal C}^-_v\cap \tilde \pi^{-1}(\Delta)$ as follows: 
\begin{enumerate}
\item[$\bullet$] Using the first equation $zx_k-a_k=0$ of \eqref{dd44}, on $P_k$ we set 
$$G = \frac{a_k}{x_k} = a_k y_k\,.$$
\item[$\bullet$] Using the second equation $y_kx_{k-1}-a_{k-1}=0$ of \eqref{dd44}, on $P_{k-1}$ we set
$$G= \frac{a_ka_{k-1}}{x_{k-1}} = a_k a_{k-1} y_{k-1}\,.$$
\item[$\bullet$] By repeatedly applying the equations
 of \eqref{dd44}, we see that
$G$ can be extended
to $\mathcal {C}_v^-\cap \tilde \pi^{-1}(\Delta)$.
\end{enumerate}
In fact, a stronger statement can be made. Let $q$ denote the node on $P_1$ corresponding to the attaching point of $D$, and suppose that the deformation of the node $q$ corresponds to the local equation $$y_1u=a_0, \quad a_0\in \mathfrak m\,,$$ where $u$ is a local coordinate on $D$ and $\mathfrak{m}$ is the maximal ideal of $A$. Then the extension $G$ satisfies $$G=\frac{a_k\cdots a_0}{u}$$ in an analytic neighborhood of $q\in D$. 

The deformation \eqref{dd44} may not be the most general. In fact, the general deformation is given by \begin{equation} \label{dd45}
zx_k-a_k\, ,\ \  f_{k-1}(y_k)x_{k-1}-a_{k-1}\, ,\ \ f_{k-2}(y_{k-1})x_{k-2} - a_{k-2}\, , \ \ \ldots \, , \ \ 
f_1(y_2)x_1 - a_1\,, 
\end{equation}
where $a_k,a_{k-1},a_{k-2}, \ldots, a_1 \in \mathfrak m$. Furthermore, $f_1, \ldots, f_{k-1}$ are formal changes of coordinates centered at the origin. We may assume $f_j(0)=0$ and $f_j'(0)=1$, after normalization of the $a$'s. In this case, the extension can be constructed as follows. On $P_k$, no changes are necessary, and $G= {a_k}{y_k}$ is still valid. We consider the inverse change of coordinate $g_{k-1}$ such that $$y=g_{k-1}(f_{k-1}(y))\, , \quad g_{k-1}(0)=0\, , \quad g'_{k-1}(0)=1\, .$$ On $P_{k-1}$, we set $$G={a_k}g_{k-1}\left(\frac{a_{k-1}}{x_{k-1}}\right)=a_k g_{k-1}(a_{k-1} y_{k-1})=a_k a_{k-1} y_{k-1} +\ldots\, ,$$ where the higher order terms contain coefficients divisible by $a_k a_{k-1}^2$. Since the maximal ideal $\mathfrak m$ is nilpotent, $a_{k-1}\in \mathfrak m$ is nilpotent as well, and the last expression consists only in finitely many terms. We can continue in the same fashion over the remaining components $P_{k-2}, \ldots, P_1$, and then to an analytic neighborhood of the node $q$ in $D$. Near $q$, we then obtain \begin{equation}\label{formg}G=\frac{a_k\cdots a_0}{u}+\text{finitely many higher order terms in } \frac{1}{u}\,.\end{equation} The coefficients of the higher order terms necessarily belong to the ideal spanned by $a_k\cdots a_0 \mathfrak m$. The extra factor of $\mathfrak m$ comes from the fact that the higher powers of $\frac{1}{u}$ contribute extra $a$'s, and all $a_i\in \mathfrak m$. We will use these facts in Lemma \ref{annoyingyetnice} below.  

The above procedure defines $G$ over $\mathcal C_v^-\cap \tilde \pi^{-1}(\Delta)$. The curve ${\mathcal C}^{-}$ contains other rational components. These correspond to internal vertices lying on minimal paths that join a leaf $w$ in $\mathcal L\smallsetminus\{v\}$ to one of the vertices $v_j$. (The case we just did corresponds to the leaf $v$ in $\mathcal L$.) For those components, the argument is similar: we can extend along genus $0$ components with the aid of the equations of the nodes. Since $\mathsf{T}$ possesses no cycles, the extension is a well-defined regular function $G$ on ${\mathcal C}^{-}\cap \tilde \pi^{-1}(\Delta)$.

Let $h=F-G$. Since $F, G$ are both regular on ${\mathcal C}^{-}\cap \tilde \pi^{-1}(\Delta)$, so is $h$. Studying $h$ will be crucial for the proof of the following result. 

\begin{lem}\label{annoyingyetnice}
For the family \eqref{xxyy}, the pullback
from the versal deformation space of the nodes
of $\mathsf{Mon}(w)$ vanishes on
$\mathsf{Spec}(A)$ for all leaves $w$ in $\mathcal {L}.$
\end{lem}

Lemma \ref{annoyingyetnice} is exactly the vanishing of $\textsf{Mon}(v)$ we claimed, and completes the proof of the induction
step and of Theorem \ref{reddd}. The main additional point is that the vanishing of $\textsf{Mon}(v)$ has to be proven simultaneously with the vanishings coming from all leaves in $\mathcal L$.  

\begin{proof} We present a
detailed argument in a
representative case. For the general case, no new ideas are needed, but the
notation is more complicated.  

For simplicity, we assume $v$ and $\mathsf{root}$ are separated by a single genus $0$ vertex $v_1$ and $v_1$ is incident
to only one other leaf $w$. We have $\mathcal L=\{v, w\}$, and we seek to prove the vanishing of the monomials $\mathsf{Mon}(v)$ and $\mathsf{Mon}(w)$. 

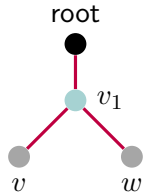
\begin{center}
\begin{tikzpicture}[scale=.5,
    mycirc/.style={circle,fill=cyan!40, minimum size=0.5cm}
    ]
    \node[circle,fill=black, scale=.8, label=above:{$\mathsf{root}$}] (n1) at (2,0) {};
    \node[circle,fill=teal!35, scale=.8, label=right:{$v_1$}] (n2) at (2,-1.5) {};
    \node[circle,fill=black!35, scale=.8, label=below:{$v$}] (n3) at (0.5,-3) {};
    \node[circle,fill=black!35, scale=.8, label=below:{$w$}] (n4) at (3.5,-3) {};
    
 \draw [very thick, purple](n1) -- (n2); 
\draw [very thick, purple](n2) -- (n3);
\draw [very thick, purple](n2) -- (n4); 
\end{tikzpicture}
\end{center}
The curve $C$ here has the following components:
\begin{itemize}
    \item a genus $1$ component $E$ corresponding to the root,
\item a genus zero component $P\simeq \mathbb P^1$ corresponding to $v_1$ with coordinates $x$, $y=1/x$ on $P$, 
\item a positive genus curve $D_v$ corresponding to the leaf $v$ attached to $[1: 0]\in P$ at the node $q$
\item a positive genus curve $D_w$ corresponding to the leaf $w$ attached to the point $[r: 1]\in P$. \end{itemize} 
The curves of compact type $D_v$ and $D_w$ may not be irreducible. However, the nodes corresponding to the intersections with $P$ do not lie on genus $0$ components of $D_v$ and $D_w$ by assumption (ii) of Section \ref{indsetup}. 

Assume first that the local equations of the nodes take the form $$zx=a\, , \quad yu=b\, , \quad (x-r)t=c\,$$ where $a, b, c\in \mathfrak m$, and $u, t$ are local coordinates on $D_v,D_w$ near the respective nodes. We seek to show $$ab=0\, , \quad ac=0\,.$$ 
The function $h$ is regular on $${\mathcal C}^{-}\cap \tilde \pi^{-1}(\Delta)=\widetilde \pi^{-1}(\Delta) \smallsetminus (D_v\cup D_w)\, .$$
In particular, $h$ is regular on $\mathcal P^-$, the open set above $P^-=P\smallsetminus \{[r:1], [1:0]\}$. Since $P^-$ is affine and nonsingular, the deformation $\mathcal P^-$ is trivial. Thus, the regular functions on $\mathcal P^-$ are of the form $A\otimes_{\mathbb C} \mathcal O(P^-)$. 
We can write \begin{equation}\label{hh}h=\alpha+\sum_{i=1}^M s_i \otimes \frac{1}{y^i}+\sum_{i=1}^M  t_i\otimes \frac{1}{(x-r)^i}\,,\end{equation} for $\alpha\in A$, $s_i\in A$, and $t_i\in A$. Let $\mathfrak s$ and $\mathfrak t$ be the two ideals in $A$ spanned by the $s_i$ and $t_i$ respectively. By definition, $F=h+G$ is regular over $\widetilde \pi^{-1}(\Delta)$. 

We inspect next the curves $D_v$ and $D_w$, which we assume for now to be irreducible. Let $D^-_v, D^-_w$ be the two affine curves obtained from $D_v$, $D_w$ by removing the nodes corresponding to the intersections with $P$. The induced deformations $\mathcal D_v^-, \mathcal D_w^-$ of $D_v^-$, $D_w^-$ are trivial. Recall that $G=ab\otimes \frac{1}{u}$ over $\mathcal D_v^-$. Furthermore, over the trivial deformation $\mathcal D_v^-$ we have $$F=\sum_{i=1}^{d} a_i\otimes f_i$$ where $a_i\in A$ and $f_i$ are regular on $D_v^-$. Therefore, we can write locally \begin{equation}\label{hanother}h=\sum_{i=1}^{d} a_i\otimes f_i-ab\otimes \frac{1}{u}=\sum_{i=1}^{d} a_i\otimes \left(\sum_{j} f_{ij} u^j\right)-ab\otimes \frac{1}{u}\,,\end{equation} where $f_{ij} \in \mathbb C$. On the other hand, we examine expression \eqref{hh}. Expanding near the node $q$ (with coordinate $y=0$): $$\frac{1}{x-r}=\sum_{i\geq 0} c_i y^i\,,$$ we obtain{\footnote{Let $x,y\in A$. By $x/y\in A$, we mean an element of $A$ satisfying the property $y\cdot (x/y) =  x \in A$.}} 
$$h=\alpha+\sum_{i>0} \frac{s_i}{b^i}\otimes u^i+\sum_{i>0} \tilde t_i b^i \otimes \frac{1}{u^i}\,,$$ where $\tilde t_i\in \mathfrak t$ are combinations of the $t_i$'s and $c_i$'s. The second sum only requires finitely many terms since $b\in \mathfrak m$ is nilpotent. Comparing with \eqref{hanother}, we conclude 
$s_i/b^i\in A$ for all $i>0$. 
Now, we analyze the expression $h$ over $\textsf{Spec} (A/\mathfrak t)$. Reducing mod $\mathfrak t$, we obtain $$h=\alpha+\sum_{i>0} \frac{s_i}{b^i}\otimes u^i$$ over $D_v^-\times \textsf {Spec}(A/\mathfrak t)$. On the other hand, $G=ab\otimes \frac{1}{u}$ and we can write the reduction \begin{equation}\label{keyred}F\mod \mathfrak t=\sum_{i=1}^{N} \tilde{a_i}\otimes \tilde {f_i}\end{equation} where $\tilde a_i$'s are a basis for the Artinian ring $A/\mathfrak t$ as a $\mathbb C$-vector space, and $\tilde f_i$ are regular{\footnote{We are {\em not} claiming that $a_i$, $f_i$ project to $\tilde a_i$, $\tilde {f_i}$ under $A\to A/\mathfrak t$.}} 
on $D_v^-$. Therefore, $$\sum_{i=1}^{N} \tilde{a_i}\otimes \tilde {f_i}=ab\otimes \frac{1}{u}+ \alpha+\sum_{i>0} \frac{s_i}{b^i}\otimes u^i\,.$$ Expressing the images of $\alpha,$ $ab,$ $\frac{s_i}{b^i}$ under the map $A\to A/\mathfrak t$ in terms of the basis $\tilde {a_i}$, we see that one of the $\tilde{f_i}$ on $D_v^-$ must admit at worst a simple pole at $u=0$ (and only there). No such function exists on a positive genus curve. In fact, all functions with at worst simple pole only at $u=0$ are constant, and hence their expansion contains no nonnegative powers of $u$. We thus obtain \begin{equation}\label{keyred1}ab=0 \in   A/\mathfrak t \quad \text{and } \quad s_i/b^i=0 \text{ in }A/\mathfrak t\,.\end{equation} Therefore, for $i>0$  we have $$s_i/b^i\in \mathfrak t\implies s_i\in b^i\mathfrak t\subset \mathfrak m \mathfrak t\,.$$ We conclude $$ab\in \mathfrak t\, , \quad \mathfrak s\subset \mathfrak m\mathfrak t\,.$$ The parallel analysis for $D_w^-$ yields $$ac\in \mathfrak s\, , \quad \mathfrak t\subset \mathfrak m \mathfrak s\,.$$ 
Let $\mathfrak i=\mathfrak s+\mathfrak t$. 
The above conclusions show that $$ab\in \mathfrak i\, , \quad ac\in \mathfrak i\, , \quad \mathfrak i\subset \mathfrak m \mathfrak i\,.$$ Since $\mathfrak m$ is nilpotent in $A$, we find $\mathfrak i=0$ and hence $ab=ac=0$, as required.


To address the most general deformation, consider the equations of the nodes of the form $$zx=a\, , \quad f(y)u=b\, , \quad g(x-r)t=c$$ where $f, g$ are normalized changes of coordinates with $f(0)=g(0)=0, f'(0)=g'(0)=1$. Let $\tilde y=f(y)$, so that $$\tilde y u=b\, .$$ Since $\tilde y$ is a local coordinate near the node $q$, we have $$\frac{1}{y}=\frac{1}{\tilde y}+\sum_{i\geq 0} \epsilon_i \tilde y^i,\quad \epsilon_i\in \mathbb C\,.$$ Similarly, we can expand near $q$: $$\frac{1}{x-r}=\frac{y}{1-yr}=\sum_{i\geq 0} \tau_i \tilde y^i, \quad \tau_i\in \mathbb C\,.$$ Using $\tilde yu=b$, and substituting into the expression \eqref{hh}, we can write $$h=\alpha' + \sum_{i=1}^M \frac{s_i'}{b^i}\otimes u^i+ \sum_{i>0} \tilde {s_i} b^i \otimes \frac{1}{u^i} + \sum_{i> 0} \tilde{t_i} b^i \otimes \frac{1}{u^i}$$ on $\mathcal D_v^-$. Since $b$ is nilpotent, all sums are finite. As before $s_i'/b^i\in A$ for $1\leq i\leq M$. It is not hard to write down the expressions for the new coefficients $s_i'\in A$. In fact, for $1\leq i\leq M$, we find $$s_i'=s_i+\text{terms involving } s_{i+1}, \ldots, s_M \text{ with coefficients that depend on } \epsilon\,.$$ Thus, we have\begin{equation}\label{ide}\mathfrak s=\langle s_1', \ldots, s_M'\rangle\,,\end{equation} where as before $\mathfrak s$, $\mathfrak t$ are the two ideals generated by $s_i$ and $t_i$. 
Furthermore, $$\widetilde {s_i}\in \mathfrak s, \quad \tilde{t_i}\in \mathfrak t\,.$$ We also have by \eqref{formg}: $$G=ab\otimes \frac{1}{u}+ \text {higher order terms in }\frac{1}{u} \text{ with coefficients in }ab\cdot \mathfrak m\,.$$ 

The next step is to reduce modulo the ideal ${\mathfrak a}=\mathfrak m(\mathfrak s+\mathfrak t+\langle ab\rangle).$ This reduction kills many terms in $h$ (using $b\in \mathfrak m$) and $G$: $$h\mod \mathfrak a=\alpha'+ \sum_{i=1}^M \frac{s_i'}{b^i}\otimes u^i, \quad G\mod \mathfrak a=ab\otimes \frac{1}{u}\,.$$ Writing $$F=\sum_{i=1}^{N} \widetilde{a_i} \otimes \widetilde{f_i}$$ over $\textsf{Spec} (A/{\mathfrak a})\times D_v^-$, with $\tilde a_i$ giving a basis for $A/\mathfrak a$, we find that one of the functions $\widetilde{f_i}$ has at worst simple pole only at $u=0$. As before, this implies $$ab=0 \mod \mathfrak a, \quad s_i'/b^i=0\mod \mathfrak a\,.$$ Consequently, $ab\in \mathfrak a$. Moreover, $s_i'\in \mathfrak a$, and hence by \eqref{ide} we have $\mathfrak s\subset \mathfrak a.$ Therefore, we established $$\langle ab\rangle +\mathfrak s\subset \mathfrak m(\mathfrak s+\mathfrak t+\langle ab\rangle)\,.$$ A similar argument shows $$\langle ac\rangle+\mathfrak t\subset \mathfrak m(\mathfrak s+\mathfrak t+\langle ac\rangle)\,.$$ Let $\mathfrak i=\mathfrak s+\mathfrak t+\langle ab,ac\rangle$. Adding the two inclusions above gives $$\mathfrak i\subset \mathfrak m\mathfrak i\,.$$ Since $\mathfrak m$ is nilpotent, it follows $\mathfrak i=0$, hence $ab=0,$ $ac=0$. 

When the curves $D_v$ or $D_w$ are not irreducible, the argument is parallel. In the irreducible case, a key step in the argument is \eqref{keyred1}. This relied upon the fact that there are no nonconstant functions over nonsingular projective curves of positive genus possessing at worst one simple pole. The same is true over curves of compact type $(X, q)$, provided $q\in X$ is a nonsingular  point of an irreducible component $T$ of positive genus: \begin{equation}\label{keyred2}H^0(X, \mathcal O_X(q))=\mathbb C\,.\end{equation} To see this, let $T_1, \ldots, T_\ell$ denote the connected components of the closure of $X\smallsetminus T$ in $X$. Let $q_1, \ldots, q_\ell$ denote the corresponding nodes. The claim follows from the exact sequence $$0\to \bigoplus_{i=1}^{\ell} \mathcal O_{T_i}(-q_i) \to \mathcal O_{X}(q)\to \mathcal O_T(q)\to 0\, .$$ 

Let us give more details on how the proof is completed from here. By cohomology and base change, we first promote \eqref{keyred2} to the following family version. Write for simplicity $\mathcal Y=\textsf{Spec}(A)$. Assume $\pi: \Xb\to \Yb$ is a flat proper family with a section $\mathcal Q:\Yb \to \Xb$, such that basechanging to $A/\mathfrak m\simeq \mathbb C$, the pair $(X, q)$ is a pointed curve of compact type satisfying the above conditions. Then \begin{equation}\label{keyred3}\pi_{*}(\mathcal O_{\mathcal X}(\mathcal Q))={\mathcal O}_{\Yb}\, .\end{equation} The argument is standard. We first form the commutative diagram \[\begin{tikzcd}\mathcal O_{\Yb}\otimes \mathbb C_y \ar[r] \ar[d, equal]&\pi_{*}\mathcal O_{\Xb}\otimes \mathbb C_y\ar[r]\ar[d] &H^0(X, \mathcal O_X)=\mathbb C_y\ar[d, "\simeq"] \\ 
\mathcal O_{\Yb}\otimes \mathbb C_y\ar[r] & \pi_{*}(\mathcal O_{\Xb}(\mathcal Q))\otimes \mathbb C_y\ar[r] & H^0(X, \mathcal O_X(q))=\mathbb C_y\, .\end{tikzcd}\]
The composition of the arrows on the second row is surjective (because the same is true for the top row). By cohomology and base change, it follows that the second map on the second row is in fact an isomorphism, and furthermore $\pi_{*}(\mathcal O_{\Xb}(\mathcal Q))$ is locally free of rank $1$. Thus, the first map on the second row is surjective $$\mathcal O_{\Yb}\otimes \mathbb C_y \to \pi_{*}(\mathcal O_{\Xb} (\mathcal Q))\otimes \mathbb C_y\, .$$ By Nakayama's Lemma, this implies $$\mathcal O_{\Yb}\to \pi_{*} (\mathcal O_{\Xb}(\mathcal Q))$$ is a surjective morphism of vector bundles of the same rank, hence an isomorphism. 

Returning to the original proof, let us assume $D_v$ is reducible, and let $T$ be the irreducible component intersecting the genus zero curve $P$ at the node $q$. Let $T^-$ be the smooth affine curve obtained by removing from $T$ all nodes, and $\mathcal T^-\to \textsf{Spec}(A)$ be the deformation obtained by restricting the flat family $\mathcal C\to \textsf{Spec} (A)$. The deformation $\mathcal T^-$ is necessarily trivial. On the other hand, removing from $\mathcal C$ the components $E$ and $P$, we obtain a flat curve $\mathcal Z\to \mathsf{Spec}(A)$. We glue $\mathcal Z$ to the trivial deformation of $T$ over $\textsf{Spec}(A)$ along $\mathcal T^-$, yielding a flat curve $\mathcal X\to \mathsf{Spec}(A)$ with a section $\mathcal Q$ corresponding to the node $q$. Now, keeping the same notation as in the proof of \eqref{keyred1}, the function $F$ has the property that in a neighborhood of $q$, we have $$F=\frac{ab}{u}+\text{positive powers of }u \mod \mathsf t\, .$$ Thus, $F$ is a section of $\mathcal O_{\Xb}(\mathcal Q)$, of course after basechanging to $\textsf{Spec}(A/\mathsf t)$. Therefore, by \eqref{keyred3}, we have $F\in A/\mathfrak t$. We thus obtain assertion \eqref{keyred1}, and the proof is completed as before. 
\end{proof}

\subsection{Local equations}
As a consequence of the proof of Theorem \ref{reddd},
we have constructed canonical equations for 
$\Tor^{-1}(\A_1\times \A_{g-1})$ 
at every point (expressed in 
the
versal deformation space of the corresponding
compact type curve). Since we will 
require these equations for
the excess intersection calculation, we
record the result as follows.

\begin{prop}\label{localequations} The local equations of $\Tor^{-1}(\A_1\times \A_{g-1})$ near a point in the strict stratum indexed by $\mathsf{T}$ are given by the pullback
from $\prod_{e\in \mathsf{E}(\mathsf{T})} \mathbb{C}_e$
of
the 
monomial set
$$\{ \, \mathsf{Mon}(v) \, | \,  \text{\em $v$ is a leaf of $\mathsf{T}$}\, \} \subset \mathbb{C}\big[\{x_e\}_{e\in \mathsf{E}(\mathsf{T})}\big]\, .$$
\end{prop}

\section{Excess intersection theory}\label{excesstheory}
\subsection{Overview}
We have the fiber product diagram
\[
\begin{tikzcd}
\Tor^{-1}(\A_1\times \A_{g-1}) \arrow[d] \arrow[r] & \M_g^{\ct} \arrow[d, "\Tor"] \\
\A_1\times \A_{g-1} \arrow[r]       & \ \A_g \, .                       
\end{tikzcd}
\]
By Fulton's intersection theory \cite{Fulton}, the class $\Tor^*[\A_1\times \A_{g-1}]$ is the pushforward to $\M_g^{\ct}$ 
of a refined 
intersection class on the fiber product $\Tor^{-1}(\A_1\times\A_{g-1})$. We give an inductive method to compute the refined class based on the local equations of Section \ref{tpullbacks} for the strata of $\Tor^{-1}(\A_1\times\A_{g-1})$. We illustrate the method with several examples that will be used later to prove Theorems \ref{tortaut}, \ref{vang}, and \ref{Delta6}. 

\subsection{Inductive method for the excess calculation}\label{inductivesection}


The fiber product $\Tor^{-1}(\A_1\times \A_{g-1})$ is stratified
with strata
indexed by extremal trees of genus $g$. The partial ordering on the strata 
corresponds to smoothing
of the extremal trees: 
an extremal tree $\mathsf{T}'$ is a smoothing of an extremal tree $\mathsf{T}$ if $\mathsf{T}$ has a nontrivial $\mathsf{T}'$ structure, see Section \ref{tstd}.

By repeated application of the excision sequence, $\Tor^*[\A_1\times \A_{g-1}]$ can be expressed as a
sum of contributions $\mathsf{Cont}_\mathsf{T}$ supported on $\M^{\ct}_\mathsf{T}$ for each extremal tree $\mathsf{T}$ of genus $g$.  Because the degree of $\Tor^*[\A_1\times \A_g]$ is $g-1$, only extremal trees with at most $g-1$ edges contribute: if $|\mathsf{E}(\mathsf{T})|\geq g$, then $\mathsf{Cont}_\mathsf{T}=0$. The contributions will be computed inductively. The base cases for the induction are the extremal trees that admit no smoothings. These are the irreducible extremal trees, which correspond to the irreducible components of $\Tor^{-1}(\A_1\times\A_{g-1})$.


The formula for the contributions is in terms of the Chern classes of the normal bundle to $\A_1\times \A_{g-1}\subset \A_g$ and the Chern classes of the normal bundles of the substacks in the stratification of $\Tor^{-1}(\A_1\times \A_{g-1})$ by extremal trees. These contributions can be found using excess residual intersections as in \cite[Chapter IX]{Fulton}. When one of the components is divisorial and the residual scheme is a regular embedding, \cite[Corollary 9.2.1]{Fulton} gives a formula for the residual contribution in terms of the Chern classes of the normal bundles (of the residual scheme and its intersection with the divisorial part). The arbitrary case is reduced to this situation using suitable blowups and is treated in \cite[Corollary 9.2.3]{Fulton}. Crucially for us, the exact residual contributions are {\it universal expressions} depending only on the normal bundle data. We can therefore compute these contributions in a suitable local model. The local equations in Proposition \ref{localequations} will be used for the local model
calculations.

Let $\mathsf{T}$ be an extremal tree with $n$ edges and $k$ leaves. The local model near the
stratum 
$\mathcal{M}_\mathsf{T}^{\circ \mathsf{ct}}$
of $\Tor^{-1}(\A_1\times \A_{g-1})$
is constructed as follows. We start with torus equivariant space $\mathbb{C}^n$. The coordinates on $\mathbb{C}^n$ are placed in bijection with edges $e$ of $\mathsf{T}$, and so we label the coordinates by $\{z_e\}_{e\in \mathsf{E}(\mathsf{T})}$. The variable $z_e$ corresponds to the weight of the torus in the local model and to the normal bundles of the smoothing of the node corresponding to the edge $e$ in the moduli of curves.

Let $v\in \mathsf{V}(\mathsf{T})$ be a leaf and $\mathsf{path}(v)\subset \mathsf{E}(\mathsf{T})$ be the set of edges on the minimal path from $v$ to the root of $\mathsf{T}$. We set 
\[
\mathsf{Mon}(v)=\prod_{e\in \mathsf{path}(v)} z_e\,.
\]
Let $\N$ be a rank $g-1$ vector bundle on $\cc^n$ of the form
\[
\N=\O^{\oplus k}\oplus L_1\oplus \dots \oplus L_{g-1-k}\,,
\]
where the $L_i$ are arbitrary torus equivariant line bundles. Consider the section 
\[
s=(\mathsf{Mon}(v_1),\dots,\mathsf{Mon}(v_k),0,\dots,0)\in H^0(\N)\,,
\]
where $v_1,\dots,v_k$ are the leaves of $\mathsf{T}$. 
The local model for the excess intersection
geometry of  
$\Tor^*[\A_1\times \A_g]$ near 
the stratum 
$\mathcal{M}_\mathsf{T}^{\circ \mathsf{ct}}$
is the excess calculation of 
$c_{g-1}(\mathcal{N})$ determined by the zero locus of $s$. 

In the local model, we have
\[
c_{g-1}(\N)=(\ell_1\cdot\ldots\cdot \ell_{g-1-k})\prod_{i=1}^k\big(\sum_{e\in \mathsf{path}(v_i)} z_e\, \big)\,,
\]
where the $\ell_i$ are the equivariant Chern classes of $L_i$.

\begin{itemize}
    \item [(i)] First consider the case where $\mathsf{T}$ is an irreducible extremal tree. Then, $n=k$. The contribution $\mathsf{Cont}_\mathsf{T}$ can be computed by the usual excess intersection formula:
    \begin{equation}\label{compcont}
    \mathsf{Cont}_\mathsf{T}=\left[\frac{c(\N)}{\prod_{e\in \mathsf E(\mathsf T)} \,(1+z_e)\,}\right]_{g-1-k}\,.
    \end{equation} The subscript indicates that only the part of degree $g-1-k$ is considered. 
    The pushforward to the ambient torus equivariant $\mathbb{C}^{k}$ is computed by multiplying by the top Chern class of the normal bundle which equals $z_1\cdot\ldots\cdot z_k$. Thus
    \[
    \iota_{\mathsf{T}*}\mathsf{Cont}_{\mathsf{T}}=z_1\cdot\ldots\cdot z_k\left[\frac{c(\N)}{\prod_{e\in \mathsf E(\mathsf T)} \,(1+z_e)\,}\right]_{g-1-k}\,.
    \]
    \item [(ii)] Next,  let $\mathsf{T}$ be an arbitrary extremal tree. By induction, we can assume we have computed $\mathsf{Cont}_{\mathsf{T}'}$ for all smoothings $\mathsf{T}'$ of $\mathsf{T}$. \footnote{The contributions $\mathsf {Cont}_{\mathsf T'}$ depend on the variables $\{z_e\}$, for $e\in \mathsf E(\mathsf T')\subset \mathsf E(\mathsf T)$. The latter inclusion holds since each edge of $\mathsf T'$ corresponds to a unique edge of $\mathsf{T}$ thanks to Definition \ref{defdef} (iii).} 
    We set
    \begin{equation}\label{inductive}
    \mathsf{Cont}_{\mathsf{T}}\cdot\prod_{e\in \mathsf{E}(\mathsf{T})} z_e=c_{g-1}(\N)-\sum_{\mathsf{T}'}\iota_{\mathsf{T}'*}\mathsf{Cont}_{\mathsf{T}'}\,.
    \end{equation}
    Solving equation \eqref{inductive} gives a formula for $\mathsf{Cont}_{\mathsf{T}}$.   
    \end{itemize}
    
    The expression for $\mathsf{Cont}_{\mathsf T}$ thus obtained depends on the variables $z_e$ and $\ell_i$. 
    Since $\mathsf{Cont}_{\mathsf T}$ is symmetric
    in the $\ell_i$,
    we can write
$$\mathsf{Cont}_{\mathsf{T}}= \mathsf{P}_{\mathsf{T}}(Z,\mathcal{N})\, ,$$
where 
$\mathsf{P}_{\mathsf{T}}$
is a uniquely determined polynomial in the
variables
$Z=\{z_e\}_{e\in \mathsf{E}(\mathsf{T})}$
and the Chern classes of $\mathcal{N}$.
    
The formula for
$\mathsf{Cont}_{\mathsf T}$
in terms of tautological classes is then obtained via substitution of
variables:
\begin{itemize}
    \item[$\bullet$] we
replace each edge variable $z_e$ by the
normal factor corresponding
to the smoothing of the edge $e$ (the sum of tangent
lines corresponding to the
two half-edges of $e$),
 \item[$\bullet$] we replace 
the Chern classes of $\mathcal{N}$
by the Chern classes of the normal bundle of the 
immersion
$$\A_1\times \A_{g-1}\to \A_g\, .$$
\end{itemize}
In the end, $\mathsf{Cont}_{\mathsf T}$ is expressed in terms
of tautological $\psi$ and $\lambda$ classes obtained from the
moduli of curves.

\subsection{Excess contributions of the irreducible components.} We continue to work with the fiber diagram 
\[
\begin{tikzcd}
\Tor^{-1}(\A_1\times \A_{g-1}) \arrow[d] \arrow[r] & \M_{g}^{\ct} \arrow[d, "\Tor"] \\
\A_1\times \A_{g-1} \arrow[r]       & \A_g \, .                       
\end{tikzcd}
\]
Recall from Section \ref{ets} that the irreducible components of the fiber product are indexed by irreducible extremal trees $\mathsf{I}$, 
$$\Tor^{-1}(\A_1\times \A_g)=\bigcup_{\mathsf I} \M_{\mathsf I}^{\ct}\,.$$ We let $k$ denote the number of leaves in $\mathsf{I},$ and we let $g_1, \ldots, g_k$ denote the genus assignment for each of the $k$ leaves, so that $$g_1+\ldots+g_k=g-1\,.$$ Thus 
$\M_{\mathsf I}^{\ct}$ is covered by the product $$\M_{1, k}^{\ct}\times \M_{g_1, 1}^{\ct} \times \ldots \times \M_{g_k, 1}^{\ct}\,.$$  
The irreducible component $\M_{\mathsf{I}}^{\ct}$ has codimension $k$ in $\M_g^{\ct}.$ On the other hand, the expected codimension of $\Tor^{-1}(\A_1\times \A_{g-1})$ in $\M_g^{\ct}$ is $g-1$. Thus, the only component $\M_{\mathsf I}^{\ct}$ of the expected codimension corresponds to the tree with $g-1$ leaves attached to the root and genus distribution $(1, \ldots, 1)$.   

The excess contribution of the locus $\M_{\mathsf I}^{\ct}$ is given by \eqref{compcont}. By \eqref{normalbun}, the normal bundle of the immersion $\A_1\times \A_{g-1}\to \A_g$ is $$\Sym^2(\mathbb E_1^{\vee}\boxplus \mathbb E_{g-1}^{\vee})-\Sym^2 \mathbb E_1^{\vee}-\Sym^2 \mathbb E_{g-1}^{\vee}=\mathbb E_1^{\vee}\boxtimes \mathbb E_{g-1}^{\vee}\, .$$ When pulled back to $\M_{\mathsf I}^{\ct}$, the normal bundle splits as \begin{equation}\label{normspl}\mathbb E_1^{\vee}\boxtimes (\mathbb E_{g_1}^{\vee}\boxplus \ldots \boxplus \mathbb E^{\vee}_{g_k})\,.\end{equation} The normal bundle of $\M_{\mathsf I}^{\ct}$ in $\M_g^{\ct}$ is the sum of contributions corresponding to the smoothings of each of the $k$ nodes of the curve. Therefore, the excess contribution for $\M_{\mathsf I}^{\ct}$ equals \begin{equation}\label{excess}\left[\frac{\prod_{i=1}^{k}c(\mathbb E^{\vee}_{g_i})}{\prod_{e\in \mathsf{E}(\mathsf{I})} (1-\psi_e'-\psi''_e)}\right]_{g-1-k}\,,\end{equation} where $\psi'_e, \psi''_e$ are the cotangent classes at the node associated to $e$. The Hodge bundle over $\mathbb{E}_1\rightarrow\A_1$ does not enter the expression since $c_1(\mathbb{E}_1)$ vanishes. The subscript indicates that only the part of degree $g-1-k$ is considered. 

\subsection{Examples} We work out a few explicit examples that will play a role in Section \ref{lowgenus}. The examples can all be calculated by hand. For the reader's convenience, we provide code for the computations in  \cite{COP}.

\begin{example}
\label{answer3} 

Consider the following extremal tree $\mathsf{T}$
with $4$ vertices: the root shown as a black dot, an internal vertex of genus $0$, and two leaves of genera $a, b$. 
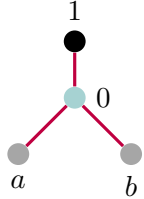
\begin{center}
\begin{tikzpicture}[scale=.5,
    mycirc/.style={circle,fill=cyan!40, minimum size=0.5cm}
    ]
    \node[circle,fill=black, scale=.8, label=above:{$1$}] (n1) at (2,0) {};
    \node[circle,fill=teal!35, scale=.8, label=right:{$0$}] (n2) at (2,-1.5) {};
    \node[circle,fill=black!35, scale=.8, label=below:{$a$}] (n3) at (0.5,-3) {};
    \node[circle,fill=black!35, scale=.8, label=below:{$b$}] (n4) at (3.5,-3) {};
    
 \draw [very thick, purple](n1) --  (n2); 
\draw [very thick, purple](n2) -- (n3);
\draw [very thick, purple](n2) -- (n4); 
\end{tikzpicture}
\end{center}
The extremal tree $\mathsf{T}$ has two nontrivial smoothings, $\mathsf{R}$ and $\mathsf{S}$. 
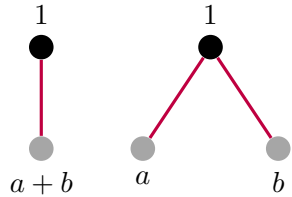
\begin{center}
    \begin{tikzpicture}[scale=.45,
    mycirc/.style={circle,fill=cyan!40, minimum size=0.5cm}
    ]
    \node[circle,fill=black, scale=.9, label=above:{$1$}] (n1) at (0,0) {};
    \node[circle,fill=black!35, scale=.9, label=below:{$a+b$}] (n2) at (0,-3) {};
 \draw [very thick, purple](n1) -- (n2); 
    
     \node[circle,fill=black, scale=.9, label=above:{$1$}] (n3) at (5,0) {};
    \node[circle,fill=black!35, scale=.9, label=below:{$a$}] (n4) at (3, -3) {};
    \node[circle,fill=black!35, scale=.9, label=below:{$b$}] (n5) at (7, -3) {};
\draw [very thick, purple](n3) -- (n4);
\draw [very thick, purple] (n3) -- (n5);
\end{tikzpicture}
\end{center}
The extremal trees $\mathsf R$, $\mathsf S$ have no further smoothings, and their contributions can be computed using equation \eqref{compcont} for the irreducible case:
\begin{align*}
    \mathsf{Cont}_{\mathsf{R}}&=\left[\frac{c(\N)}{1+ z_1}\right]_{a+b-1}\, ,\quad
\mathsf{Cont}_{\mathsf{S}}=\left[\frac{c(\N)}{(1+ z_2)(1+z_3)}\right]_{a+b-2}\,.
\end{align*}
 Here, we label the edge  of $\mathsf T$ incident to  $\mathsf{root}$ by $z_1$, while the remaining edges are labelled by $z_2, z_3$.  
From equation \eqref{inductive}, we have
\[
\mathsf{Cont}_{\mathsf{T}}\cdot z_1z_2z_3=c_{a+b}(\N)-z_1\left[\frac{c(\N)}{1+ z_1}\right]_{a+b-1}-z_2z_3\left[\frac{c(\N)}{(1+ z_2)(1+z_3)}\right]_{a+b-2}\,.
\]
For later use, we explicitly record a few special cases. 
\begin{itemize}
    \item [(i)] Assume first $g-1=a+b=3$. 
Then, the total Chern class of $\N$ is
\[
c(\N)=(1+z_1+z_2)(1+z_1+z_3)(1+\ell_1)\,. 
\]
Expanding the two power series and dividing through by $z_1z_2z_3$, we obtain
$$\mathsf{Cont}_{\mathsf T}=-3\,.$$
\item [(ii)] Asume now that $a+b=g-1=4.$
Then, the total Chern class of $\N$ is
\[
c(\N)=(1+z_1+z_2)(1+z_1+z_3)(1+\ell_1)(1+\ell_2)\,. 
\]
Expanding the two power series and dividing through by $z_1z_2z_3$, we obtain
\[
\mathsf{Cont}_{\mathsf{T}}=z_2+z_3-3\ell_1-3\ell_2=-3c_1(\N)+6z_1+4z_2+4z_3\,.
\]
\item [(iii)]
In the same scenario as above, except with $g-1=a+b=5$, we write 
$$c(\mathcal N)=(1+z_1+z_2)(1+z_1+z_3)(1+\ell_1)(1+\ell_2)(1+\ell_3)\,.$$ 
Solving the recursion, we find
\begin{align*}
\mathsf{Cont}_{\mathsf{T}}
=-3c_2(\mathcal N)+c_1(\mathcal N)\cdot (6z_1+4z_2+4z_3)-10z_1^2-10 z_1\cdot (z_2+z_3)-5(z_2+z_3)^2+5z_2z_3\,.
\end{align*}
\end{itemize}
\end{example}  

\begin{example}\label{answer4} Next, consider the extremal tree $\mathsf{T}$ shown below.
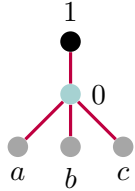
\begin{center}
\begin{tikzpicture}[scale=.35,
    mycirc/.style={circle,fill=cyan!40, minimum size=0.5cm}
    ]
   \node[circle,fill=black, scale=.75, label=above:{$1$}] (m1) at (15+8,0) {};
    \node[circle,fill=teal!35, scale=.75, label=right:{$0$}] (m2) at (15+8,-2) {};
 \node[circle,fill=black!35, scale=.75, label=below:{$a$}] (m3) at (13+8,-4) {};
  \node[circle,fill=black!35, scale=.75, label=below:{$b$}] (m4) at (15+8,-4) {};
   \node[circle,fill=black!35, scale=.75, label=below:{$c$}] (m5) at (17+8,-4) {};
 \draw [very thick, purple](m1) -- (m2); 
   \draw [very thick, purple](m2) -- (m3);
   \draw [very thick, purple](m2) -- (m4);
\draw [very thick, purple](m2) -- (m5);
\end{tikzpicture}
\end{center}
There are two nontrivial smoothings, $\mathsf{R}$ and $\mathsf{S}$. 
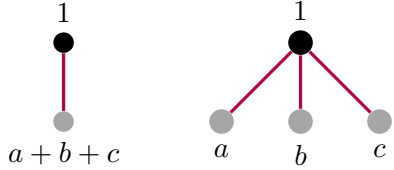
\begin{center}
\begin{tikzpicture}[scale=.35,
    mycirc/.style={circle,fill=cyan!40, minimum size=0.5cm}
    ]
    \node[circle,fill=black, scale=.75, label=above:{$1$}] (n1) at (0,0) {};
    \node[circle,fill=black!35, scale=.75, label=below:{$a+b+c$}] (n2) at (0,-3) {};
 \draw [very thick, purple](n1) -- (n2); 
      \node[circle,fill=black,scale=.9,  label=above:{$1$}] (n7) at (9,0) {};
    \node[circle,fill=black!35, scale=.9, label=below:{$a$}] (n8) at (6, -3) {};
    \node[circle,fill=black!35, scale=.9, label=below:{$b$}] (n9) at (9, -3) {};
        \node[circle,fill=black!35, scale=.9, label=below:{$c$}] (n10) at (12, -3) {};
\draw  [very thick, purple](n7) -- (n8);
\draw  [very thick, purple](n7) -- (n9);
\draw  [very thick, purple](n7) -- (n10); 
\end{tikzpicture}
\end{center}
We label the edge incident to the root by $z_1$, while the remaining edges are labelled by $z_2, z_3, z_4$. The contributions of $\mathsf{R}$ and $\mathsf{S}$ are obtained from \eqref{compcont}:
\begin{align*}
    \mathsf{Cont}_{\mathsf{R}}&=\left[\frac{c(\N)}{1+ z_1}\right]_{a+b+c-1}\, ,\\
    \mathsf{Cont}_{\mathsf{S}}&=\left[\frac{c(\N)}{(1+ z_2)(1+z_3)(1+z_4)}\right]_{a+b+c-3}\,.
\end{align*}

From equation \eqref{inductive}, we have
\[
\mathsf{Cont}_{\mathsf{T}}\cdot z_1z_2z_3z_4=c_{a+b+c}(\N)-z_1\left[\frac{c(\N)}{1+ z_1}\right]_{a+b+c-1}-z_2z_3z_4\left[\frac{c(\N)}{(1+z_2)(1+z_3)(1+z_4)}\right]_{a+b+c-3}\,.
\]
\begin{itemize}
\item [(i)] When $a+b+c=g-1=4$, we have $$c(\mathcal N)=(1+z_1+z_2)(1+z_1+z_3)(1+z_1+z_4)(1+\ell_1)\,.$$ Solving for $\mathsf{Cont}_{\mathsf{T}}$, we obtain $$\mathsf{Cont}_{\mathsf T}=-4\,.$$
\item [(ii)] Assuming $a+b+c=g-1=5,$ we have $$c(\mathcal N)=(1+z_1+z_2)(1+z_1+z_3)(1+z_1+z_4)(1+\ell_1)(1+\ell_2)\,.$$ 
Solving for $\mathsf{Cont}_{\mathsf{T}}$, we obtain
$$\mathsf{Cont}_{\mathsf{T}}= -4\ell_1 - 4\ell_2 - 2z_1 + z_2 + z_3 + z_4=-4c_1(\mathcal N)+10z_1+5(z_2+z_3+z_4)\,.$$ 
\end{itemize}

\end{example}
\vskip.05in
\begin{example}\label{answer5} Consider the extremal tree $\mathsf T$
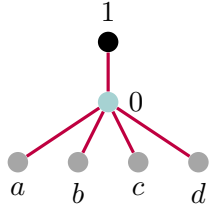
\begin{center}
\begin{tikzpicture}[scale=.4]
\node[circle,fill=black, scale=.75, label=above:{$1$}] (p1) at (3,0) {};
    \node[circle,fill=teal!35, scale=.75, label=right:{$0$}] (p2) at (3,-2) {};
 \node[circle,fill=black!35, scale=.75, label=below:{$a$}] (p3) at (0,-4) {};
  \node[circle,fill=black!35, scale=.75, label=below:{$b$}] (p4) at (2,-4) {};
   \node[circle,fill=black!35, scale=.75, label=below:{$c$}] (p5) at (4,-4) {};
    \node[circle,fill=black!35, scale=.75, label=below:{$d$}] (p6) at (6,-4) {};
 \draw [very thick, purple](p1) -- (p2); 
   \draw [very thick, purple](p2) -- (p3);
   \draw [very thick, purple](p2) -- (p4);
\draw [very thick, purple](p2) -- (p5);
\draw [very thick, purple](p2) -- (p6);
\end{tikzpicture}
\end{center}
with smoothings 
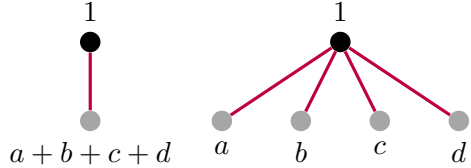
\begin{center}
\begin{tikzpicture}[scale=.35,
    mycirc/.style={circle,fill=cyan!40, minimum size=0.5cm}
    ]
    \node[circle,fill=black, scale=.75, label=above:{$1$}] (n1) at (0,0) {};
    \node[circle,fill=black!35, scale=.75, label=below:{$a+b+c+d$}] (n2) at (0,-3) {};
 \draw [very thick, purple](n1) -- (n2); 

 \node[circle,fill=black, scale=.75, label=above:{$1$}] (m1) at (39.5-30,0) {};
    \node[circle,fill=black!35, scale=.75, label=below:{$a$}] (m2) at (35-30, -3) {};
    \node[circle,fill=black!35, scale=.75, label=below:{$b$}] (m3) at (38-30, -3) {};
        \node[circle,fill=black!35, scale=.75, label=below:{$c$}] (m4) at (41-30, -3) {};
        \node[circle,fill=black!35,scale=.75,  label=below:{$d$}] (m5) at (44-30, -3) {};
\draw  [very thick, purple](m1) -- (m2);
\draw  [very thick, purple](m1) -- (m3);
\draw  [very thick, purple](m1) -- (m4); 
\draw  [very thick, purple](m1) -- (m5); 

 \end{tikzpicture}
 \end{center}
We label the edge incident to the root by $z_1$, while the remaining edges are labelled by $z_2, z_3, z_4, z_5$. We find $$\mathsf{Cont}_{\mathsf T}=-5$$ when $a+b+c+d=g-1=5.$ The contribution is computed from the recursion $$\textsf{Cont}_{\mathsf T}\cdot z_1z_2z_3z_4z_5=c_5(\mathcal N)-z_1\left[\frac{c_1(\mathcal N)}{1+z_1}\right]_{4}-z_2z_3z_4z_5\left[\frac{c(\mathcal N)}{(1+z_2)(1+z_3)(1+z_4)(1+z_5)}\right]_{1}\,,$$ where $$c(\mathcal N)=(1+z_1+z_2)(1+z_1+z_3)(1+z_1+z_4)(1+z_1+z_5)(1+\ell_1)\, .$$
\end{example}
\begin{rem} In general, for an extremal tree $\mathsf T$ with a single genus $0$ vertex attached to the root, and with adjacent leaves of genera $g_1, \ldots, g_k$ with $g_1+\ldots+g_k=k+1=g-1,$ the solution of the above recursion yields $$\mathsf {Cont}_{\mathsf{T}}=-(k+1)\,, $$ which is consistent with Examples \ref{answer3}(i), \ref{answer4}(i) and \ref{answer5}.
\end{rem}

\begin{example}\label{weird}
Next, we consider the more complicated extremal tree $\mathsf{T}$ shown below. 
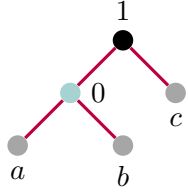
\begin{center}
\begin{tikzpicture}[scale=.35,
    mycirc/.style={circle,fill=cyan!40, minimum size=0.5cm}
    ]
\node[circle,fill=black, scale=.75, label=above:{$1$}] (p1) at (4,0) {};
    \node[circle,fill=teal!35, scale=.75, label=right:{$0$}] (p2) at (2,-2) {};
 \node[circle,fill=black!35, scale=.75, label=below:{$c$}] (p3) at (6,-2) {};
  \node[circle,fill=black!35, scale=.75, label=below:{$a$}] (p4) at (0,-4) {};
   \node[circle,fill=black!35, scale=.75, label=below:{$b$}] (p5) at (4,-4) {};
 \draw [very thick, purple](p1) -- (p2); 
   \draw [very thick, purple](p1) --++(1,-1) node {}-- (p3);
   \draw [very thick, purple](p2) -- (p4);
\draw [very thick, purple](p2) -- (p5);
\end{tikzpicture}
\end{center}
There are two nontrivial smoothings, $\mathsf{R}$ and $\mathsf{S}$. 
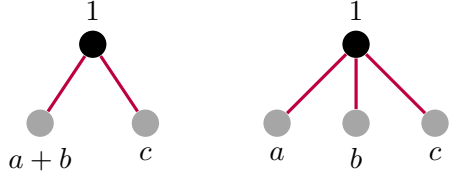
\begin{center}
\begin{tikzpicture}[scale=.35,
    mycirc/.style={circle,fill=cyan!40, minimum size=0.5cm}
    ]
     \node[circle,fill=black, scale=1, label=above:{$1$}] (n3) at (7,0) {};
    \node[circle,fill=black!35, label=below:{$a+b$}] (n4) at (5, -3) {};
    \node[circle,fill=black!35, label=below:{$c$}] (n5) at (9, -3) {};
\draw [very thick, purple](n3) -- (n4);
\draw [very thick, purple] (n3) -- (n5);

     \node[circle,fill=black,, label=above:{$1$}] (n6) at (17,0) {};
    \node[circle,fill=black!35, label=below:{$a$}] (n7) at (14, -3) {};
    \node[circle,fill=black!35, label=below:{$b$}] (n8) at (17, -3) {};
        \node[circle,fill=black!35, label=below:{$c$}] (n9) at (20, -3) {};
\draw [very thick, purple](n6) -- (n7);
\draw [very thick, purple](n6) -- (n8);
\draw [very thick, purple](n6) -- (n9);     
\end{tikzpicture}
\end{center}
The contributions of $\mathsf{R}$ and $\mathsf{S}$ are obtained from \eqref{compcont},
\begin{align*}
    \mathsf{Cont}_{\mathsf{R}}&=\left[\frac{c(\N)}{(1+ z_1)(1+z_2)}\right]_{a+b+c-2}\, ,\\
    \mathsf{Cont}_{\mathsf{S}}&=\left[\frac{c(\N)}{(1+ z_2)(1+z_3)(1+z_4)}\right]_{a+b+c-3}\,.
\end{align*}
 Here, $z_1$ corresponds to the edge joining the genus $0$ vertex to the root, $z_2$ corresponds to the edge joining the genus $c$ vertex to the root, while $z_3, z_4$ correspond to the remaining edges.
 
We only consider the case $a+b+c=g-1=5$. From equation \eqref{inductive}, we obtain 
\begin{align*}
\mathsf{Cont}_{\mathsf{T}}\cdot z_1z_2z_3z_4=c_{5}(\N)-z_1z_2\left[\frac{c(\N)}{(1+ z_1)(1+z_2)}\right]_{3}-z_2z_3z_4\left[\frac{c(\N)}{(1+z_2)(1+z_3)(1+z_4)}\right]_{2}\,,
\end{align*}
where $$c(\mathcal N)=(1+z_2)(1+z_1+z_3)(1+z_1+z_4)(1+\ell_1)(1+\ell_2)\,.$$ 
Therefore $$\mathsf{Cont}_{\mathsf T}=-3c_1(\mathcal N)+6z_1+3z_2+4(z_3+z_4)\,.$$
\end{example}
\begin{example}\label{answer15}  Similarly, we compute the contribution of the extremal tree $\mathsf T$ shown below. 
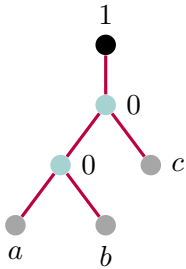
\begin{center}
\begin{tikzpicture}[scale=.4]
\node[circle,fill=black, scale=.75, label=above:{$1$}] (n1) at (11-8,0) {};
    \node[circle,fill=teal!35, scale=.75, label=right:{$0$}] (n2) at (11-8,-2) {};
 \node[circle,fill=teal!35, scale=.75, label=right:{$0$}] (n3) at (9.5-8,-4) {};
  \node[circle,fill=black!35, scale=.75, label=right:{$c$}] (n4) at (12.5-8,-4) {};
   \node[circle,fill=black!35, scale=.75, label=below:{$a$}] (n5) at (8-8,-6) {};
    \node[circle,fill=black!35, scale=.75, label=below:{$b$}] (n6) at (11-8,-6) {};
 \draw [very thick, purple](n1) -- (n2); 
   \draw [very thick, purple](n2) -- (n3);
   \draw [very thick, purple](n2) -- (n4);
\draw [very thick, purple](n3) -- (n5);
\draw [very thick, purple](n3) -- (n6);
\end{tikzpicture}
\end{center}
There are $6$ smoothings indexed by the following trees $\mathsf R_1-\mathsf R_6$. 
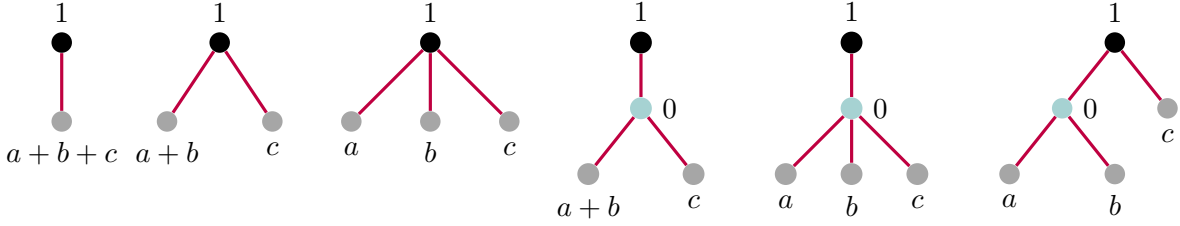
\begin{center}
\begin{tikzpicture}[scale=.35,
    mycirc/.style={circle,fill=cyan!40, minimum size=0.5cm}
    ]
    \node[circle,fill=black, scale=.75, label=above:{$1$}] (n1) at (0,0) {};
    \node[circle,fill=black!35, scale=.75, label=below:{$a+b+c$}] (n2) at (0,-3) {};
 \draw [very thick, purple](n1) -- (n2); 
    
     \node[circle,fill=black, scale=.75, label=above:{$1$}] (n3) at (5+1,0) {};
    \node[circle,fill=black!35, scale=.75, label=below:{$a+b$}] (n4) at (3+1, -3) {};
    \node[circle,fill=black!35, scale=.75, label=below:{$c$}] (n5) at (7+1, -3) {};
\draw [very thick, purple](n3) -- (n4);
\draw [very thick, purple] (n3) -- (n5);

     \node[circle,fill=black, scale=.75, label=above:{$1$}] (n7) at (20-6,0) {};
    \node[circle,fill=black!35, scale=.75, label=below:{$a$}] (n8) at (17-6, -3) {};
    \node[circle,fill=black!35, scale=.75, label=below:{$b$}] (n9) at (20-6, -3) {};
        \node[circle,fill=black!35, scale=.75, label=below:{$c$}] (n10) at (23-6, -3) {};
\draw  [very thick, purple](n7) -- (n8);
\draw  [very thick, purple](n7) -- (n9);
\draw  [very thick, purple](n7) -- (n10);

    \node[circle,fill=black, scale=.8, label=above:{$1$}] (m1) at (7+15,0) {};
    \node[circle,fill=teal!35, scale=.8, label=right:{$0$}] (m2) at (7+15,-2.5) {};
    \node[circle,fill=black!35, scale=.8, label=below:{$a+b$}] (m3) at (5+15,-5) {};
    \node[circle,fill=black!35, scale=.8, label=below:{$c$}] (m4) at (9+15,-5) {};
    
 \draw [very thick, purple](m1) -- (m2); 
\draw [very thick, purple](m2) -- (m3);
\draw [very thick, purple](m2) -- (m4);

 \node[circle,fill=black, scale=.8, label=above:{$1$}] (p1) at (12+18,0) {};
    \node[circle,fill=teal!35, scale=.8, label=right:{$0$}] (p2) at (12+18,-2.5) {};
    \node[circle,fill=black!35, scale=.8, label=below:{$a$}] (p3) at (9.5+18,-5) {};
    \node[circle,fill=black!35, scale=.8, label=below:{$b$}] (p4) at (12+18,-5) {};
    \node[circle,fill=black!35, scale=.8, label=below:{$c$}] (p5) at (14.5+18,-5) {};
 \draw [very thick, purple](p1) -- (p2); 
\draw [very thick, purple](p2) -- (p3);
\draw [very thick, purple](p2) -- (p4); 
\draw [very thick, purple](p2) -- (p5); 

\node[circle,fill=black, scale=.75, label=above:{$1$}] (p11) at (4+36,0) {};
    \node[circle,fill=teal!35, scale=.75, label=right:{$0$}] (p12) at (2+36,-2.5) {};
 \node[circle,fill=black!35, scale=.75, label=below:{$c$}] (p13) at (6+36,-2.5) {};
  \node[circle,fill=black!35, scale=.75, label=below:{$a$}] (p14) at (0+36,-5) {};
   \node[circle,fill=black!35, scale=.75, label=below:{$b$}] (p15) at (4+36,-5) {};
 \draw [very thick, purple](p11) -- (p12); 
   \draw [very thick, purple](p11) -- (p13);
   \draw [very thick, purple](p12) -- (p14);
\draw [very thick, purple](p12) -- (p15);

\end{tikzpicture}
\end{center}
We have $$c(\mathcal N)=(1+z_1+z_2+z_4)(1+z_1+z_2+z_5)(1+z_1+z_3)(1+\ell_1)(1+\ell_2),$$ where the edge emanating from the root is labelled $z_1$, and the edges at the adjacent genus $0$ vertex are $z_2, z_3$ from left to right, and $z_4, z_5$ are the remaining edges, again labeled from left to right. We assume $a+b+c=5=g-1.$

The contributions of the first $3$ irreducible trees are 
\begin{align*}\mathsf {Cont}_{\mathsf R_1}&=\left[\frac{c(\mathcal N)}{1+z_1}\right]_{4}\,,\\
\mathsf{Cont}_{\mathsf {R_2}}&=\left[\frac{c(\mathcal N)}{(1+z_2)(1+z_3)}\right]_{3}\, ,\\
\mathsf{Cont}_{\mathsf {R_3}}&=\left[\frac{c(\mathcal N)}{(1+z_3)(1+z_4)(1+z_5)}\right]_{2}\,.
\end{align*} 
The contributions of $\mathsf{R}_4$, $\mathsf{R}_5$, and $\mathsf{R}_6$ can be calculated using Examples \ref{answer3} (iii), \ref{answer4}(ii), and \ref{weird}, respectively. We find 
\begin{align*}
\mathsf{Cont}_{\mathsf R_4}&=-3c_2(\mathcal N)+c_1(\mathcal N)\cdot (6z_1+4z_2+4z_3)-10z_1^2-10 z_1\cdot (z_2+z_3)-5(z_2+z_3)^2+5z_2z_3\,,\\
\mathsf{Cont}_{\mathsf R_5}&=-4c_1(\mathcal N)+10z_1+5(z_3+z_4+z_5)\,,\\
\mathsf{Cont}_{\mathsf {R}_6}&=-3c_1(\mathcal N)+6z_2+3z_3+4(z_4+z_5)\,.
\end{align*}
The recursion to be solved is 
\begin{align*}c_5(\mathcal N)=z_1z_2z_3z_4z_5{\mathsf {Cont}}_{\mathsf T}&+z_1\mathsf{Cont}_{\mathsf {R}_1}
+z_2z_3\mathsf{Cont}_{\mathsf R_2}+z_3z_4z_5\mathsf{Cont}_{\mathsf R_3}+z_1z_2z_3\mathsf{Cont}_{\mathsf R_4}\\&+z_1z_3z_4z_5\mathsf{Cont}_{\mathsf R_5} +z_2z_3z_4z_5\mathsf{Cont}_{\mathsf R_6}\,,\end{align*} which gives $\mathsf{Cont}_{\mathsf T}=15.$
\end{example}

\begin{proof}[Proof of Theorem \ref{tortaut}]
    The algorithm in Section \ref{inductivesection} yields
    the equation \begin{equation}\label{torexpl}\Tor^*[\A_1\times \A_{g-1}]=\sum_{\mathsf T}\frac{1}{|\Aut \mathsf{T}|} \iota_{\mathsf T *} \mathsf{Cont}_{\mathsf{T}}\,,
\end{equation} where $\mathsf{Cont}_{\mathsf{T}}$ is a polynomial in $\lambda$ and $\psi$ classes. The contributions can be computed recursively one tree at a time, with \eqref{excess} providing the base case of the recursion. In particular, formula \eqref{torexpl} shows that \[\Tor^*[\A_1\times \A_{g-1}]\in \R^{g-1}(\M_g^{\ct})\,.\qedhere\]\end{proof}

\subsection{Pixton's formula}\label{Pixformula}
Pixton has solved our recursion to provide a beautiful and concise expression for 
$\mathsf{Cont}_{\mathsf{T}}$. 
Though not needed for the results of our paper, we present his formula here. The proof will appear in \cite{Pixx}.

Let $\mathsf{T}$ be an extremal tree of genus $g$
with $n$ edges and $k$ leaves. Let
$Z=\{ z_e\}_{e\in \mathsf{E}(\mathsf{T})}$
be the set of edge variables as before.
Consider first the expression
\begin{equation}
\label{apap}
(-1)^k \frac{\prod_{v\in \mathsf{V}(\mathsf{T})}
\big(1+ \sum_{e\in \mathsf{path}(v)} z_e\big)^{\mathsf{val}(v)-2}
}
{\prod_{e\in \mathsf{E}(\mathsf{T})}z_e}\, .
\end{equation}
After expanding the numerator in \eqref{apap}, we
obtain a Laurent series in the variables $Z$.
Let 
\begin{equation*}
\left[ (-1)^k \frac{\prod_{v\in \mathsf{V}(\mathsf{T})}
\big(1+ \sum_{e\in \mathsf{path}(v)} z_e\big)^{\mathsf{val}(v)-2}
}
{\prod_{e\in \mathsf{E}(\mathsf{T})}z_e} \right]_{Z^{\geq0}}\, 
\end{equation*}
denote the {\em Taylor part}: the power series  in $Z$
obtained by removing all the strictly polar parts of the
Laurent series \eqref{apap}.

When considering power series
in the variables $Z$ and the Chern classes
$c_i(\mathcal{N})$, we will use the
standard Chow degree: $z_e$ has degree 1,
$c_i(\mathcal{N})$ has degree $i$.

\vspace{8pt}
\noindent {\bf Theorem} [Pixton's formula]{\bf{.}}
{\em The polynomial $\mathsf{P}_{\mathsf{T}}(Z,\mathcal{N})$
determining $\mathsf{Cont}_{\mathsf{T}}$ is the degree
$g-1-n$ part of the power series
\begin{equation*}
\left[ (-1)^k \frac{\prod_{v\in \mathsf{V}(\mathsf{T})}
\big(1+ \sum_{e\in \mathsf{path}(v)} z_e\big)^{\mathsf{val}(v)-2}
}
{\prod_{e\in \mathsf{E}(\mathsf{T})}z_e} \right]_{Z^{\geq0}} \cdot \, c(\mathcal{N})\, , 
\end{equation*}
where $c(\mathcal{N})$ denotes the total Chern class.}

\section{Calculations for \texorpdfstring{$g\leq7$}{g<=7}}\label{lowgenus}
\subsection{Genus 4 and 5}
 We implement here the excess intersection theory developed in Section \ref{excesstheory} to calculate the Torelli pullback of
\[
\Delta_g=[\mathcal A_1\times \mathcal A_{g-1}]-\frac{g}{6|B_{2g}|}\lambda_{g-1} \,. 
\] 
As discussed in
Section \ref{mainr},
$\Delta_g=0 \in \mathsf{CH}^{g-1}(\A_g)$ for $1\leq g \leq 3$. 

\begin{prop}\label{delta4}
    For $g=4$, we have
    $
    \Tor^*\Delta_4=\Tor^*[\A_1\times \A_3]-20\lambda_3=0
    \in \mathsf{R}^{3}(\M^{\ct}_4)$.
\end{prop}
\begin{proof}

In genus $4$, there are four extremal trees with at most $3$ edges: $\mathsf{A},\mathsf{B},\mathsf{C}$, and $\mathsf{D}$, drawn below:

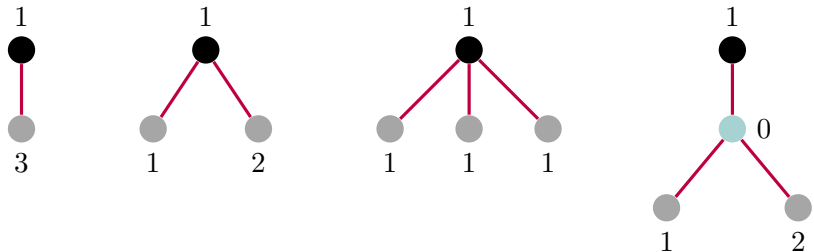
\begin{center}
\begin{tikzpicture}[scale=.35]
    \node[circle,fill=black, scale=1, label=above:{$1$}] (n1) at (0,0) {};
    \node[circle,fill=black!35, scale=1, label=below:{$3$}] (n2) at (0,-3) {};
 \draw [very thick, purple](n1) -- (n2); 
    
     \node[circle,fill=black, scale=1, label=above:{$1$}] (n3) at (7,0) {};
    \node[circle,fill=black!35, label=below:{$1$}] (n4) at (5, -3) {};
    \node[circle,fill=black!35, label=below:{$2$}] (n5) at (9, -3) {};
\draw [very thick, purple](n3) -- (n4);
\draw [very thick, purple] (n3) -- (n5);

     \node[circle,fill=black,, label=above:{$1$}] (n6) at (17,0) {};
    \node[circle,fill=black!35, label=below:{$1$}] (n7) at (14, -3) {};
    \node[circle,fill=black!35, label=below:{$1$}] (n8) at (17, -3) {};
        \node[circle,fill=black!35, label=below:{$1$}] (n9) at (20, -3) {};
\draw [very thick, purple](n6) -- (n7);
\draw [very thick, purple](n6) -- (n8);
\draw [very thick, purple](n6) -- (n9);

     \node[circle,fill=black, label=above:{$1$}] (n10) at (27,0) {};
    \node[circle,fill=teal!35, label=right:{$0$}] (n11) at (27, -3) {};
    \node[circle,fill=black!35, label=below:{$1$}] (n12) at (25-0.5, -6) {};
        \node[circle,fill=black!35,, label=below:{$2$}] (n13) at (29+0.5, -6) {};
\draw  [very thick, purple](n10) -- (n11);
\draw  [very thick, purple](n11) -- (n12);
\draw  [very thick, purple](n11) -- (n13); 

\end{tikzpicture}
\end{center} The first three, $\mathsf{A}$,
$\mathsf{B}$, and $\mathsf{C}$, correspond to the irreducible components of $\Tor^{-1}(\A_1\times \A_{3})$, and the fourth  $\mathsf{D}$ is the intersection of the first two components. Other extremal trees corresponding to the remaining intersections occur in higher codimension, and thus do not contribute to the calculation.  The contribution from $\mathsf{A}$ is computed via \eqref{excess}:
\[\left[ \frac{c(\mathbb{E}^{\vee})}{1-\psi_1}\right]_{2}=[1,\lambda_2-\lambda_1\psi_1+\psi_1^2]\,.
\]
Here, the component $\mathsf{A}$ is $\M_{1, 1}^{\ct}\times \M_{3, 1}^{\ct}$ and the notation $[, ]$ indicates the contribution from each factor, respecting the order of the factors in the product. 
Similarly, $\mathsf{B}$ corresponds to the product $\M_{1, 2}^{\ct}\times \M_{1, 1}^{\ct}\times \M_{2, 1}^{\ct}$. The contribution of $\mathsf B$ can also be found via \eqref{excess} yielding 
\[
[\psi_1+\psi_2,1,1]+[1,1, \psi_1-\lambda_1]\,.
\] The extremal tree $\mathsf{C}$ occurs in the correct codimension and has an automorphism group of order $6$. Finally, by Example \ref{answer3}(i) in Section \ref{excesstheory}, the contribution of $\mathsf{D}$ is $-3$ times the fundamental class.

We push forward the contributions of $\mathsf A, \mathsf B, \mathsf C, \mathsf D$ to $\M_{4}^{\ct}$, dividing by the orders of the respective automorphism groups, and subtract $20\lambda_3$. Using \texttt{admcycles}\cite{admcycles}, we verify 
$$\Tor^*[\A_1\times \A_3]-20\lambda_3=0 \in \mathsf{R}^3(\M_4^{\ct})\, .$$ The code for the calculation can be found in \cite {COP}. 
 \end{proof}


\begin{prop}\label{4tautological} The classes $[\A_1\times \A_1\times \A_2]$ and $[\A_1\times \A_1\times \A_1\times \A_1]$ are tautological in $\mathsf{CH}^*(\A_4)$: $$[\A_1\times \A_1\times \A_2]=420\lambda_2\lambda_3\, , \quad [\A_1\times \A_1\times \A_1\times \A_1]=4200\lambda_1\lambda_2\lambda_3\, .$$
\end{prop}

Proposition \ref{4tautological}
implies Conjecture \ref{socprod} in genus 4.
Whether the class $[\A_1\times \A_3]\in \CH^3(\A_4)$ is tautological remains an open question. By Proposition \ref{algcyclee}, 
$[\A_1\times \A_3]$ is tautological in cohomology. 
\begin{proof}
    
    In genus $4$, the Schottky locus is a divisor in $\A_4$, hence
    the class 
    $\Tor_*[\M_4^{\ct}]\in \mathsf{CH}^1(\A_4)$  is a multiple{\footnote{By results of Igusa the multiple equals 8, but our argument does not require knowledge of the multiple.}}
        of $\lambda_1$ since the Picard rank of $\A_4$ equals 1.  By Proposition \ref{delta4} and the projection formula, we find  
    \[
    \Tor^*\Delta_4=0\implies \Tor_*\Tor^*([\A_1\times \A_3]-20\lambda_3)=0\implies \lambda_1([\A_1\times \A_3]-20\lambda_3)=0\,.
    \]
Intersecting with $\lambda_1$, we obtain $$\lambda_1^2\,[\A_1\times \A_3]=20\lambda_1^2\lambda_3=40\lambda_2\lambda_3\implies [\A_1\times \A_1\times \A_2]=420\lambda_2\lambda_3\,,$$ where the Mumford relation $\lambda^2_1=2\lambda_2$ was used in the first equation, and the relation $$[\A_1\times \A_2]=\frac{21}{2}\lambda_1^2 \in \mathsf{CH}^2(\A_3)$$of \cite[Proposition 3.2]{vdg2} is used for the second equation. Intersecting with $\lambda_1$ one more time, and using $[\A_1\times \A_1]=10\lambda_1\in \mathsf{CH}^1(\A_2)$ by \cite[Lemma 2.2]{vdg2}, we obtain \[[\A_1\times \A_1\times \A_1\times \A_1]=10\lambda_1[\A_1\times \A_1\times \A_2]=4200\,\lambda_1\lambda_2\lambda_3\,.\qedhere\]
\end{proof}
\vskip.1in
\begin{prop}\label{delta5}
        We have
    $
    \Tor^*\Delta_5=\Tor^*[\A_1\times \A_4]-11\lambda_4=0\in
    \mathsf{R}^4(\M_5^{\ct})$.
    \end{prop}

\begin{proof} 
We calculate as in the proof of Proposition \ref{delta4}, but there are more trees to consider.

The irreducible components of $\Tor^{-1} (\A_1\times \A_4)$ are indexed by the following extremal trees labelled $\mathsf{A},$ $\mathsf{B},$ $\mathsf{C},$ $\mathsf {D},$ $\mathsf{E}$ respectively:

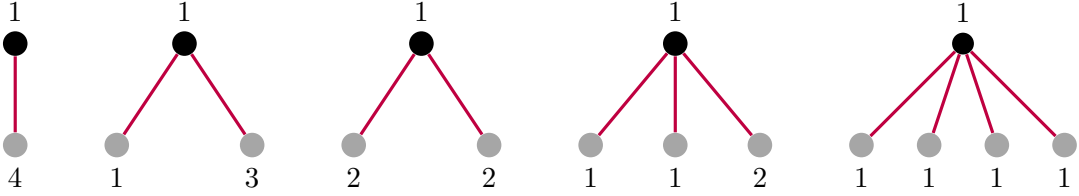
\begin{center}
\begin{tikzpicture}[scale=.45,
    mycirc/.style={circle,fill=cyan!40, minimum size=0.5cm}
    ]
    \node[circle,fill=black, scale=.9, label=above:{$1$}] (n1) at (0,0) {};
    \node[circle,fill=black!35, scale=.9, label=below:{$4$}] (n2) at (0,-3) {};
 \draw [very thick, purple](n1) -- (n2); 
    
     \node[circle,fill=black, scale=.9, label=above:{$1$}] (n3) at (5,0) {};
    \node[circle,fill=black!35, scale=.9, label=below:{$1$}] (n4) at (3, -3) {};
    \node[circle,fill=black!35, scale=.9, label=below:{$3$}] (n5) at (7, -3) {};
\draw [very thick, purple](n3) -- (n4);
\draw [very thick, purple] (n3) -- (n5);

      \node[circle,fill=black, scale=.9, label=above:{$1$}] (n6) at (12,0) {};
    \node[circle,fill=black!35, scale=.9, label=below:{$2$}] (n7) at (10, -3) {};
    \node[circle,fill=black!35,scale=.9,  label=below:{$2$}] (n8) at (14, -3) {};
\draw [very thick, purple](n6) -- (n7);
\draw [very thick, purple] (n6) -- (n8);

     \node[circle,fill=black,scale=.9,  label=above:{$1$}] (n7) at (19+.5,0) {};
    \node[circle,fill=black!35, scale=.9, label=below:{$1$}] (n8) at (17, -3) {};
    \node[circle,fill=black!35, scale=.9, label=below:{$1$}] (n9) at (19+.5, -3) {};
        \node[circle,fill=black!35, scale=.9, label=below:{$2$}] (n10) at (21+1, -3) {};
\draw  [very thick, purple](n7) -- (n8);
\draw  [very thick, purple](n7) -- (n9);
\draw  [very thick, purple](n7) -- (n10); 

    \node[circle,fill=black, scale=.8, label=above:{$1$}] (n11) at (27+1,0) {};
    \node[circle,fill=black!35, scale=.9, label=below:{$1$}] (n12) at (24+1, -3) {};
    \node[circle,fill=black!35, scale=.9, label=below:{$1$}] (n13) at (26+1, -3) {};
        \node[circle,fill=black!35,scale=.9,  label=below:{$1$}] (n14) at (28+1, -3) {};
    \node[circle,fill=black!35,scale=.9,  label=below:{$1$}] (n15) at (30+1, -3) {};

\draw  [very thick, purple](n11) -- (n12);
\draw  [very thick, purple](n11) -- (n13);
\draw  [very thick, purple](n11) -- (n14);
\draw  [very thick, purple](n11) -- (n15); 

\end{tikzpicture}
\end{center}
We also list the intersections that have codimension at most $4$: 

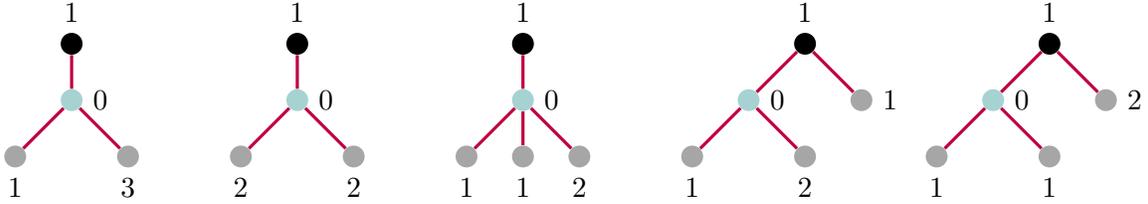
\begin{center}
\begin{tikzpicture}[scale=.5,
    mycirc/.style={circle,fill=cyan!40, minimum size=0.5cm}
    ]
    \node[circle,fill=black, scale=.8, label=above:{$1$}] (n1) at (2,0) {};
    \node[circle,fill=teal!35, scale=.8, label=right:{$0$}] (n2) at (2,-1.5) {};
    \node[circle,fill=black!35, scale=.8, label=below:{$1$}] (n3) at (0.5,-3) {};
    \node[circle,fill=black!35, scale=.8, label=below:{$3$}] (n4) at (3.5,-3) {};
    
 \draw [very thick, purple](n1) -- (n2); 
\draw [very thick, purple](n2) -- (n3);
\draw [very thick, purple](n2) -- (n4);

    \node[circle,fill=black, scale=.8, label=above:{$1$}] (m1) at (7+1,0) {};
    \node[circle,fill=teal!35, scale=.8, label=right:{$0$}] (m2) at (7+1,-1.5) {};
    \node[circle,fill=black!35, scale=.8, label=below:{$2$}] (m3) at (5.5+1,-3) {};
    \node[circle,fill=black!35, scale=.8, label=below:{$2$}] (m4) at (8.5+1,-3) {};
    
 \draw [very thick, purple](m1) -- (m2); 
\draw [very thick, purple](m2) -- (m3);
\draw [very thick, purple](m2) -- (m4);

 \node[circle,fill=black, scale=.8, label=above:{$1$}] (p1) at (12+2,0) {};
    \node[circle,fill=teal!35, scale=.8, label=right:{$0$}] (p2) at (12+2,-1.5) {};
    \node[circle,fill=black!35, scale=.8, label=below:{$1$}] (p3) at (10.5+2,-3) {};
    \node[circle,fill=black!35, scale=.8, label=below:{$1$}] (p4) at (12+2,-3) {};
    \node[circle,fill=black!35, scale=.8, label=below:{$2$}] (p5) at (13.5+2,-3) {};
 \draw [very thick, purple](p1) -- (p2); 
\draw [very thick, purple](p2) -- (p3);
\draw [very thick, purple](p2) -- (p4); 
\draw [very thick, purple](p2) -- (p5);

   \node[circle,fill=black, scale=.8, label=above:{$1$}] (q1) at (18.5+3,0) {};
    \node[circle,fill=teal!35, scale=.8, label=right:{$0$}] (q2) at (17+3,-1.5) {};
    \node[circle,fill=black!35, scale=.8, label=right:{$1$}] (q3) at (20+3,-1.5) {};
    \node[circle,fill=black!35, scale=.8, label=below:{$1$}] (q4) at (15.5+3,-3) {};
    \node[circle,fill=black!35, scale=.8, label=below:{$2$}] (q5) at (18.5+3,-3) {};
 \draw [very thick, purple](q1) -- (q2); 
\draw [very thick, purple](q1) -- (q3);
\draw [very thick, purple](q2) -- (q4); 
\draw [very thick, purple](q2) -- (q5);

   \node[circle,fill=black, scale=.8, label=above:{$1$}] (r1) at (25+3,0) {};
    \node[circle,fill=teal!35, scale=.8, label=right:{$0$}] (r2) at (23.5+3,-1.5) {};
    \node[circle,fill=black!35, scale=.8, label=right:{$2$}] (r3) at (26.5+3,-1.5) {};
    \node[circle,fill=black!35, scale=.8, label=below:{$1$}] (r4) at (22+3,-3) {};
    \node[circle,fill=black!35, scale=.8, label=below:{$1$}] (r5) at (25+3,-3) {};
 \draw [very thick, purple](r1) -- (r2); 
\draw [very thick, purple](r1) -- (r3);
\draw [very thick, purple](r2) -- (r4); 
\draw [very thick, purple](r2) -- (r5);

\end{tikzpicture}
\end{center}
The above $5$ extremal trees correspond to the intersections $\mathsf A\cap \mathsf B$, $\mathsf A\cap \mathsf C$, $\mathsf A\cap \mathsf D$, $\mathsf B\cap \mathsf D$, $\mathsf C\cap \mathsf D$ respectively.  

We compute the contributions of each of the $10$ trees above as follows: 
\begin{itemize}
    \item [(i)]
The contributions of the extremal trees $\mathsf A-\mathsf E$ are computed using \eqref{excess}. The trees $\mathsf{C}$ and $\mathsf{D}$ have $2$ automorphisms each, while the tree $\mathsf E$ has $24$ automorphisms. There are no automorphisms for $\mathsf{A}$, $\mathsf {B}$. We obtain
\begin{align*}
\frac{1}{|\Aut(\mathsf{A})|}\mathsf{Cont}_{\mathsf{A}}&=[1,-\lambda_3+\lambda_2\psi_1-\lambda_1\psi_1^2+\psi_1^3]\,,\\
\frac{1}{|\Aut(\mathsf{B})|}\mathsf{Cont}_{\mathsf{B}}&= [1, 1, \lambda_2]-[\psi_1+\psi_2, 1, \lambda_1]-[1, 1, \lambda_1\psi_1]+[\psi_1, 1, \psi_1]+2[\psi_2, 1, \psi_1]+[1, 1, \psi_1^2]\,,\\
\frac{1}{|\Aut(\mathsf{C})|}\mathsf{Cont}_{\mathsf{C}}&=\frac{1}{2} ([1, \lambda_1, \lambda_1]-[\psi_1+\psi_2, \lambda_1, 1]-[\psi_1+\psi_2, 1, \lambda_1]-[1, \psi_1\lambda_1, 1]-[1, \psi_1, \lambda_1]\\& +[2\psi_1+\psi_2, \psi_1, 1] +[1, \psi_1^2, 1]-[1, \lambda_1, \psi_1]-[1, 1, \lambda_1\psi_1]+[\psi_1+2\psi_2, 1, \psi_1]\\&+[1, \psi_1, \psi_1]+[1, 1, \psi_1^2])\,,\\
\frac{1}{|\Aut(\mathsf{D})|}\mathsf{Cont}_{\mathsf{D}}&=\frac{1}{2}\left([\psi_1+\psi_2+\psi_3,1,1,1]+[1,1, 1, \psi_1-\lambda_1]\right)\,\\
\frac{1}{|\Aut(\mathsf{E})|}\mathsf{Cont}_{\mathsf{E}}&=\frac{1}{24}[1,1,1,1,1]\, .
\end{align*}
The order of the terms in the brackets [] places the root contribution on the first position, followed by the contribution of the non-root vertices from left to right in increasing order of the genus. 
We also ignore the terms of degree $>2g-3+n$ on any component $\M_{g, n}^{\ct}$ since such terms vanish by \eqref{vanish}. 

\item [(ii)] Moving on to the intersections, we consider the first extremal tree which represents $\mathsf A\cap \mathsf B$. The corresponding locus is $\M_{1, 1}^{\ct}\times \M_{0, 3}^{\ct}\times \M_{1, 1}^{\ct}\times \M_{3, 1}^{\ct}.$
By Example \ref{answer3}(ii) the contribution equals $$-3c_1(\mathcal N)+6z_1+4z_2+4z_3\,.$$ Here, $\mathcal N$ is the restriction of the normal bundle of $$\A_1\times \A_4\to \A_5$$ to $\mathsf A\cap \mathsf B.$ By \eqref{normspl} and using that the Hodge bundle in genus $0, 1$ has trivial Chern classes, we find $$c_1(\mathcal N)=[1, 1, 1, -\lambda_1]\, .$$ 
Next, as explained in Section \ref{inductivesection}, we substitute the edge variables by the negative sums of cotangent classes: $$z_1\mapsto 0, \quad z_2\mapsto 0, \quad z_3\mapsto [1, 1, 1, -\psi_1]\,,$$ where we have used that the $\psi$-classes on the factors $\M_{1, 1}^{\ct}$ and $\M_{0, 3}^{\ct}$ vanish.  
Collecting terms, we see $$\frac{1}{|\Aut(\mathsf{AB})|}\mathsf{Cont}_{\mathsf{AB}}=[1, 1, 1, 3\lambda_1-4\psi_1]\,.$$

For the second tree corresponding to $\mathsf A\cap \mathsf C$ the calculation is similar. The contribution  equals 
$$-3c_1(\mathcal N)+6z_1+4z_2+4z_3$$ over $\M_{1, 1}^{\ct}\times \M_{0, 3}^{\ct}\times \M_{2, 1}^{\ct}\times \M_{2, 1}^{\ct}.$ We must divide by $2$ because of automorphisms. We obtain  $$\frac{1}{|\Aut(\mathsf{AC})|}\mathsf{Cont}_{\mathsf A\mathsf C}=\frac{1}{2}\left([1, 1, 3\lambda_1-4\psi_1, 1 ]+[1, 1, 1, 3\lambda_1-4\psi_1]\right).$$

\item [(iii)] For the last $3$ extremal trees, the corresponding loci have codimension $4$. The excess contributions are computed  by Example \ref{answer4}(i) and Example \ref{answer3}(i) and they equal $-4, -3, -3$. The number of automorphisms are $2, 1, 2.$ The contributions of these loci divided by the order of the automorphism group are $-\frac{4}{2}, -3, -\frac{3}{2}$ times their fundamental classes, respectively. 
\end{itemize}
We collect all terms in (i)-(iii), push forward the weighted contributions to $\M_{5}^{\ct}$ and subtract $11\lambda_4$. Using \texttt{admcycles} \cite{admcycles}, we verify  
$$\Tor^*[\A_1\times \A_4]-11\lambda_4=0\in \mathsf{R}^4(\M_5^{\ct})\, .$$ The code for the calculation can be found in \cite {COP}. 
\end{proof}
\begin{rem}
    An alternate proof of Propositions \ref{delta4} and \ref{delta5} is as follows. By Theorem \ref{vang}, $\Tor^*\Delta_g$ is in the kernel of the $\lambda_g$-pairing, but the $\lambda_4$ and $\lambda_5$-pairings are perfect \cite{CLS}. 
\end{rem}

\subsection{Genus 6: Proof of Theorem \ref{Delta6}}

    As explained in  Section \ref{mainr}, the last assertion of Theorem \ref{Delta6} follows from Proposition \ref{tautprop}.
    The kernel of the $\lambda_6$-pairing was computed in \cite{CLS}: it is an explicit 1-dimensional subspace of $\mathsf{R}^5(\M_{6}^{\ct})$. We will compute $\Tor^*\Delta_6$ 
    using the excess calculus. Then, we 
    will show that 
    $\Tor^*\Delta_6$ 
    generates the kernel of the
    $\lambda_6$-pairing
    using \texttt{admcycles} \cite{admcycles}. 

    There are $24$ extremal trees contributing to the calculation of $\Tor^*\Delta_6.$ The irreducible components of $\Tor^{-1}(\A_1\times \A_5)$ are indexed by the following $7$ extremal trees denoted $\mathsf{A} - \mathsf{G}$: 
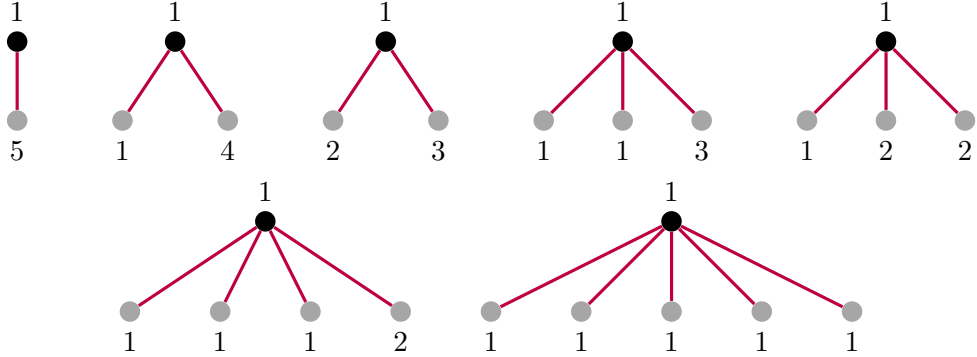
\begin{center}
\begin{tikzpicture}[scale=.35,
    mycirc/.style={circle,fill=cyan!40, minimum size=0.5cm}
    ]
    \node[circle,fill=black, scale=.75, label=above:{$1$}] (n1) at (0,0) {};
    \node[circle,fill=black!35, scale=.75, label=below:{$5$}] (n2) at (0,-3) {};
 \draw [very thick, purple](n1) -- (n2); 
    
     \node[circle,fill=black, scale=.75, label=above:{$1$}] (n3) at (5+1,0) {};
    \node[circle,fill=black!35, scale=.75, label=below:{$1$}] (n4) at (3+1, -3) {};
    \node[circle,fill=black!35, scale=.75, label=below:{$4$}] (n5) at (7+1, -3) {};
\draw [very thick, purple](n3) -- (n4);
\draw [very thick, purple] (n3) -- (n5);

      \node[circle,fill=black, scale=.75, label=above:{$1$}] (n6) at (12+2,0) {};
    \node[circle,fill=black!35, scale=.75, label=below:{$2$}] (n7) at (10+2, -3) {};
    \node[circle,fill=black!35, scale=.75, label=below:{$3$}] (n8) at (14+2, -3) {};
\draw [very thick, purple](n6) -- (n7);
\draw [very thick, purple] (n6) -- (n8);

     \node[circle,fill=black, scale=.75, label=above:{$1$}] (n7) at (20+3,0) {};
    \node[circle,fill=black!35, scale=.75, label=below:{$1$}] (n8) at (17+3, -3) {};
    \node[circle,fill=black!35, scale=.75, label=below:{$1$}] (n9) at (20+3, -3) {};
        \node[circle,fill=black!35, scale=.75, label=below:{$3$}] (n10) at (23+3, -3) {};
\draw  [very thick, purple](n7) -- (n8);
\draw  [very thick, purple](n7) -- (n9);
\draw  [very thick, purple](n7) -- (n10);

  \node[circle,fill=black, scale=.75, label=above:{$1$}] (n11) at (29+4,0) {};
    \node[circle,fill=black!35, scale=.75, label=below:{$1$}] (n12) at (26+4, -3) {};
    \node[circle,fill=black!35, scale=.75, label=below:{$2$}] (n13) at (29+4, -3) {};
        \node[circle,fill=black!35, scale=.75, label=below:{$2$}] (n14) at (32+4, -3) {};
\draw  [very thick, purple](n11) -- (n12);
\draw  [very thick, purple](n11) -- (n13);
\draw  [very thick, purple](n11) -- (n14); 
\end{tikzpicture}
\begin{tikzpicture}[scale=.4]
 \node[circle,fill=black, scale=.75, label=above:{$1$}] (m1) at (39.5,0) {};
    \node[circle,fill=black!35, scale=.75, label=below:{$1$}] (m2) at (35, -3) {};
    \node[circle,fill=black!35, scale=.75, label=below:{$1$}] (m3) at (38, -3) {};
        \node[circle,fill=black!35, scale=.75, label=below:{$1$}] (m4) at (41, -3) {};
        \node[circle,fill=black!35,scale=.75,  label=below:{$2$}] (m5) at (44, -3) {};
\draw  [very thick, purple](m1) -- (m2);
\draw  [very thick, purple](m1) -- (m3);
\draw  [very thick, purple](m1) -- (m4); 
\draw  [very thick, purple](m1) -- (m5); 

 \node[circle,fill=black, scale=.75, label=above:{$1$}] (p1) at (53,0) {};
    \node[circle,fill=black!35, scale=.75, label=below:{$1$}] (p2) at (47, -3) {};
    \node[circle,fill=black!35, scale=.75, label=below:{$1$}] (p3) at (50, -3) {};
        \node[circle,fill=black!35,scale=.75,  label=below:{$1$}] (p4) at (53, -3) {};
        \node[circle,fill=black!35, scale=.75,  label=below:{$1$}] (p5) at (56, -3) {};
         \node[circle,fill=black!35,scale=.75,  label=below:{$1$}] (p6) at (59, -3) {};
\draw  [very thick, purple](p1) -- (p2);
\draw  [very thick, purple](p1) -- (p3);
\draw  [very thick, purple](p1) -- (p4); 
\draw  [very thick, purple](p1) -- (p5); 
\draw  [very thick, purple](p1) -- (p6); 
\end{tikzpicture}
\end{center}
Additionally, there are $8$ extremal trees with at most $4$ edges which arise from intersections of the components: 

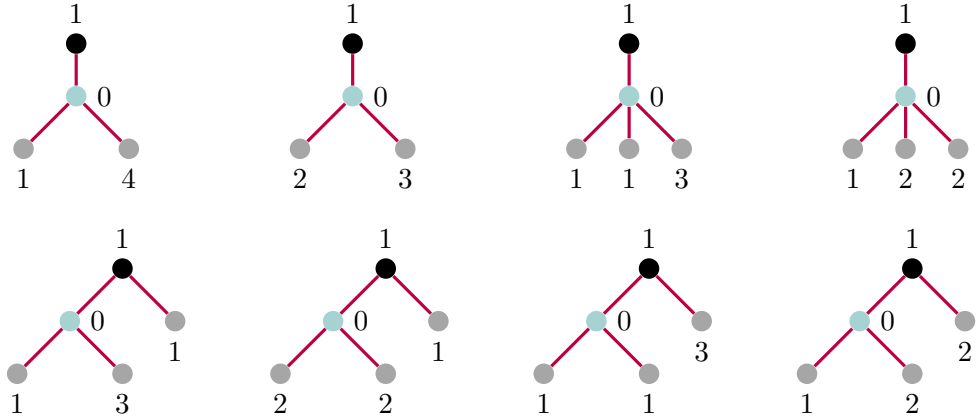
\begin{center}
\begin{tikzpicture}[scale=.35,
    mycirc/.style={circle,fill=cyan!40, minimum size=0.5cm}
    ]
    \node[circle,fill=black, scale=.75, label=above:{$1$}] (n1) at (2,0) {};
    \node[circle,fill=teal!35, scale=.75, label=right:{$0$}] (n2) at (2,-2) {};
 \node[circle,fill=black!35, scale=.75, label=below:{$1$}] (n3) at (0,-4) {};
  \node[circle,fill=black!35, scale=.75, label=below:{$4$}] (n4) at (4,-4) {};
 \draw [very thick, purple](n1) -- (n2); 
   \draw [very thick, purple](n2) -- (n3);
   \draw [very thick, purple](n2) -- (n4);

     \node[circle,fill=black, scale=.75, label=above:{$1$}] (n5) at (8.5+4,0) {};
    \node[circle,fill=teal!35, scale=.75, label=right:{$0$}] (n6) at (8.5+4,-2) {};
 \node[circle,fill=black!35, scale=.75, label=below:{$2$}] (n7) at (6.5+4,-4) {};
  \node[circle,fill=black!35, scale=.75, label=below:{$3$}] (n8) at (10.5+4,-4) {};
 \draw [very thick, purple](n5) -- (n6); 
   \draw [very thick, purple](n6) -- (n7);
   \draw [very thick, purple](n6) -- (n8);

   \node[circle,fill=black, scale=.75, label=above:{$1$}] (m1) at (15+8,0) {};
    \node[circle,fill=teal!35, scale=.75, label=right:{$0$}] (m2) at (15+8,-2) {};
 \node[circle,fill=black!35, scale=.75, label=below:{$1$}] (m3) at (13+8,-4) {};
  \node[circle,fill=black!35, scale=.75, label=below:{$1$}] (m4) at (15+8,-4) {};
   \node[circle,fill=black!35, scale=.75, label=below:{$3$}] (m5) at (17+8,-4) {};
 \draw [very thick, purple](m1) -- (m2); 
   \draw [very thick, purple](m2) -- (m3);
   \draw [very thick, purple](m2) -- (m4);
\draw [very thick, purple](m2) -- (m5);

  \node[circle,fill=black, scale=.75, label=above:{$1$}] (m6) at (21.5+12,0) {};
    \node[circle,fill=teal!35, scale=.75, label=right:{$0$}] (m7) at (21.5+12,-2) {};
 \node[circle,fill=black!35, scale=.75, label=below:{$1$}] (m8) at (19.5+12,-4) {};
  \node[circle,fill=black!35, scale=.75, label=below:{$2$}] (m9) at (21.5+12,-4) {};
   \node[circle,fill=black!35, scale=.75, label=below:{$2$}] (m10) at (23.5+12,-4) {};
 \draw [very thick, purple](m6) -- (m7); 
   \draw [very thick, purple](m7) -- (m8);
   \draw [very thick, purple](m7) -- (m9);
\draw [very thick, purple](m7) -- (m10);

\end{tikzpicture}
\end{center}
\begin{center}
\begin{tikzpicture}[scale=.35]

\node[circle,fill=black, scale=.75, label=above:{$1$}] (p1) at (4,0) {};
    \node[circle,fill=teal!35, scale=.75, label=right:{$0$}] (p2) at (2,-2) {};
 \node[circle,fill=black!35, scale=.75, label=below:{$1$}] (p3) at (6,-2) {};
  \node[circle,fill=black!35, scale=.75, label=below:{$1$}] (p4) at (0,-4) {};
   \node[circle,fill=black!35, scale=.75, label=below:{$3$}] (p5) at (4,-4) {};
 \draw [very thick, purple](p1) -- (p2); 
   \draw [very thick, purple](p1) -- (p3);
   \draw [very thick, purple](p2) -- (p4);
\draw [very thick, purple](p2) -- (p5);

\node[circle,fill=black, scale=.75, label=above:{$1$}] (q1) at (4+10,0) {};
    \node[circle,fill=teal!35, scale=.75, label=right:{$0$}] (q2) at (2+10,-2) {};
 \node[circle,fill=black!35, scale=.75, label=below:{$1$}] (q3) at (6+10,-2) {};
  \node[circle,fill=black!35, scale=.75, label=below:{$2$}] (q4) at (0+10,-4) {};
   \node[circle,fill=black!35, scale=.75, label=below:{$2$}] (q5) at (4+10,-4) {};
 \draw [very thick, purple](q1) -- (q2); 
   \draw [very thick, purple](q1) -- (q3);
   \draw [very thick, purple](q2) -- (q4);
\draw [very thick, purple](q2) -- (q5);

\node[circle,fill=black, scale=.75, label=above:{$1$}] (n1) at (20+4,0) {};
    \node[circle,fill=teal!35, scale=.75, label=right:{$0$}] (n2) at (18+4,-2) {};
 \node[circle,fill=black!35, scale=.75, label=below:{$3$}] (n3) at (22+4,-2) {};
  \node[circle,fill=black!35, scale=.75, label=below:{$1$}] (n4) at (16+4,-4) {};
   \node[circle,fill=black!35, scale=.75, label=below:{$1$}] (n5) at (20+4,-4) {};
 \draw [very thick, purple](n1) -- (n2); 
   \draw [very thick, purple](n1) -- (n3);
   \draw [very thick, purple](n2) -- (n4);
\draw [very thick, purple](n2) -- (n5);

\node[circle,fill=black, scale=.75, label=above:{$1$}] (m1) at (20+8+6,0) {};
    \node[circle,fill=teal!35, scale=.75, label=right:{$0$}] (m2) at (18+8+6,-2) {};
 \node[circle,fill=black!35, scale=.75, label=below:{$2$}] (m3) at (22+8+6,-2) {};
  \node[circle,fill=black!35, scale=.75, label=below:{$1$}] (m4) at (16+8+6,-4) {};
   \node[circle,fill=black!35, scale=.75, label=below:{$2$}] (m5) at (20+8+6,-4) {};
 \draw [very thick, purple](m1) -- (m2); 
   \draw [very thick, purple](m1) --(m3);
   \draw [very thick, purple](m2) -- (m4);
\draw [very thick, purple](m2) -- (m5);

\end{tikzpicture}
\end{center}
In order, these correspond to the intersections $$\mathsf A\cap \mathsf B\, ,\,\,\mathsf A\cap \mathsf C\, ,\,\,\mathsf A\cap \mathsf D\, ,\,\,\mathsf A\cap \mathsf E\, , \,\,\mathsf B\cap \mathsf D\, , \,\,\mathsf B\cap \mathsf E\, ,\,\,\mathsf C\cap \mathsf D\, , \,\,\mathsf C\cap \mathsf E\,.$$ Finally, the remaining $9$ extremal trees have $5$ edges:
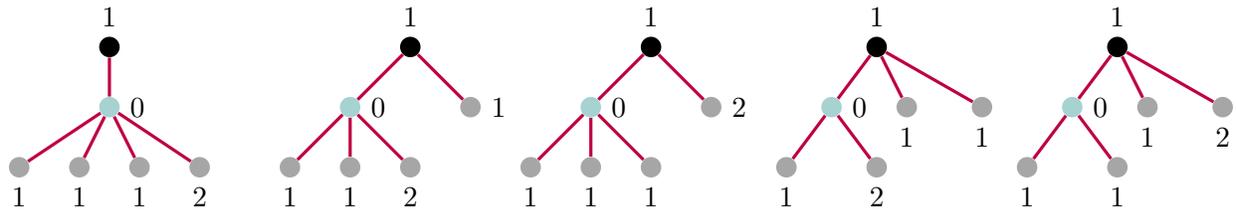
\begin{center}
\begin{tikzpicture}[scale=.4]
\node[circle,fill=black, scale=.75, label=above:{$1$}] (p1) at (3,0) {};
    \node[circle,fill=teal!35, scale=.75, label=right:{$0$}] (p2) at (3,-2) {};
 \node[circle,fill=black!35, scale=.75, label=below:{$1$}] (p3) at (0,-4) {};
  \node[circle,fill=black!35, scale=.75, label=below:{$1$}] (p4) at (2,-4) {};
   \node[circle,fill=black!35, scale=.75, label=below:{$1$}] (p5) at (4,-4) {};
    \node[circle,fill=black!35, scale=.75, label=below:{$2$}] (p6) at (6,-4) {};
 \draw [very thick, purple](p1) -- (p2); 
   \draw [very thick, purple](p2) -- (p3);
   \draw [very thick, purple](p2) -- (p4);
\draw [very thick, purple](p2) -- (p5);
\draw [very thick, purple](p2) -- (p6);

\node[circle,fill=black, scale=.75, label=above:{$1$}] (n1) at (13,0) {};
    \node[circle,fill=teal!35, scale=.75, label=right:{$0$}] (n2) at (11,-2) {};
 \node[circle,fill=black!35, scale=.75, label=right:{$1$}] (n3) at (15,-2) {};
  \node[circle,fill=black!35, scale=.75, label=below:{$1$}] (n4) at (9,-4) {};
   \node[circle,fill=black!35, scale=.75, label=below:{$1$}] (n5) at (11,-4) {};
    \node[circle,fill=black!35, scale=.75, label=below:{$2$}] (n6) at (13,-4) {};
 \draw [very thick, purple](n1) -- (n2); 
   \draw [very thick, purple](n1) -- (n3);
   \draw [very thick, purple](n2) -- (n4);
\draw [very thick, purple](n2) -- (n5);
\draw [very thick, purple](n2) -- (n6);

\node[circle,fill=black, scale=.75, label=above:{$1$}] (m1) at (13+8,0) {};
    \node[circle,fill=teal!35, scale=.75, label=right:{$0$}] (m2) at (11+8,-2) {};
 \node[circle,fill=black!35, scale=.75, label=right:{$2$}] (m3) at (15+8,-2) {};
  \node[circle,fill=black!35, scale=.75, label=below:{$1$}] (m4) at (9+8,-4) {};
   \node[circle,fill=black!35, scale=.75, label=below:{$1$}] (m5) at (11+8,-4) {};
    \node[circle,fill=black!35, scale=.75, label=below:{$1$}] (m6) at (13+8,-4) {};
 \draw [very thick, purple](m1) -- (m2); 
   \draw [very thick, purple](m1) -- (m3);
   \draw [very thick, purple](m2) -- (m4);
\draw [very thick, purple](m2) -- (m5);
\draw [very thick, purple](m2) -- (m6);

\node[circle,fill=black, scale=.75, label=above:{$1$}] (q1) at (28.5,0) {};
    \node[circle,fill=teal!35, scale=.75, label=right:{$0$}] (q2) at (27,-2) {};
 \node[circle,fill=black!35, scale=.75, label=below:{$1$}] (q3) at (29.5,-2) {};
  \node[circle,fill=black!35, scale=.75, label=below:{$1$}] (q4) at (32,-2) {};
   \node[circle,fill=black!35, scale=.75, label=below:{$1$}] (q5) at (25.5,-4) {};
    \node[circle,fill=black!35, scale=.75, label=below:{$2$}] (q6) at (28.5,-4) {};
 \draw [very thick, purple](q1) -- (q2); 
   \draw [very thick, purple](q1) -- (q3);
   \draw [very thick, purple](q1) -- (q4);
\draw [very thick, purple](q2) -- (q5);
\draw [very thick, purple](q2) -- (q6);

\node[circle,fill=black, scale=.75, label=above:{$1$}] (r1) at (28.5+8,0) {};
    \node[circle,fill=teal!35, scale=.75, label=right:{$0$}] (r2) at (27+8,-2) {};
 \node[circle,fill=black!35, scale=.75, label=below:{$1$}] (r3) at (29.5+8,-2) {};
  \node[circle,fill=black!35, scale=.75, label=below:{$2$}] (r4) at (32+8,-2) {};
   \node[circle,fill=black!35, scale=.75, label=below:{$1$}] (r5) at (25.5+8,-4) {};
    \node[circle,fill=black!35, scale=.75, label=below:{$1$}] (r6) at (28.5+8,-4) {};
 \draw [very thick, purple](r1) -- (r2); 
   \draw [very thick, purple](r1) -- (r3);
   \draw [very thick, purple](r1) -- (r4);
\draw [very thick, purple](r2) -- (r5);
\draw [very thick, purple](r2) -- (r6);

\end{tikzpicture}
\end{center}

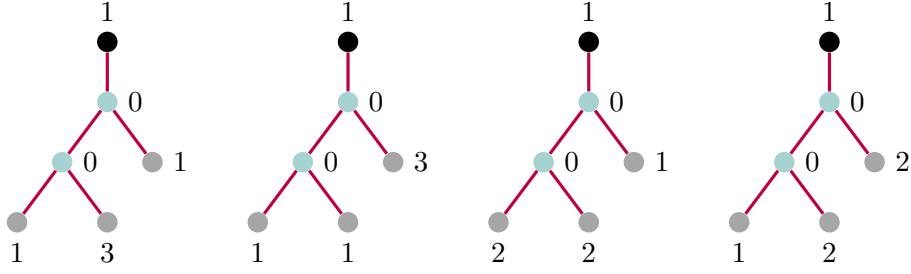
\begin{center}
\begin{tikzpicture}[scale=.4]

\node[circle,fill=black, scale=.75, label=above:{$1$}] (n1) at (11-8,0) {};
    \node[circle,fill=teal!35, scale=.75, label=right:{$0$}] (n2) at (11-8,-2) {};
 \node[circle,fill=teal!35, scale=.75, label=right:{$0$}] (n3) at (9.5-8,-4) {};
  \node[circle,fill=black!35, scale=.75, label=right:{$1$}] (n4) at (12.5-8,-4) {};
   \node[circle,fill=black!35, scale=.75, label=below:{$1$}] (n5) at (8-8,-6) {};
    \node[circle,fill=black!35, scale=.75, label=below:{$3$}] (n6) at (11-8,-6) {};
 \draw [very thick, purple](n1) -- (n2); 
   \draw [very thick, purple](n2) -- (n3);
   \draw [very thick, purple](n2) -- (n4);
\draw [very thick, purple](n3) -- (n5);
\draw [very thick, purple](n3) -- (n6);

\node[circle,fill=black, scale=.75, label=above:{$1$}] (m1) at (11,0) {};
    \node[circle,fill=teal!35, scale=.75, label=right:{$0$}] (m2) at (11,-2) {};
 \node[circle,fill=teal!35, scale=.75, label=right:{$0$}] (m3) at (9.5,-4) {};
  \node[circle,fill=black!35, scale=.75, label=right:{$3$}] (m4) at (12.5,-4) {};
   \node[circle,fill=black!35, scale=.75, label=below:{$1$}] (m5) at (8,-6) {};
    \node[circle,fill=black!35, scale=.75, label=below:{$1$}] (m6) at (11,-6) {};
 \draw [very thick, purple](m1) -- (m2); 
   \draw [very thick, purple](m2) -- (m3);
   \draw [very thick, purple](m2) -- (m4);
\draw [very thick, purple](m3) -- (m5);
\draw [very thick, purple](m3) -- (m6);

\node[circle,fill=black, scale=.75, label=above:{$1$}] (q1) at (11+8,0) {};
    \node[circle,fill=teal!35, scale=.75, label=right:{$0$}] (q2) at (11+8,-2) {};
 \node[circle,fill=teal!35, scale=.75, label=right:{$0$}] (q3) at (9.5+8,-4) {};
  \node[circle,fill=black!35, scale=.75, label=right:{$1$}] (q4) at (12.5+8,-4) {};
   \node[circle,fill=black!35, scale=.75, label=below:{$2$}] (q5) at (8+8,-6) {};
    \node[circle,fill=black!35, scale=.75, label=below:{$2$}] (q6) at (11+8,-6) {};
 \draw [very thick, purple](q1) -- (q2); 
   \draw [very thick, purple](q2) -- (q3);
   \draw [very thick, purple](q2) -- (q4);
\draw [very thick, purple](q3) -- (q5);
\draw [very thick, purple](q3) -- (q6);

\node[circle,fill=black, scale=.75, label=above:{$1$}] (r1) at (11+16,0) {};
    \node[circle,fill=teal!35, scale=.75, label=right:{$0$}] (r2) at (11+16,-2) {};
 \node[circle,fill=teal!35, scale=.75, label=right:{$0$}] (r3) at (9.5+16,-4) {};
  \node[circle,fill=black!35, scale=.75, label=right:{$2$}] (r4) at (12.5+16,-4) {};
   \node[circle,fill=black!35, scale=.75, label=below:{$1$}] (r5) at (8+16,-6) {};
    \node[circle,fill=black!35, scale=.75, label=below:{$2$}] (r6) at (11+16,-6) {};
 \draw [very thick, purple](r1) -- (r2); 
   \draw [very thick, purple](r2) -- (r3);
   \draw [very thick, purple](r2) -- (r4);
\draw [very thick, purple](r3) -- (r5);
\draw [very thick, purple](r3) -- (r6);

\end{tikzpicture}
\end{center}
The extremal trees on the first row correspond to the intersections $$\mathsf A\cap \mathsf F\, ,\,\, \mathsf B\cap \mathsf F\, ,\,\,\mathsf C\cap \mathsf F\, ,\,\,\mathsf D\cap \mathsf F\, , \,\,\mathsf E\cap \mathsf F\,.$$ The extremal trees on the second row correspond to the triple intersections 
$$\mathsf A\cap \mathsf B\cap \mathsf D\, , \,\,\mathsf A\cap \mathsf C\cap \mathsf D\, ,\,\, \mathsf A\cap \mathsf B\cap \mathsf E\, ,\,\, \mathsf A\cap \mathsf C\cap \mathsf E\,.$$ 

We compute the contributions of the $24$ extremal trees above.

\begin{itemize}
\item [(i)] For the trees $\mathsf A-\mathsf G$, the contributions can be found from \eqref{excess}, weighted by the number of automorphisms $1, 1, 1, 2, 2, 6, 120$ respectively. We obtain
\begin{align*}
\frac{1}{|\Aut(\mathsf{A})|}\mathsf{Cont}_{\mathsf A}&=[1, \lambda_4-\psi_1\lambda_3+\psi_1^2\lambda_2-\psi_1^3\lambda_1+\psi_1^4]\,, \\
\frac{1}{|\Aut(\mathsf{B})|}\mathsf{Cont}_{\mathsf B}&=[1, 1, -\lambda_3+\psi_1\lambda_2-\psi_1^2\lambda_1+\psi_1^3]+[\psi_1+\psi_2,1, \lambda_2]+[-\psi_1-2\psi_2,1, \lambda_1\psi_1]\\&+[\psi_1+3\psi_2,1, \psi_1^2]\,,\\
\frac{1}{|\Aut(\mathsf{C})|}\mathsf{Cont}_{\mathsf C}&=[1, 1, \psi_1\lambda_2-\psi_1^2\lambda_1+\psi_1^3]+[1, \lambda_1, -\lambda_2+\psi_1\lambda_1-\psi_1^2]+[1,\psi_1^2-\psi_1\lambda_1, \psi_1-\lambda_1]\\&+[\psi_1+\psi_2, 1, \lambda_2]+[\psi_1+\psi_2,\lambda_1,\lambda_1]+[3\psi_1+\psi_2,\psi_1^2, 1]+[\psi_1+3\psi_2, 1, \psi_1^2]\\&+[-\psi_1-2\psi_2, 1, \psi_1\lambda_1]+[-2\psi_1-\psi_2,\psi_1\lambda_1, 1]+[-\psi_1-2\psi_2,\lambda_1, \psi_1]\\&+[-2\psi_1-\psi_2,\psi_1, \lambda_1]+[1, \psi_1, \lambda_2-\lambda_1\psi_1+\psi_1^2]+[2\psi_1+2\psi_2,\psi_1,\psi_1]\,,
\\
\frac{1}{|\Aut(\mathsf{D})|}\mathsf{Cont}_{\mathsf D}&=\frac{1}{2}([1, 1, 1, \lambda_2-\psi_1\lambda_1+\psi_1^2]+[\psi_1^2+\psi_2^2+\psi_3^2+\psi_1\psi_2+\psi_1\psi_3+\psi_2\psi_3, 1, 1, 1]\\
&+[-\psi_1-\psi_2-\psi_3, 1, 1, \lambda_1]+[\psi_1+\psi_2+2\psi_3, 1, 1, \psi_1])\,,\\
\frac{1}{|\Aut(\mathsf{E})|}\mathsf{Cont}_{\mathsf E}&=\frac{1}{2}([\psi_1^2+\psi_2^2+\psi_3^2+\psi_1\psi_2+\psi_1\psi_3+\psi_2\psi_3,1,1,1]+[-\psi_1-\psi_2-\psi_3,1,1, \lambda_1]\\&+[-\psi_1-\psi_2-\psi_3,1,\lambda_1,1]+[1,1,1, \psi_1^2-\psi_1\lambda_1]+[1,1,\psi_1^2-\psi_1\lambda_1,1]\\&+[\psi_1+\psi_2+2\psi_3,1,1, \psi_1]+[\psi_1+2\psi_2+\psi_3,1,\psi_1,1]+[1,1,\lambda_1-\psi_1,\lambda_1-\psi_1])\,,\\
\frac{1}{|\Aut(\mathsf{F})|}\mathsf{Cont}_{\mathsf F}&=\frac{1}{6}([\psi_1+\psi_2+\psi_3+\psi_4,1,1,1,1]+[1,1,1,1, \psi_1-\lambda_1])\,,\\
\frac{1}{|\Aut(\mathsf{G})|}\mathsf{Cont}_{\mathsf G}&=\frac{1}{120}[1, 1, 1, 1, 1, 1]\,.
\end{align*}
As before, the first position in the bracket records the contribution of the root, while the next entries correspond to the remaining vertices, listed in increasing order by genus (from left to right in the picture). We slightly simplified the answer by ignoring terms of degree $>2g-3+n$ for each vertex of genus $g$ with $n$ markings, due to  \eqref{vanish}. 
\item [(ii)] We next consider the intersection of strata. The first extremal tree on the list corresponds to $\mathsf A\cap \mathsf B$. The locus $\mathsf A$ has codimension $1$, $\mathsf B$ has codimension $2$, the intersection has codimension $3$, while the expected codimension is $5$. By Example \ref{answer3}(iii), the excess contribution is given by
$$-3c_2(\mathcal N)+c_1(\mathcal N)\cdot (6z_1+4z_2+4z_3)-10z_1^2-10 z_1\cdot (z_2+z_3)-5(z_2+z_3)^2+5z_2z_3\,.$$ Here, $\mathcal N$ is the restriction of the normal bundle of $\A_1\times \A_5\to \A_6$ to $\mathsf A\cap \mathsf B.$ We express the answer in terms of the standard tautological classes over $\M_{1, 1}^{\ct}\times \M_{0, 3}^{\ct}\times \M_{1,1}^{\ct}\times \M^{\ct}_{4,1},$ with the bracket entries reflecting the ordering of the factors. Using \eqref{normspl}, we obtain $$c_1(\mathcal N)=[1, 1, 1, -\lambda_1]\,, \quad c_2(\mathcal N)=[1, 1, 1, \lambda_2]\,.$$ Next, as explained in Section \ref{inductivesection}, we substitute the edge variables in terms of the cotangent classes at the nodes:
 $$z_1\mapsto 0\,,\quad z_2\mapsto 0\,, \quad z_3\mapsto [0, 0, 0, -\psi_1]\,.$$ We have used here the vanishing of the $\psi$ classes on $\M_{0, 3}^{\ct}$ and $\M_{1, 1}^{\ct}.$
We obtain 
\begin{align*}
\frac{1}{|\Aut(\mathsf{AB})|}\mathsf{Cont}_{\mathsf{AB}}=[1,1,1, -3\lambda_2+4\lambda_1\psi_1-5\psi_1^2]\,.
\end{align*} 

The contribution of the intersection $\mathsf A\cap \mathsf C$ corresponding to the locus $$\M_{1, 1}^{\ct}\times \M_{0, 3}^{\ct}\times \M_{2,1}^{\ct}\times \M^{\ct}_{3,1}$$ is found using the same method. We obtain 
\begin{align*}
\frac{1}{|\mathrm{Aut} (\mathsf{AC})|}\mathsf{Cont}_{\mathsf{AC}}&=[1,1,1,-3\lambda_2+4\lambda_1\psi_1-5\psi_1^2]+[1,1,4\lambda_1\psi_1-5\psi_1^2,1]\\&+[1,1,-3\lambda_1+4\psi_1,\lambda_1]+[1,1,4\lambda_1-5\psi_1,\psi_1] \,.\end{align*}
\item [(iii)] We consider the codimension $1$ locus $\mathsf A$ and the codimension $3$ locus $\mathsf D$ intersecting the codimension $4$ locus $\mathsf A\cap \mathsf D.$ The excess contribution is found by Example \ref{answer4}(ii): 
$$-4c_1(\mathcal N)+10 z_1+ 5(z_2+z_3+z_4)\,,$$
where $\mathcal N$ is the restriction of the normal bundle of $\A_1\times \A_5\to \A_6$ to $\mathsf A\cap \mathsf D.$ Expressing in terms of tautological classes over the product $$\M_{1,1}^{\ct}\times \M_{0, 4}^{\ct}\times \M_{1, 1}^{\ct}\times \M_{1, 1}^{\ct}\times \M_{3, 1}^{\ct}\,,$$
and accounting for automorphisms, we obtain 
\begin{align*}
\frac{1}{|\Aut(\mathsf{AD})|}\mathsf{Cont}_{\mathsf{AD}}&=\frac{1}{2}([1,1,1,1,4\lambda_1-5\psi_1]+[1,-10\psi_1-5\psi_2-5\psi_3-5\psi_4,1,1,1])\,.
\end{align*} Over the genus $0$ vertex, the markings and the $\psi$ classes are numbered starting from the edge connecting to the root. The convention is necessary to make precise the second term above. 

The intersection $\mathsf A\cap \mathsf E$ corresponds to the product 
$$\M_{1,1}^{\ct}\times \M_{0, 4}^{\ct}\times \M_{1, 1}^{\ct}\times \M_{2, 1}^{\ct}\times \M_{2, 1}^{\ct}\,,$$
and the associated contribution is computed by the same formula. We find 
\begin{align*}
\frac{1}{|\Aut(\mathsf{AE})|}\mathsf{Cont}_{\mathsf{AE}}&=\frac{1}{2}([1,1,1,4\lambda_1-5\psi_1,1]+[1,1,1,1,4\lambda_1-5\psi_1]\\&+[1,-10\psi_1-5\psi_2-5\psi_3-5\psi_4,1,1,1])\,.
\end{align*}

\item [(iv)] Next, we consider the codimension $2$ locus $\mathsf B$ and the codimension $3$ locus $\mathsf D$ intersecting in the codimension $4$ locus $\mathsf B\cap \mathsf D$. The contribution is found from Example \ref{weird}: 
$$-3c_1(\mathcal N)+6z_1+3z_2+4(z_3+z_4)\,,$$
where as usual  $\mathcal N$ is the restriction of the normal bundle of $\A_1\times \A_5\to \A_6$ to $\mathsf B\cap \mathsf D.$ 
Simple geometry yields $$c_1(\mathcal N)=[1, 1, 1, -\lambda_1, 1]\, , \quad z_1\mapsto [-\psi_1, 1, 1, 1,  1]\,,\quad z_2\mapsto [-\psi_2, 1, 1, 1, 1]\,,\quad z_3\mapsto 0\,,$$ $$\quad z_4\mapsto [1, 1, 1, -\psi_1, 1]\,.$$ 
We obtain 
\begin{align*}\frac{1}{|\Aut(\mathsf{BD})|}\mathsf{Cont}_{\mathsf{BD}}&=[1,1,1, 3\lambda_1-4\psi_1, 1]+[-6\psi_1-3\psi_2,1,1,1,1]\,.\end{align*} 
Here, the ordering in the bracket corresponds to the natural ordering in the product $$\M_{1, 2}^{\ct}\times \M_{0, 3}^{\ct}\times \M_{1, 1}^{\ct}\times \M_{3, 1}^{\ct}\times \M_{1,1}^{\ct}\, .$$ 
We read the tree from the root down, and from left to right.

The intersections $\mathsf B\cap \mathsf E$, $\mathsf C\cap \mathsf D$ and $\mathsf C\cap \mathsf E$ are computed in the same manner. These loci correspond to the products $$\M_{1, 2}^{\ct} \times \M_{0, 3}^{\ct}\times \M_{2, 1}^{\ct}\times \M_{2, 1}^{\ct}\times \M_{1,1}^{\ct}\,,\quad \M_{1, 2}^{\ct}\times \M_{0, 3}^{\ct}\times \M_{1, 1}^{\ct}\times \M_{1, 1}^{\ct}\times \M_{3,1}^{\ct}\,,$$ $$\M_{1, 2}^{\ct}\times \M_{0, 3}^{\ct}\times \M_{1, 1}^{\ct}\times \M_{2, 1}^{\ct}\times \M_{2,1}^{\ct}\,,$$ respectively. 
After accounting for automorphisms, we obtain 
\begin{align*}
\frac{1}{|\Aut(\mathsf{BE})|}\mathsf{Cont}_{\mathsf{BE}}&=\frac{1}{2}([1,1,1,3\lambda_1-4\psi_1,1]+[1,1, 3\lambda_1-4\psi_1, 1, 1]+[-6\psi_1-3\psi_2,1,1,1,1])\,,\\
\frac{1}{|\Aut(\mathsf{CD})|}\mathsf{Cont}_{\mathsf{CD}}&=\frac{1}{2}([1,1, 1, 1, 3\lambda_1-3\psi_1]+[-6\psi_1-3\psi_2,1,1,1,1])\,,\\
\frac{1}{|\Aut(\mathsf{CE})|}\mathsf{Cont}_{\mathsf{CE}}&=[1,1,1,1, 3\lambda_1-3\psi_1]+[-6\psi_1-3\psi_2,1,1,1,1]+[1,1,1,3\lambda_1-4\psi_1, 1]\,.
\end{align*}

    \item [(v)] For the $9$ extremal trees with $5$ edges, there are $6, 2, 6, 2, 2$ automorphisms, respectively, for the $5$ trees on the first row, and the excess contributions are $-5, -4, -4, -3, -3$, respectively, see Examples \ref{answer5}, \ref{answer4}(i) and \ref{answer3}(i). The contributions of these loci equal the fundamental class multiplied by $$-\frac{5}{6}\, ,\ -\frac{4}{2}\, ,\ -\frac{4}{6}\, ,\ -\frac{3}{2}\, ,\ -\frac{3}{2}\,.$$
    For the remaining $4$ trees on the second row, the number of automorphisms is $1, 2, 2, 1$ and the excess contribution is $15$ for each of these extremal trees, see Example \ref{answer15}. The contributions of these loci equal the fundamental class times $$15,\,\, \frac{15}{2},\,\, \frac{15}{2},\,\, 15\,.$$ 
\end{itemize}

To complete the proof, we collect the terms (i)-(v), push forward to $\M_6^{\ct}$, and subtract $\frac{2370}{691}\lambda_5.$ 
We verify using \texttt{admcycles} \cite{admcycles} that the resulting class pairs trivially with all elements in $\mathsf{R}^4(\M_6^{\ct}),$ as expected from Theorem \ref{vang}. Furthermore, we see
$$
\Tor^*\Delta_6\neq 0 \in
\mathsf{R}^5(\M_6^{\ct})\, $$
using completeness of Pixton's relations
in $\mathsf{R}^*(\M_6^{\ct})$ proven in \cite{CLS}.
The implementation can be found in \cite{COP}.  \qed

\subsection{Genus 7: Proof of Proposition \ref{Delta7}}
\label{g777}
 Proposition \ref{Delta7} can be proven by analyzing the extremal trees and their contributions (as
 in the proof of Theorem \ref{Delta6}).
 We instead give a simpler proof based on the structure of the Gorenstein kernel of $\R^*(\M_7^{\ct})$, which was suggested to us by Aaron Pixton. The methods here are developed systematically in
 \cite{CLS} to study the 
 Gorenstein kernel of 
 $\mathsf{R}^*(\M_{g,n}^{\ct})$ for
 general $g$ and $n$.
 
   By \cite{CLS}, the kernel of the $\lambda_7$-pairing
   on $\R^*(\M_7^{\ct})$ is a 1-dimensional
   subspace of $\R^6(\M_7^{\ct})$ in the
   graded piece
   \begin{equation}
   \label{xxvvb}
   \R^5(\M_7^{\ct})\times \R^6(\M_7^{\ct})\rightarrow \R^{11}(\M_7^{\ct})\,.
   \end{equation}

   We define
   a class
    $\alpha\in \R^6(\M_7^{\ct})$
    by
    pulling back $\Tor^*\Delta_6$ along the forgetful map $$\pi:\M_{6,1}^{\ct}\rightarrow \M_6^{\ct}$$ and attaching an elliptic tail via
$$j:\M^{\ct}_{6,1}\times \M_{1,1}^{\ct}\to \M_7^{\ct}\, .$$
In other words,
$\alpha= j_*(\pi^*\Tor^*\Delta_6 \times [\M_{1,1}^{\ct}])
\in \R^6(\M_7^{\ct})$.

\begin{prop} \label{p777}
The class
$\alpha= j_*(\pi^*\Tor^*\Delta_6 \times [\M_{1,1}^{\ct}])
\in \R^6(\M_7^{\ct})$
spans the 1-dimension Gorenstein kernel of $\mathsf{R}^*(\M_7^{\ct})$.
\end{prop}


\begin{proof} We first show $\alpha$  does not vanish. Consider the pull back
$$j^*(\alpha) = (-\psi_1 \cdot \pi^*\Tor^*\Delta_6) \times [\M_{1,1}^{\ct}]\, .$$
The right side is 
nonzero. Indeed, the class $-\psi_1\cdot \pi^*\Tor^*\Delta_6\neq 0$ since its pushforward under $\pi$ is a nonzero multiple of $\Tor^*\Delta_6\neq 0,$ using Theorem \ref{Delta6}. 
Hence $\alpha \neq 0
\in \R^6(\M_7^{\ct})$.

   We prove next that
   $\alpha$ lies in the
   Gorenstein kernel of \eqref{xxvvb}.
   For every $\beta\in \R^5(\M_7^{\ct})$, we must show that $$\alpha\cdot \beta = j_*(\Tor^*\Delta_6\times [\M_{1,1}^{\ct}])\cdot \beta=j_*((\pi^*\Tor^*\Delta_6\times [\M_{1,1}^{\ct}])\, \cdot\,  j^*\beta) = 0 \in \mathsf{R}^{11}(\M_7^{\ct})\,.$$
   It suffices to show $$(\pi^*\Tor^*\Delta_6\times [\M_{1,1}^{\ct}] )\, \cdot\,  j^*\beta=0$$
   on $\M_{6,1}^{\ct}\times \M_{1,1}^{\ct}$. By \cite[Propositon 12]{GP}, the K\"unneth components of $j^*\beta$ are tautological, $$j^*\beta\in \R^*(\M_{6, 1}^{\ct})\otimes \R^*(\M_{1, 1}^{\ct})\,.$$ Since $\R^*(\M_{1, 1}^{\ct})=\mathbb Q$, we need only show that $$\pi^*\Tor^*\Delta_6\cdot \gamma=0\in \R^{10}(\M_{6, 1}^{\ct})$$ for any class $\gamma\in \R^5(\M_{6, 1}^{\ct}).$ We have 
$$\R^{10}(\M_{6, 1}^{\ct})=\R^9(\M_6^{\ct})=\mathbb Q\,.$$  Furthermore, using the description of the socle generator in \cite[Section 4.1.2]{FP2} or \cite[Section 5.6]{GV}, we know $$\pi_*:\R^{10}(\M_{6, 1}^{\ct})\to \R^9(\M_6^{\ct})$$
is an isomorphism. Therefore, it remains to prove $$\pi_*(\pi^*\Tor^*\Delta_6\cdot \gamma)=0 \text{ or } \Tor^*\Delta_6\cdot \pi_*\gamma=0\, ,$$ which is clear since $\Tor^*\Delta_6\in \R^5(\M_6^{\ct})$ is in the Gorenstein kernel and $\pi_*\gamma\in \R^4(\M_6^{\ct}).$
\end{proof}

So
$\Tor^*\Delta_6\in \mathsf{R}^5(\M_6^{\ct})$ explains
not only the Gorenstein kernel of
$\mathsf{R}^*(\M_6^{\ct})$, but also
the Gorenstein kernel of $\mathsf{R}^*(\M_7^{\ct})$!

\begin{proof}[Proof of Proposition \ref{Delta7}] Since 
$\alpha\in \R^6(\M_7^{\ct})$
is a generator of the
Gorenstein kernel of \eqref{xxvvb} and $\Tor^*\Delta_7$ also lies in the Gorenstein kernel by Theorem \ref{vang}, there exists a constant $c\in \mathbb{Q}$ for which  \begin{equation}\label{l}\Tor^*\Delta_7=c\cdot \alpha\,.\end{equation}
The pullback $j^*(\Tor^*\Delta_7)$  vanishes by the proof of Theorem \ref{vang}.
Since we have
seen $j^*\alpha$
does not vanish,
we must have $c=0$.
\end{proof}

\subsection{Outlook in higher genus}
For $g\geq 8$, the full structure of $\mathsf{R}^*(\M_g^{\ct})$ is not yet understood,
but a complete proposal is provided by  Pixton's conjecture \cite{Pix}.

Assuming  Pixton's relations are complete for $\mathsf{R}^*(\M_g^{\ct})$, we have shown that  $\Tor^*\Delta_8\in \mathsf{R}^7(\M_8^{\ct})$ and 
$\Tor^*\Delta_9\in \mathsf{R}^8(\M_9^{\ct})$ are nonzero using Pixton's formula in Section \ref{Pixformula}
(and computing with
\texttt{admcycles} \cite{admcycles}). 
Because of the computational complexity, higher genus calculations using these methods remain out of reach.  On the other hand, Iribar López has shown that $$\Delta_g\neq 0\in \mathsf{CH}^{g-1}(\A_g)$$ for $g=12$ and even $g\geq 16$ \cite{Iribar}. 

Using the methods of \cite[Theorem 33]{HT2}, Taïbi has shown that $\mathsf{IH}^{2g-2}(\A_g^{\mathrm{Sat}})$ is not generated by $\lambda$ classes when $g\geq 8$. We view his calculations as evidence that $\Delta_g$ is nonzero for $g\geq 8$.
 
\section{Virtual fundamental classes on the Noether-Lefschetz loci}\label{virtual}
We study the virtual geometry of the Noether-Leschetz loci. The components of the Noether-Lefschetz locus 
$\mathsf {NL}_g^2$ have been classified by Debarre and Lazslo \cite{DL}, see Theorem \ref{DebLaz}. 
We will follow the notation of
Theorem \ref{DebLaz}.
All irreducible components are nonsingular \cite {DL}. The components of type (i)  have codimension $k(g-k)$, while the components of type (ii) have codimension $g(n+1)/2$. However, the expected codimension of each Noether-Lefschetz component is the larger number
$$\binom{g}{2}=\dim H^{2, 0}(A)\, ,$$ where $A$ is an abelian variety, see for instance \cite [3.a.25]{CGGH}. 

Let $j:\mathcal S\to \A_g$ denote a Noether-Lefschetz component. Consider the universal family $$\pi: \mathcal X_g\to \A_g\,,$$ and the variation of Hodge structure on the second cohomology $$\mathcal F^2\subset \mathcal F^1\subset \mathcal F^0={\bf R}^2\pi_{*} \mathbb C\otimes \mathcal O_{\A_g}.$$ Griffiths transversality yields a map $$\nabla: \mathcal F^1/\mathcal F^2\to \mathcal F^0/\mathcal F^1 \otimes \Omega_{\A_g}.$$ Over the marked Noether-Lefschetz locus $\mathcal S$, the additional generator of the Picard group furnishes a section of $\mathcal F^1/\mathcal F^2$, while $\mathcal F^0/\mathcal F^1=\wedge^2 \mathbb E^{\vee}.$ Thus, over $\mathcal S$, we dually have a natural map $T_{\A_g}\to j^{*} \wedge^2 \mathbb E$ whose kernel is the tangent space to the Noether-Lefschetz locus $\mathcal S$, see \cite [Lemma 5.16]{V}. Writing $\mathcal N$ for the normal bundle of $\mathcal S
$, we find $$0\to \mathcal N\to j^*\wedge^2 \mathbb E^{\vee}.$$ 

The simplest example of a Noether-Lefschetz locus is the product $j:\A_{k}\times \A_{g-k}\to \A_g$ with normal bundle  $\mathcal N=\mathbb E_k^{\vee}\boxtimes \mathbb E_{g-k}^{\vee}.$ Over the product locus, we consider the obstruction bundle $$0\to \mathcal N\to j^*\wedge^2 \mathbb E_g^{\vee}\to \textnormal{Obs}\to 0\,.$$ Since $j^*\mathbb E_g^{\vee}=\mathbb E_k^{\vee}\boxplus \mathbb E_{g-k}^{\vee}$, 
we find $$\textnormal {Obs}=\wedge^2 \mathbb E_k^{\vee}\boxplus \wedge^2\mathbb E_{g-k}^{\vee}\,.$$  The product locus carries a virtual fundamental class of the expected dimension: $$[\A_k\times \A_{g-k}]^{\textnormal{vir}}=\mathsf e(\textnormal{Obs}).$$ Of course, the same construction makes sense over all the Noether-Lefschetz loci of Theorem \ref{DebLaz}.

In the case of products, an explicit description of the virtual fundamental class is possible. The following result proves Proposition \ref{vclass} and implies that $j_{*} [\A_k\times \A_{g-k}]^{\textnormal{vir}}\in \R^*_{\textnormal{pr}}(\A_g).$

\begin{lem} We have
$\mathsf e(\wedge^2 \mathbb E)= \lambda_{1}\cdots \lambda_{g-1}\in \R^*(\A_g).$
\end{lem}
\proof For any vector bundle $V$ of rank $r$, we have 
\begin{equation} \label{schu}
\mathsf e(\wedge^2 V) = s_{\delta} (x_1, ..., x_r)\, ,
\end{equation}
where $x_1, \ldots, x_r$ are the Chern roots of $V$ and $s_{\delta}$ is the Schur polynomial corresponding to $$\delta=(r-1, r-2, \ldots, 0)\, .$$ 
The equality \eqref{schu}
follows from the formula for Schur polynomials using alternants \cite{FultonHarris}: $$s_{\lambda} = \frac{a_{\lambda+\delta}}{a_{\delta}}\, ,$$ where $a_{\mu}=\det (x_j^{\mu_i + r - i})$.   
If $\lambda= \delta$, the numerator and denominator are both Vandermonde determinants with values $\prod_{i<j} (x_i^2-x_j^2)$ and $\prod_{i<j} (x_i-x_j)$ respectively. As a consequence, we have $$\mathsf e(\wedge^2 V)=\prod_{i<j}(x_i+x_j)=s_{\delta}(x_1, \ldots, x_r)\, . $$

We apply \eqref{schu} to the Hodge bundle. We use the second Jacobi-Trudi formula to compute the Schur polynomial in terms of the elementary symmetric functions, which correspond to the Chern classes of $\mathbb E$. 
The Jacobi-Trudi determinant has the following shape 
$$\begin{vmatrix} 
  \lambda_{g-1}    &   \lambda_g        &     0    &0    &\ldots    &   0   &      0 \\  \lambda_{g-3}   &   \lambda_{g-2}   &   \lambda_{g-1} &     \lambda_g  &  \ldots &    0 & 0\\
                 \vdots & \vdots & \vdots & \vdots & \ldots &\vdots & \vdots \\
                 0 & 0 & 0 & 0 & \ldots & 0 & 1
\end{vmatrix} $$
In other words, the determinant has $\lambda_{g-1}, \lambda_{g-2},\ldots, \lambda_1, 1$ on the diagonal and the indices increase in the rows. 

Write $\mathsf D_g$ for the Jacobi-Trudi determinant, which is a polynomial in $\lambda_1, \ldots, \lambda_g$. We let $\R_g$ be the polynomial ring generated by classes $\lambda_1, \ldots, \lambda_g$ subject to Mumford's relations. We have $$\R^*(\A_g)=\R_g/(\lambda_g)=\R_{g-1}\, .$$ We seek to show that $$\mathsf D_g= \lambda_{1} \cdots \lambda_{g-1}$$ in $\R^*(\A_g)=\R_{g-1}$. We proceed by induction on $g$, the base case being clear. By induction, we have $$\mathsf D_{g-1} = \lambda_{1} \cdots \lambda_{g-2}$$ in $\R_{g-2}=\R_{g-1}/(\lambda_{g-1})$. Therefore, we must have $$\mathsf D_{g-1}=\lambda_{1} \cdots \lambda_{g-2} + \lambda_{g-1} \cdot \mathsf P$$ in $\R_{g-1}$, for some polynomial $\mathsf P$ in the $\lambda$-classes. We expand $\mathsf D_g$ on the first row. Since $\lambda_g=0$ in $\R_{g-1}$, we obtain that in $\R_{g-1}$ we have $$\mathsf D_g=\lambda_{g-1} \mathsf D_{g-1} =\lambda_{g-1} (\lambda_{1} \cdots \lambda_{g-2} + \lambda_{g-1} \cdot \mathsf P)=\lambda_{1} \cdots \lambda_{g-1}.$$ Here, we used that $\lambda_{g-1}^2=0$ in $\R_{g-1}$ by Mumford's relation. We have completed the inductive step. \qed

\bibliographystyle{amsplain}
\bibliography{refs}
\end{document}